\documentclass[11pt,reqno]{amsart}
\usepackage[colorlinks=true,
pdfstartview=FitV, linkcolor=cyan, citecolor=magenta,
urlcolor=blue]{hyperref}
\usepackage{amsmath,amsfonts,latexsym,amssymb}
\usepackage{mathrsfs}
\usepackage[latin1]{inputenc}
\usepackage[T1]{fontenc}
\usepackage{ae,aecompl}
\usepackage{braket}
\usepackage{comment}
\usepackage{color}
\usepackage{graphicx}
\usepackage{subfig}
\usepackage{bbm}

\usepackage{enumitem}

\newtheorem{theorem}{Theorem}[section]
\newtheorem{lemma}[theorem]{Lemma}
\newtheorem{proposition}[theorem]{Proposition}
\newtheorem{corollary}[theorem]{Corollary}

\newtheorem*{thm*}{\protect\theoremname}
\theoremstyle{definition}
\newtheorem{definition}[theorem]{Definition}
\newtheorem{remark}[theorem]{Remark}
\newtheorem*{cor*}{\protect\corollaryname}

\newtheorem{observation}[theorem]{Observation}

\long\def\@savemarbox#1#2{\global\setbox#1\vtop{\hsize\marginparwidth 
  \@parboxrestore\tiny\raggedright #2}}
\renewcommand{\d}{{\rm d}}

\newcommand{\Cc}{{\mathcal C}}

\newcommand{\GG}{\mathsf G}

\newcommand{\PSL}{\mathsf{PSL}}
\newcommand{\PGL}{\mathsf{PGL}}

\newcommand{\mfg}{{\mathfrak g}}
\newcommand{\mfa}{{\mathfrak a}}
\newcommand{\mfk}{{\mathfrak k}}
\newcommand{\mfp}{{\mathfrak p}}
\newcommand{\mfu}{{\mathfrak u}}

\newcommand{\abs}[1]{\left|#1\right|}

\newcommand{\Bc}{\mathcal B}
\newcommand{\Sc}{\mathcal S}
\newcommand{\Oc}{\mathcal O}

\newcommand{\Rc}{\mathcal R}

\newcommand{\Zc}{\mathcal Z}

\DeclareMathOperator{\Aut}{\mathrm{Aut}}
\DeclareMathOperator{\id}{\mathrm{id}}

\DeclareMathOperator{\Cb}{\mathbb{C}}

\DeclareMathOperator{\Hb}{\mathbb{H}}

\DeclareMathOperator{\Kb}{\mathbb{K}}

\DeclareMathOperator{\Nb}{\mathbb{N}}
\DeclareMathOperator{\Pb}{\mathbb{P}}
\DeclareMathOperator{\Rb}{\mathbb{R}}

\DeclareMathOperator{\Usf}{\mathsf{U}}
\DeclareMathOperator{\Psf}{\mathsf{P}}
\DeclareMathOperator{\Ksf}{\mathsf{K}}

\DeclareMathOperator{\Fc}{\mathcal{F}}

\DeclareMathOperator{\dist}{dist}

\DeclareMathOperator{\SL}{\mathsf{SL}}
\newcommand{\norm}[1]{\left\|#1\right\|}
\newcommand{\ip}[1]{\left\langle #1\right\rangle}

\DeclareMathOperator{\Ac}{\mathcal{A}}

\newcommand{\wh}[1]{\widehat{#1}}

\usepackage{ mathrsfs }

\providecommand{\corollaryname}{Corollary}
\providecommand{\theoremname}{Theorem}

\begin{document}

\title{Patterson-Sullivan measures for transverse subgroups}
\author[Canary]{Richard Canary}
\address{University of Michigan}
\author[Zhang]{Tengren Zhang}
\address{National University of Singapore}
\author[Zimmer]{Andrew Zimmer}
\address{University of Wisconsin-Madison}
\thanks{Canary was partially supported by grant  DMS-1906441 from the National Science Foundation.
Zhang was partially supported by the NUS-MOE grant R-146-000-270-133 and A-8000458-00-00. Zimmer was partially supported by a Sloan research fellowship and grants DMS-2105580 and DMS-2104381 from the
National Science Foundation.}

\begin{abstract} 

We study Patterson-Sullivan measures for a class of discrete subgroups of higher rank semisimple Lie groups, called transverse groups, whose limit set is well-defined and transverse
in a partial flag variety. This class of groups includes both Anosov and relatively Anosov groups, as well as all discrete subgroups of rank one Lie groups. We prove an analogue of the Hopf-Tsuji-Sullivan dichotomy and then use this dichotomy to prove a variant of Burger's Manhattan curve theorem. We also use the Patterson-Sullivan measures to obtain conditions for when a subgroup has critical exponent strictly less than the original transverse group. These gap results are new even for Anosov groups. 

\end{abstract}

\maketitle

\tableofcontents

\section{Introduction}
If $\Gamma$ is a discrete subgroup of the group  $\mathsf{PO}(d,1)$ of isometries of hyperbolic $d$-space $\mathbb H^d$, Patterson \cite{patterson} and Sullivan  \cite{sullivan-ergodic} constructed a probability measure $\mu$ supported on the limit set  $\Lambda(\Gamma)$ of $\Gamma$ which transforms like the $\delta$-dimensional Hausdorff measure, 
where $\delta$ is the critical exponent of the Poincar\'e series of $\Gamma$. Alternatively, one may view $\delta$ as the exponential growth rates of the number
of orbit points of $\Gamma$ in a ball of radius $T$.  The Hopf-Tsuji-Sullivan dichotomy asserts, in part, that the action of $\Gamma$ on the set $\Lambda(\Gamma)^{(2)}$ of
pairs of distinct points in the limit set is ergodic with respect to the measure $\mu\otimes\mu$ if and only if the Poincar\'e series  of $\Gamma$ diverges at its
critical exponent. Equivalently, it says that  the non-wandering part of the geodesic flow on $\Gamma \backslash T^1\mathbb H^d$ is 
ergodic with respect to its Bowen-Margulis-Sullivan measure if and only if the Poincar\'e series  of $\Gamma$ diverges at its
critical exponent.

In this paper, we study Patterson-Sullivan measures for a class of discrete subgroups of higher rank semisimple Lie groups, called transverse groups. This class of groups includes both Anosov and relatively Anosov groups as well as all discrete subgroups of rank one Lie groups. 
Transverse groups were previously studied by Kapovich, Leeb and Porti \cite{KLP1}, who called them regular, antipodal groups.
Patterson-Sullivan measures for discrete subgroups of higher rank Lie groups were first studied  by Albuquerque \cite{albuquerque} and Quint \cite{quint-ps}. Recently Patterson-Sullivan measures 
for Anosov groups have been extensively studied by Dey-Kapovich \cite{dey-kapovich}, Sambarino \cite{sambarino-dichotomy}, Burger-Landesberg-Lee-Oh \cite{BLLO}, Lee-Oh \cite{lee-oh-invariant,lee-oh-dichotomy}
and others. 

We prove a generalization of the Hopf-Tsuji-Sullivan dichotomy to our setting. 
Using this dichotomy we prove a variant of Burger's Manhattan curve theorem~\cite{burger}. 
We also use  Patterson-Sullivan measures to obtain conditions for when a subgroup has critical exponent strictly less than 
the original transverse group. These gap results are new even for Anosov groups. 

In this introduction, we will restrict our discussion to the setting of transverse subgroups of $\mathsf{PSL}(d,\mathbb K)$, 
where $\mathbb K$ is either the real numbers $\mathbb R$ or the complex numbers $\mathbb C$. In the body of the paper, we will consider transverse subgroups of connected semisimple real Lie groups of non-compact type with finite center.

\bigskip

In this setting Patterson-Sullivan measures are probability measures on partial flag manifolds defined using a natural cocycle (studied by Quint \cite{quint-ps}) for the action of $\mathsf{PSL}(d,\mathbb K)$ on the partial flag manifold, which is an analogue of the Busemann cocycle in rank one. To define this cocycle we need some preliminary definitions. Let 
$$
\mathfrak{a}:=\{{\rm diag}(a_1,\ldots,a_d)\in\mathfrak{sl}(d,\Kb) :\ a_1+\cdots+a_d=0 \}
$$
denote the standard Cartan subspace of $\mathfrak{sl}(d,\Kb)$ and let $\kappa:\mathsf{PSL}(d,\mathbb K)\to\mathfrak{a}$ denote the \emph{Cartan projection} which is given by 
$$
\kappa(g)={\rm diag}(\log\sigma_1(g),\cdots,\log\sigma_d(g))
$$ 
where $\sigma_1(g) \ge \cdots \ge \sigma_d(g)$ are the singular values of some (equivalently, any) lift of $g$ to $\mathsf{SL}(d,\Kb)$. Let $\Delta :=\{ \alpha_1,\dots, \alpha_{d-1}\} \subset \mathfrak{a}^*$ denote the standard system of simple restricted roots, i.e.
$$
\alpha_j( {\rm diag}(a_1,\ldots,a_d)) = a_j-a_{j+1}
$$
for all ${\rm diag}(a_1,\ldots,a_d) \in \mathfrak{a}$. 

When $\theta=\{\alpha_{i_1},\ldots,\alpha_{i_k}\}\subset\Delta$ is symmetric (i.e. $\alpha_k\in\theta$ if and only if $\alpha_{d-k}\in\theta$), we say that a subgroup $\Gamma$ of $\mathsf{PSL}(d,\mathbb K)$ is \emph{$\Psf_\theta$-divergent} if 
$$
+\infty = \lim_{n \rightarrow \infty} \min_{\alpha_k \in \theta} \alpha_k(\kappa(\gamma_n)) = \lim_{n \rightarrow \infty}  \min_{\alpha_k \in \theta} \log\frac{\sigma_k(\gamma_n)}{\sigma_{k+1}(\gamma_n)}
$$
whenever $\{\gamma_n\}$ is a sequence of distinct elements of $\Gamma$. A $\Psf_\theta$-divergent group is discrete and has a well-defined
limit set $\Lambda_\theta (\Gamma)$ in the partial flag variety 
$$
\mathcal F_\theta : = \left\{ (F^{i_1},\dots, F^{i_k}) :\ \mathrm{dim}\left(F^{j}\right)=j \ \text{ for all }\alpha_j\in\theta, \text{ and }F^{i_1}\subset F^{i_2} \subset\cdots\subset F^{i_k}\right\}.
$$ 
A $\Psf_\theta$-divergent subgroup  $\Gamma \subset \mathsf{PSL}(d,\mathbb K)$ is called
\emph{$\Psf_\theta$-transverse} if whenever $F, G \in \Lambda_\theta(\Gamma)$ are distinct, then $F$ and $G$ are transverse (i.e. for all $\alpha_j\in\theta$ the $j$-plane component $F^j$ of $F$ is transverse to the $(d-j)$-plane component $G^{d-j}$ of $G$). We note that in the literature, divergent groups are sometimes called regular and transverse groups are sometimes called antipodal groups (e.g. \cite{KLP1}).

Let 
$$
\mathfrak{a}_\theta:=\{{\rm diag}(a_1,\ldots,a_d) \in \mathfrak{a} :a_j=a_{j+1} \text{ for all } \alpha_j\notin\theta\}
$$
denote the  partial Cartan subspace and let 
$$
\mathfrak{a}_\theta^+:=\{{\rm diag}(a_1,\ldots,a_d) \in \mathfrak{a}_\theta :\ a_1\ge a_2\ge \cdots \ge a_d\}
$$
denote the partial positive Weyl Chamber. For $\alpha \in \Delta$, let $\omega_\alpha \in \mathfrak{a}^*$ denote the fundamental weight associated to $\alpha$. One can check that $\{\omega_{\alpha}|_{\mathfrak{a}_\theta}\}_{\alpha\in\theta}$ is a basis of  $\mathfrak{a}_\theta^*$. Then there is a well-defined partial Cartan projection $\kappa_\theta:\mathsf{PSL}(d,\mathbb K)\to\mathfrak{a}_\theta$ with the defining property that 
$$
\omega_\alpha(\kappa(g))=\omega_\alpha(\kappa_\theta(g))
$$ 
for all $\alpha\in\theta$ and $g\in\mathsf{PSL}(d,\mathbb K)$.

Quint \cite{quint-ps} proved that there exists a cocycle $B_\theta:\mathsf{PSL}(d,\mathbb K)\times\mathcal F_\theta\to\mathfrak{a}_\theta$, called the \emph{partial Iwasawa cocycle}, with the defining property that if $g \in \PSL(d, \Kb)$, $F \in \Fc_\theta$ and $\alpha_j \in \theta$, then 
$$
\omega_{\alpha_j}(B_\theta(g,F))=\log\frac{\norm{\left(\bigwedge^jg\right)(v)}}{\norm{v}}
$$
for any $v\in\bigwedge^jF^j-\{0\}$, where $\bigwedge^j$ is the $j$-th exterior power, and $\norm{\cdot}$ denotes both the standard norm on $\Kb^d$ and the induced norm on $\bigwedge^j\Kb^d$.

Using this cocycle we can define conformal measures and Patterson-Sullivan measures. 

\begin{definition} \label{PSdefintro}
Given $\phi\in\mathfrak{a}_\theta^*$ and a $\Psf_\theta$-divergent group $\Gamma \subset \PSL(d,\Kb)$, a probability measure $\mu$ on $\mathcal F_\theta$ is called a \emph{$\phi$-conformal measure for $\Gamma$ of dimension $\beta$} if  for any $\gamma \in \Gamma$, the measures $\mu, \gamma_*\mu$ are absolutely continuous and 
$$
\frac{d\gamma_*\mu}{d\mu}=e^{-\beta\phi(B_\theta(\gamma^{-1},\cdot))}
$$
almost everywhere. If, in addition, ${\rm supp}(\mu) \subset \Lambda_\theta(\Gamma)$, then we say that $\mu$ is a \emph{$\phi$-Patterson-Sullivan measure}. 
\end{definition} 

In our setting, we do not assume that $\Gamma$ has any irreducibility properties and so there can exist many non-interesting conformal densities, e.g. if $\Gamma$ fixes a flag $F \in \Fc_\theta$, then a Dirac measure centered at $F$ will be a conformal measure of dimension zero. Hence to develop an interesting theory in the setting of (non-irreducible) transverse groups, it is reasonable to restrict to the setting where the measure is supported on the limit set. 

Given a discrete subgroup $\Gamma \subset \PSL(d,\Kb)$ and $\phi \in \mathfrak{a}^*_\theta$, let $\delta^\phi(\Gamma)$ be the
(possibly infinite)  critical exponent of the Poincar\'e series
$$
Q_\Gamma^\phi(s)=\sum_{\gamma\in\Gamma} e^{-s\phi(\kappa_\theta(\gamma))},
$$
that is $ \delta^\phi(\Gamma) \in [0,+\infty]$ is the unique non-negative number where $Q_\Gamma^\phi(s)$ converges when $s > \delta^\phi(\Gamma)$ and diverges when $s < \delta^\phi(\Gamma)$. If $\Gamma\subset\mathsf{PO}(d,1)\subset\mathsf{PSL}(d+1,\mathbb R)$ is  a discrete group, then the traditional Busemann cocycle is $B_{\alpha_1}$,
the traditional Poincar\'e series is simply $Q_\Gamma^{\alpha_1}$ and classical Patterson-Sullivan measures are $\alpha_1$-Patterson-Sullivan measures in our language.

The standard proof, originating in work of Patterson \cite{patterson}, implies that 
if $\Gamma\subset\mathsf{PSL}(d,\mathbb K)$ is $\Psf_\theta$-divergent, $\phi\in \mathfrak{a}^*_\theta$ and $\delta^\phi(\Gamma) < +\infty$, then there exists 
a $\phi$-Patterson-Sullivan measure  for $\Gamma$ of dimension $\delta^\phi(\Gamma)$, see  Proposition \ref{prop: Patterson-Sullivan}. Dey and Kapovich \cite{dey-kapovich} previously established the same
result in the slightly more restrictive setting when $\phi$ is positive on the entire partial Weyl chamber $\mathfrak{a}_\theta^+$. It is straightforward to show that
if $\Gamma$ is $\Psf_\theta$-divergent and $\phi$ is positive on the $\theta$-Benoist limit cone, then $\delta^\phi(\Gamma) < +\infty$, see Proposition \ref{finite critexp}.

\medskip

One immediate consequence of the existence of Patterson-Sullivan measures is a criterion for 
when there is strict inequality between the critical exponent associated to a transverse  group and a subgroup. The study of this ``entropy gap'' was initiated by Brooks \cite{brooks} in the setting of convex cocompact 
Kleinian groups. Coulon, Dal'bo and Sambusetti \cite{CDS} showed that if $\Gamma$ admits a  cocompact, properly discontinuous action on a $\mathrm{CAT}(-1)$-space, then a subgroup of $\Gamma$ has strictly smaller critical exponent if and only if is is co-amenable. The most general current results are due to Coulon, Dougall, Schapira and Tapie  \cite{CDST} who work in the setting of strongly positively recurrent actions on Gromov hyperbolic spaces. Our criterion is obtained using techniques due to  Dal'bo, Otal and Peign\'e  \cite{DOP} .

\begin{theorem}[{see Theorem~\ref{entropy drop}}]
\label{entropy drop intro}
Suppose $\Gamma\subset\mathsf{PSL}(d,\mathbb K)$ is a non-elementary $\Psf_\theta$-transverse subgroup, $\phi\in\mathfrak{a}_\theta^*$ and $\delta^\phi(\Gamma) < +\infty$.
If $G$ is a subgroup of $\Gamma$ such that $Q_G^\phi(\delta^\phi(G))=+\infty$ and $\Lambda_\theta(G)$ is a proper subset of $\Lambda_\theta(\Gamma)$, then
$$\delta^\phi(\Gamma)>\delta^\phi(G).$$
\end{theorem}

In the setting of Anosov groups, we see that there is always an entropy gap for infinite index, quasiconvex subgroups.

\begin{corollary}[{see Corollary~\ref{entropy drop anosov}}]
\label{anosov-drop}
Suppose $\Gamma\subset\mathsf{PSL}(d,\Kb)$ is a non-elementary $\Psf_\theta$-Anosov subgroup and
$G$  is an infinite index quasiconvex subgroup of $\Gamma$. If $\phi\in\mathfrak{a}_\theta^*$ and $\delta^\phi(\Gamma)<+\infty$, then $\delta^\phi(\Gamma)>\delta^\phi(G).$
\end{corollary}

\medskip

For Fuchsian and Kleinian groups, there is a stark contrast in the dynamics of the action of the group which depends on whether or not the Poincar\'e series
diverges at its critical exponent. The analysis of this contrast is known as the Hopf-Tsuji-Sullivan dichotomy and has many aspects. We obtain a version of this dichotomy for transverse groups.

To state the dichotomy precisely we need a few more definitions.  A $\Psf_\theta$-transverse subgroup  $\Gamma \subset \mathsf{PSL}(d,\mathbb R)$ acts on its limit set $ \Lambda_\theta(\Gamma)$ as a convergence group (see \cite[Section 5.1]{KLP1} or \cite[Proposition 3.3]{CZZ2}), 
and hence one can define the set of conical limit points $ \Lambda_\theta^{\rm con}(\Gamma) \subset \Lambda_\theta(\Gamma)$. In the case when $\Gamma$ is $\Psf_\theta$-Anosov, $ \Lambda_\theta^{\rm con}(\Gamma) =\Lambda_\theta(\Gamma)$. We also let $ \Lambda_\theta(\Gamma)^{(2)} \subset  \Lambda_\theta(\Gamma)^2$ denote the space of pairs of transverse flags in the limit set.

Let $\iota : \mathfrak{a} \rightarrow \mathfrak{a}$ be the involution given by 
$$
\iota\big({\rm diag}(a_1,a_2\dots,a_d)\big)={\rm diag}(-a_d,-a_{d-1},\dots,-a_1).
$$ 
Then given $\phi \in \mathfrak{a}_\theta^*$, let $\bar{\phi} := \phi \circ \iota\in \mathfrak{a}_\theta^*$. More explicitly, if $\phi = \sum_{\alpha_j \in \theta} b_j \omega_{\alpha_j}$, then $\bar{\phi} = \sum_{\alpha_j \in \theta} b_{d-j} \omega_{\alpha_j}$.

The following theorem is our version of the Hopf-Sullivan-Tsuji dichotomy for transverse groups.

\begin{theorem}[{see Proposition \ref{prop: consequences of shadow lemma}, Proposition \ref{prop: support on conical limit set}, Corollary \ref{cor: ergodicity} and Corollary \ref{cor:uniqueness}}] \label{thm:dichotomy}Suppose $\Gamma\subset\mathsf{PSL}(d,\mathbb K)$ is a non-elementary $\Psf_\theta$-transverse subgroup, $\phi\in \mathfrak{a}^*_\theta$ and \hbox{$\delta:=\delta^\phi(\Gamma)=\delta^{\bar\phi}(\Gamma) < +\infty$.} Let $\mu$ be a $\phi$-Patterson-Sullivan measure of dimension $\beta$ for $\Gamma$ and let $\bar{\mu}$ be a 
$\bar{\phi}$-Patterson-Sullivan measure of dimension $\beta$ for $\Gamma$.
Then $\beta\ge\delta$ and we have the following dichotomy:
\begin{itemize}
\item If 
$Q_\Gamma^\phi(\beta)=+\infty$, then $\beta=\delta$, and  $\mu$ and $\bar{\mu}$ are the unique Patterson-Sullivan measures of dimension $\delta$. Moreover:
\begin{enumerate}
\item $\mu( \Lambda_\theta^{\rm con}(\Gamma)) = \bar{\mu}( \Lambda_\theta^{\rm con}(\Gamma))=1$. In particular, $\mu$ and $\bar{\mu}$ have no atoms. 
\item The action of $\Gamma$ on $(\Lambda_\theta(\Gamma)^{(2)}, \bar\mu \otimes \mu)$ is conservative.
\item The action of $\Gamma$ on $(\Lambda_\theta(\Gamma)^{(2)}, \bar\mu \otimes \mu)$ is ergodic. 
\item The action of $\Gamma$ on $(\Lambda_\theta(\Gamma), \mu)$ is ergodic. 
\end{enumerate} 
\item If $Q_\Gamma^\phi(\beta) < +\infty$, then:
\begin{enumerate}[label=(\Roman*)]
\item $\mu( \Lambda_\theta^{\rm con}(\Gamma))=\bar{\mu}( \Lambda_\theta^{\rm con}(\Gamma)) = 0$.
\item The action of $\Gamma$ on $(\Lambda_\theta(\Gamma)^{(2)}, \bar\mu \otimes \mu)$ is dissipative.
\item The action of $\Gamma$ on $(\Lambda_\theta(\Gamma)^{(2)}, \bar\mu \otimes \mu)$ is non-ergodic.
\end{enumerate}
\end{itemize}
\end{theorem} 

Notice that if $\beta=\delta$,
then statements (1), (2) and (3) are all equivalent to $Q^\phi_\Gamma(\delta)=+\infty$, and statements (I), (II) and (III) are all equivalent to $Q^\phi_\Gamma(\delta)<+\infty$.

\medskip

The ``divergent'' case of Theorem \ref{thm:dichotomy} contains several important classes of groups. Sambarino \cite[Cor. 5.7.2]{sambarino-dichotomy} proved that for an Anosov group, the Poincar\'e series diverges whenever the critical exponent is finite (this was previously established by Lee-Oh \cite[Lem. 7.11]{lee-oh-invariant} and Dey-Kapovich \cite[Thm. A]{dey-kapovich} in certain cases). In the sequel to this paper we will prove the same result for relatively Anosov groups. 

As an application of Theorem~\ref{thm:dichotomy}, we show that if $\Gamma$ is $\Psf_\theta$-transverse, then 
the critical exponent is a concave function on the space of linear functionals which diverge at their finite critical exponent.
Moreover, we characterize exactly when it fails to be strictly concave in terms of the associated length functions. More precisely, given $\phi \in \mathfrak{a}_\theta^*$, the \emph{$\phi$-length} of $g \in \mathsf{PSL}(d,\Kb)$ is 
$$
\ell^\phi(g) := \lim_{n \rightarrow \infty} \frac{1}{n} \phi(\kappa_\theta(g^n)). 
$$

\begin{theorem}[{see Theorem~\ref{manhattan curve}}] Suppose $\Gamma\subset\mathsf{PSL}(d,\mathbb K)$ is a non-elementary $\Psf_\theta$-transverse subgroup, $\phi_1, \phi_2 \in \mathfrak{a}^*_\theta$ and $\delta^{\phi_1}(\Gamma)=\delta^{\phi_2}(\Gamma) =1$. 
If $\phi = \lambda \phi_1+(1-\lambda)\phi_2$ where $\lambda \in (0,1)$, then 
$$
\delta^{\phi}(\Gamma) \le 1. 
$$
Moreover, if  $\delta^{\phi}(\Gamma) = 1$ and $Q_\Gamma^\phi$ diverges at its critical exponent, then $\ell^{\phi_1}(\gamma) = \ell^{\phi_2}(\gamma)$ 
for all $\gamma \in \Gamma$.
\end{theorem} 

We will explain in Section \ref{manhattan} why one might regard this as a variant of Burger's Manhattan Curve Theorem. By applying a result of Benoist \cite{benoist-limit-cone}, we can conclude that strict concavity holds whenever $\Gamma$ is Zariski dense.

\begin{corollary}[{see Corollary~\ref{cor:manhattan curve Z dense case}}]
Suppose $\Gamma\subset\mathsf{PSL}(d,\mathbb K)$ is Zariski dense and  $\Psf_\theta$-transverse, $\phi_1, \phi_2 \in \mathfrak{a}^*_\theta$, $\phi_1\ne\phi_2$
and $\delta^{\phi_1}(\Gamma)=\delta^{\phi_2}(\Gamma) =1$.
If $\phi = \lambda \phi_1+(1-\lambda)\phi_2$ where $\lambda \in (0,1)$ and $Q_\Gamma^\phi$ diverges at its critical exponent, then $\delta^{\phi}(\Gamma) <1$.
\end{corollary}

\subsection{The geometric framework for the proofs} The key idea in our proofs is to associate to any $\Psf_\theta$-transverse group $\Gamma$ a metric space that $\Gamma$ acts on by isometries, where the boundary action of $\Gamma$ on $\Lambda_\theta(\Gamma)$ embeds into the action of $\Gamma$ on a compactification of that metric space. The metric space we construct has enough hyperbolic-like behavior that some of the classical arguments in hyperbolic geometry can be adapted to work in our setting. This approach to studying transverse groups builds upon on our earlier work in~\cite{CZZ2}. 

The metric spaces we consider in this construction are properly convex domains $\Omega \subset \Pb(\Rb^{d_0})$ endowed with their Hilbert metrics. A discrete subgroup $\Gamma_0 \subset \PSL(d_0,\Rb)$ which preserves a properly convex domain $\Omega \subset \Pb(\Rb^{d_0})$ is called \emph{projectively visible} when the limit set $\Lambda_\Omega(\Gamma_0) \subset \partial \Omega$ is $C^1$-smooth and strictly convex (precise definitions are given in Sections \ref{visibility} and \ref{BMS}). 

The class of projectively visible groups contains the class of Kleinian groups, i.e. discrete subgroups of the isometry group $\mathsf{Isom}(\Hb^d_{\Rb})$ of real hyperbolic $d$-space. This follows from the identification of $\mathsf{PO}(m,1) = \mathsf{Isom}(\Hb^d_{\Rb})$ using the Klein-Beltrami model and the fact that $\mathsf{PO}(m,1)$ preserves the unit ball in an affine chart. 

Given a projectively visible group $\Gamma_0 \subset \PSL(d_0,\Rb)$, a representation  $\rho : \Gamma_0 \rightarrow \PSL(d,\Kb)$ is called \emph{$\Psf_\theta$-transverse} if its image $\Gamma:=\rho(\Gamma_0)$ is a $\Psf_\theta$-transverse subgroup and  there exists a $\rho$-equivariant boundary map $\xi : \Lambda_\Omega(\Gamma_0) \rightarrow \Fc_\theta$ which is a homeomorphism onto $\Lambda_\theta(\Gamma)$ (again, precise definitions are given in Sections \ref{visibility} and \ref{BMS}).

To continue our analogy with hyperbolic geometry, we note that if $\Gamma \subset \mathsf{Isom}(\Hb^d_{\Rb})=\mathsf{PO}(m,1)$ is convex co-compact, then the class of $\Psf_\theta$-transverse representations of $\Gamma$ coincides with the class of $\Psf_\theta$-Anosov representations of $\Gamma$. 

In \cite{CZZ2}, we proved that any $\Psf_\theta$-transverse subgroup of $\PSL(d,\Kb)$ can be realized as the image of a $\Psf_\theta$-transverse representation. In this paper we extend this result to the general semisimple Lie group case, see Theorem \ref{thm:transverse image of visible general}. Using this perspective we will prove a version of the shadow lemma, which is one of the foundational tools in our arguments. 

Shadows in Hilbert geometries can be defined exactly as in hyperbolic geometry: Given a properly convex domain $\Omega\subset\Pb(\Rb^{d_0})$, points $b,p\in\Omega$, and  $r>0$, let $\Oc_r(b,p)$ denote the set of points $x\in\partial\Omega$ for which the projective line segment in $\overline{\Omega}$ with endpoints $b$ and $x$ intersects the open ball of radius $r$ (with respect to the Hilbert metric on $\Omega$) centered at $p$.

\begin{proposition}[{see Proposition~\ref{prop:shadow estimates}}]
Suppose $\theta \subset \Delta$ is symmetric, $\Omega \subset \Pb(\Rb^{d_0})$ is a properly convex domain, $\Gamma_0 \subset \Aut(\Omega)$ is  a non-elementary projectively visible subgroup, 
$\rho : \Gamma_0 \rightarrow \mathsf{PSL}(d,\Kb)$ a 
$\Psf_\theta$-transverse representation with limit map $\xi : \Lambda_\Omega(\Gamma_0) \rightarrow \Fc_\theta$, $\Gamma:=\rho(\Gamma_0)$, $\phi \in \mathfrak{a}^*_\theta$
and $\mu$ is a $\phi$-Patterson-Sullivan measure for $\Gamma$ of dimension $\beta$. For any $b_0 \in \Omega$, there exists $R_0$ such that: if  $r>R_0$, then
there exists $C=C(b_0,r)> 1$ so that 
\begin{align*}
C^{-1} e^{-\beta \phi(\kappa_\theta(\rho(\gamma)))} \le \mu\Big(\xi\left( \Oc_r(b_0,\gamma(b_0)) \cap \Lambda_\Omega(\Gamma_0) \right) \Big)  \le Ce^{-\beta \phi(\kappa_\theta(\rho(\gamma)))}
\end{align*}
for all $\gamma \in \Gamma_0$.
\end{proposition}

The transverse representations perspective also allow us to construct a dynamical system associated to a transverse group. In particular, given a transverse representation $\rho : \Gamma_0 \rightarrow \PSL(d,\Kb)$ 
of a projectively visible group $\Gamma_0 \subset \Aut(\Omega)$ we can consider the unit tangent bundle $T^1 \Omega$ of $\Omega$ (relative to the Hilbert metric) and the subspace $\Usf(\Gamma_0) \subset T^1 \Omega$ of directions where the associated projective geodesic lines has forward and backward endpoints in $\Lambda_{\Omega}(\Gamma_0)$, the limit set of $\Gamma_0$. The subspace $\Usf(\Gamma_0)$  is invariant under the geodesic flow and, by the projectively visible assumption, homeomorphic to $\Lambda^{(2)}_\Omega(\Gamma_0) \times \Rb$. We then use our Patterson-Sullivan measures to construct a Bowen-Margulis-Sullivan measure on the quotient $\Gamma_0 \backslash \Usf(\Gamma_0)$. 

This dynamical system is critical in our work. For instance to prove that the boundary actions are ergodic in Theorem~\ref{thm:dichotomy}, we use a general version of the Hopf Lemma, due to Coud\`ene~\cite{coudeneHopf}, to show that the geodesic flow is ergodic with respect to the Bowen-Margulis-Sullivan measure.

\subsection*{Historical remarks:} In this section we briefly discuss some important prior works concerning Patterson-Sullivan measures for discrete subgroups in higher rank semisimple Lie groups.
\begin{enumerate}
\item  Both Albuquerque \cite{albuquerque}  and Quint \cite{quint-ps} study Patterson-Sullivan measures in the setting of Zariski dense, discrete subgroups of a semisimple group
with finite center. Quint's measures live on flag varieties, as ours do,  while Albuquerque's lie on the visual boundary of the associated symmetric space. 
Link \cite{link-ps} showed if the ray limit set has positive measure, then the action of the group on the ray limit set is ergodic with respect to the measures constructed by Albuquerque.

\item Dey and Kapovich \cite{dey-kapovich} study Patterson-Sullivan measures in the setting of $\Psf_\theta$-Anosov subgroups. 
They proved that when $\Gamma$ is a $\Psf_\theta$-divergent subgroup and $\phi\in\mathfrak{a}_\theta^*$ is positive on $\mathfrak{a}_\theta^+$,
that there is a $\phi$-Patterson-Sullivan measure. In addition, when $\Gamma$ is $\Psf_\theta$-Anosov, they also prove that the Patterson-Sullivan measure is unique, 
the conical limit set has full measure and the action of $\Gamma$ on $\Lambda_\theta(\Gamma)$ is ergodic. 
Their approach is based heavily on studying the action of $\Gamma$ on the associated symmetric space.

\item Sambarino \cite{sambarino15,sambarino-dichotomy} used the thermodynamical formalism to provide an alternative proof of Dey and Kapovich's results for all $\phi\in\mathfrak{a}_\theta^*$ such
that $\delta^\phi(\Gamma)<\infty$. Further, he shows that the  action of $\Gamma$ 
on $\Lambda_\theta(\Gamma)^2$ is ergodic and characterizes linear functionals with critical exponent as exactly those which are strictly positive on the Benoist limit cone.
The thermodynamical formalism requires the existence of an associated dynamical system with a Markov coding and 
this is currently only known to exist for Anosov subgroups and a few other specific groups.

\item In the case when $\Gamma$ is a $\Psf_\theta$-Anosov group which is isomorphic to the fundamental group of a closed negatively curved manifold one can use the perspective in~\cite{ledrappier} to obtain nicely behaved Patterson-Sullivan measures, for details of this approach see~\cite{sambarino14}. 

\item Lee-Oh \cite{lee-oh-dichotomy} prove that if $\Gamma$ is Zariski dense and Anosov with respect to a minimal parabolic subgroup,
then any $\phi$-conformal measure of dimension $\delta^\phi(\Gamma)$ is supported on the limit set and hence a Patterson-Sullivan measure. They also show that the $\phi$-Patterson-Sullivan measure is unique. They derive their result as a consequence of a Hopf-Tsuji-Sullivan dichotomy for the maximal diagonal actions.

\item Burger-Landesberg-Lee-Oh \cite{BLLO} establish a Hopf-Tsuji-Sullivan dichotomy for the actions of discrete Zariski dense subgroups on directional limit sets 
with respect to a directional Poincar\'e series. This version of the dichotomy is different than the one we consider, for instance in Burger-Landesberg-Lee-Oh's dichotomy Anosov groups 
always fall into the convergent case when the rank of the semisimple Lie group is at least four. Using different techniques, Sambarino~\cite{sambarino-dichotomy} gave an extension of this dichotomy to more general subsets of simple roots.

\item Quint \cite{quint-ps} proves the analogue of our shadow lemma for Zariski dense groups. His proof makes crucial use of Zariski density in place of our transversality
assumption. Our shadow lemma, unlike Quint's, can be applied to transverse subgroups whose Zariski closures are not connected or not semisimple. Albuquerque \cite{albuquerque} and Link \cite{link-ps} also establish shadow lemmas in their setting. Unlike Quint \cite {quint-ps}, we only deal with real Lie groups as opposed to Lie groups over local fields. This reality assumption is needed in order for us to associate a flow space to a transverse subgroup, see Theorem \ref{thm:transverse image of visible general} and Section \ref{sec:geodesic flow definition}.

\item  Bray~\cite{Bray}, Blayac~\cite{blayac}, Zhu~\cite{zhu-ergodicity} and Blayac-Zhu~\cite{BZ} study Patterson-Sullivan measures for discrete subgroups $\Gamma \subset \PGL(d,\Rb)$ which preserve a properly convex domain $\Omega$. In their work, the measures have Radon-Nikodym derivatives which involve the Busemann functions obtained from the Hilbert metric, instead of partial Iwasawa cocycles used in other works (including this one). When such discrete subgroups $\Gamma$ are $\{\alpha_1,\alpha_{d-1}\}$-transverse (for example, when every point in the orbital limit set $\Lambda_\Omega(\Gamma)$ of $\Gamma$ is a smooth and strongly extremal point of $\partial\Omega$), the Patterson-Sullivan measures they consider are the pushforward via the natural projection $p:\Fc_{\{\alpha_1,\alpha_{d-1}\}}\to\Pb(\Rb^d)$ of some $(\omega_{\alpha_1}+\omega_{\alpha_{d-1}})$-Patterson-Sullivan measure for $\Gamma$.

\item Quint \cite{quint-ps} defined $\phi$-Patterson-Sullivan measures as the measures $\mu$ that satisfy the (almost everywhere) equation
$$
\frac{d\gamma_*\mu}{d\mu}=e^{-\phi(B_\theta(\gamma^{-1},\cdot))}
$$
instead of the equation given in Definition \ref{PSdefintro}. Notice that $\mu$ is a $\phi$-Patterson-Sullivan measure for $\Gamma$ of dimension $\beta$ in the sense of Definition \ref{PSdefintro}, if and only if $\mu$ is a $\beta\phi$-Patterson-Sullivan measure for $\Gamma$ in the sense of Quint. Furthermore, if $\psi:=\delta^\phi(\Gamma)\phi$, then $\delta^\psi(\Gamma)=1$. Thus, every $\phi$-Patterson-Sullivan measures for $\Gamma$ of dimension $\delta^\phi(\Gamma)$ in the sense of Definition \ref{PSdefintro} is a $\psi$-Patterson-Sullivan measures for $\Gamma$ in the sense of Quint for some $\psi$ such that $\delta^\psi(\Gamma)=1$.

\end{enumerate} 

\subsection*{Acknowledgements} We thank Hee Oh, Andr\'es Sambarino and Aleksander Skenderi for helpful comments on an earlier version of this paper. We also thank the referee of
the original version for many helpful suggestions which improved the readability of the manuscript and for encouraging us to give a more complete statement of
Theorem \ref{thm:dichotomy}.

\section{Background and notation}\label{sec:background} 

In this section, we recall some required background from the theory of semisimple Lie groups, as well as certain properties of discrete subgroups of semisimple Lie groups.

\subsection{Semisimple Lie groups} 
First, we recall some basic terminology and facts from the theory of semisimple Lie groups. 
\textbf{For the rest of the paper}, let $\GG$ be a connected semisimple real Lie group without compact factors and with finite center, let $\mfg$ denote the Lie algebra of $\GG$, and let $b$ be the Killing form on $\mfg$. 

Fix a \emph{Cartan involution} $\tau$ of $\mathfrak g$, i.e. an involution for which the bilinear pairing $\langle\cdot,\cdot\rangle$ on $\mathfrak g$ given by $\langle X,Y\rangle:=-b(X,\tau(Y))$ is an inner product. Let 
\[\mfg = \mfk \oplus \mfp\] 
denote the associated \emph{Cartan decomposition}, i.e. $\mathfrak k$ and $\mathfrak p$ are respectively the $1$ and $-1$ eigenspaces of $\tau$. Note that the Killing form is negative definite on $\mfk$ and positive definite on $\mathfrak{p}$, so $\mfk$ is a maximal compact Lie subalgebra of $\mathfrak g$. Let $\Ksf\subset\GG$ denote the maximal compact Lie subgroup whose Lie algebra is $\mfk$. 

Next, fix a maximal abelian subspace $\mfa \subset \mfp$, also called a \emph{Cartan subspace}. Then let
$$
\mfg = \mfg_0 \oplus \bigoplus_{\alpha \in \Sigma} \mfg_\alpha
$$
be the \emph{restricted root space decomposition} associated to $\mfa$, i.e. for any $\alpha\in\mfa^*$
\[\mathfrak g_\alpha:=\{X\in\mathfrak g:[H,X]=\alpha(H)X\text{ for all }H\in\mfa\},\] 
and 
\[\Sigma:=\{\alpha\in\mfa^*-\{0\}:\mathfrak g_\alpha\ne 0\}\] 
is the set of \emph{restricted roots}. One can verify that $\tau(\mathfrak g_\alpha)=\mathfrak g_{-\alpha}$, \cite[Chap. VI, Prop. 6.52]{Knapp}, so $\Sigma=-\Sigma$.

Next fix an element $H_0\in\mfa-\bigcup_{\alpha\in\Sigma}\ker\alpha$, and let
\[\Sigma^+:=\{\alpha\in\Sigma:\alpha(H_0)>0\}\quad\text{and}\quad\Sigma^-:=-\Sigma^+.\]
Note that $\Sigma=\Sigma^+\cup\Sigma^-$. Let $\Delta\subset\Sigma^+$ be the associated system of \emph{simple restricted roots}, i.e. $\Delta$ consists of all the elements in $\Sigma^+$ that cannot be written as a non-trivial linear combination of elements in $\Sigma^+$. Since $\Sigma$ is an abstract root system on $\mfa^*$, see \cite[Chap. VI, Cor. 6.53]{Knapp}, it follows that $\Delta$ is a basis of $\mfa^*$ and every $\alpha\in\Sigma^+$ is a non-negative (integral) linear combination of elements in $\Delta$, see \cite[Chap. III, Thm. 10.1]{Humphreys}.

\subsubsection{The Weyl group and the opposition involution} The \emph{Weyl group} of $\mfa$ is 
\[W:=\mathsf{N}_{\mathsf{K}}(\mfa) / \mathsf{Z}_{\mathsf{K}}(\mfa),\] 
where $\mathsf{N}_{\mathsf{K}}(\mfa) \subset \Ksf$ is the normalizer of $\mfa$ in $\Ksf$ and  $\mathsf{Z}_{\mathsf{K}}(\mfa) \subset \Ksf$ is the centralizer of $\mfa$ in $\Ksf$. Then $W$ is a finite group that is generated by the reflections of $\mfa$ (equipped with $\langle\cdot,\cdot\rangle$) about the kernels of the restricted roots in $\Delta$, see \cite[Chap. VI, Thm. 6.57]{Knapp}. As such, $W$ acts transitively on the set of \emph{Weyl chambers}, that is the closure of the components of 
$$
\mfa - \bigcup_{\alpha \in \Sigma} \ker \alpha. 
$$
Of these, we refer to
$$
\mfa^+ := \{ X \in \mfa : \alpha(X)\ge 0 \text{ for all } \alpha \in \Delta\}
$$
as the \emph{positive Weyl chamber}.

In $W$, there exists a unique element $w_0$, called the \emph{longest element}, such that 
\begin{equation*}
w_0(\mfa^+) = -\mfa^+.
\end{equation*}
We can then define an involution $\iota : \mfa \rightarrow \mfa$ by $\iota(H) = -w_0 \cdot H$. This is known as the \emph{opposition involution}, and has the following properties. 

\begin{observation} \label{obs: opposition involution}\
\begin{enumerate}
\item If $k_0\in \mathsf{N}_{\mathsf{K}}(\mfa)$ is a representative of the longest element $w_0 \in W$, then
\begin{equation}
\label{eqn:action of longest word} 
{\rm Ad}(k_0) \mfg_\alpha = \mfg_{-\iota^*(\alpha)}
\end{equation}
for all $\alpha \in \Sigma$.
\item $\iota^*(\Delta) = \Delta$. 
\end{enumerate}
\end{observation} 

\subsubsection{Parabolic subgroups and flag manifolds}

Given a subset $\theta \subset \Delta$, the \emph{parabolic subgroup associated to $\theta$}, denoted by $\Psf_\theta=\Psf_\theta^+ \subset \GG$, is the normalizer of 
 $$
\mfu_\theta=\mfu_\theta^+ :=  \bigoplus_{\alpha \in \Sigma^+_\theta} \mfg_\alpha
$$
where $\Sigma^+_\theta := \Sigma^+ \setminus {\rm Span}( \Delta \setminus \theta)$. The \emph{flag manifold associated to $\theta$} is
\[\Fc_\theta=\Fc_\theta^+:=\mathsf{G}/\mathsf{P}_\theta.\]
Similarly, the \emph{standard parabolic subgroup opposite to $\Psf_\theta$}, denoted by $\Psf_\theta^-$, is the normalizer of 
 $$
\mfu_\theta^- :=  \bigoplus_{\alpha \in \Sigma^+_\theta} \mfg_{-\iota^*(\alpha)},
$$
and the \emph{standard flag manifold opposite to $\Fc_\theta$} is  
$$
\Fc_\theta^- := \GG / \Psf_\theta^-. 
$$
Notice that if $k_0\in \mathsf{N}_{\Ksf}(\mfa)$ is a representative of the longest element $w_0\in W$, then Equation~\eqref{eqn:action of longest word} implies that 
\begin{equation}\label{eqn:opposite conjugate form} 
k_0 \Psf_\theta^\pm k_0^{-1} =k_0^{-1} \Psf_\theta^\pm k_0 = \Psf_{\iota^*(\theta)}^\mp.
\end{equation}

We say that two flags $F_1 \in \Fc_\theta^+$ and $F_2 \in \Fc_\theta^-$ are \emph{transverse} if $(F_1, F_2)$ is contained in the $\GG$-orbit of $(\Psf_\theta^+, \Psf_\theta^-)$ in $\Fc_\theta^+ \times \Fc_\theta^-$. Then for any flag $F\in\Fc_\theta^\pm$, let $\Zc_F \subset \Fc_\theta^\mp$ denote the set of flags that are not transverse to $F$. One can verify that the set of transverse pairs in $\Fc_\theta^+ \times \Fc_\theta^-$ is an open and dense subset, so $\Zc_F$ is a closed subset with empty interior. Furthermore, $\Zc_F=\Zc_{F'}$ if and only if $F=F'$.

\subsubsection{Cartan projection} Let $\kappa : \GG \rightarrow \mfa^+$ denote the \emph{Cartan projection}, that is $\kappa(g) \in \mfa^+$ is the unique element such that 
$$
g = m e^{\kappa(g)} \ell
$$
for some $m, \ell \in \Ksf$ (in general $m$ and $\ell$ are not uniquely determined by $g$). Such a decomposition $g = m e^{\kappa(g)} \ell$ is called a \emph{$\mathsf{KAK}$-decomposition} of $g$, see \cite[Chap. IX, Thm. 1.1]{Helgason}. Since $\iota(-\mfa^+) = \mfa^+$, we have the following observation. 

\begin{observation}\label{obs: opposition involution 2}
$\iota( \kappa(g)) = \kappa(g^{-1})$ for all $g \in \GG$. 
\end{observation}

In terms of the $\mathsf{KAK}$-decomposition, the actions of $\GG$ on $\Fc_\theta^+$ and $\Fc_\theta^-$ have the following behavior. See Appendix~\ref{appendix: proof of result about contraction} for a proof. 

\begin{proposition}\label{prop:characterizing convergence in general} Suppose $F^\pm \in \Fc_\theta^\pm$, $\{g_n\}$ is a sequence in $\GG$ and $g_n = m_n e^{\kappa(g_n)} \ell_n$ is a $\mathsf{KAK}$-decomposition for each $n \ge 1$. Then the following are equivalent:  
\begin{enumerate}
\item $m_n\Psf_\theta \rightarrow F^+$, $\ell_n^{-1}\Psf_\theta^- \rightarrow F^-$ and $\lim_{n \rightarrow \infty} \alpha(\kappa(g_n)) = \infty$ for every $\alpha \in \theta$,
\item $g_n(F) \to F^+$ for all $F \in \Fc_\theta^+ \setminus \Zc_{F^-}$, and this convergence is uniform on compact subsets of $\Fc_\theta \setminus \Zc_{F^-}$. 
\item $g_n^{-1}(F) \to F^-$ for all $F \in \Fc_\theta^- \setminus \Zc_{F^+}$, and this convergence is uniform on compact subsets of $\Fc_\theta^- \setminus \Zc_{F^+}$. 
\item There are open sets $\mathcal{U}^\pm\subset\Fc_\theta^\pm$ such that $g_n(F) \to F^+$ for all $F \in \mathcal U^+$ and $g_n^{-1}(F) \to F^-$ for all $F \in \mathcal U^-$.
\end{enumerate}
\end{proposition}

\subsubsection{Weights and partial Cartan projections}
For any $\alpha\in\Sigma$, let $H_\alpha\in\mfa$ satisfy the defining property
\[\langle H_\alpha,X\rangle=\alpha(X)\]
for all $X\in\mfa$. Then for any non-zero $E\in\mathfrak g_\alpha$, ${\rm Span}_{\Rb}(E,\tau(E),H_\alpha)\subset\mathfrak g$ is a Lie sub-algebra isomorphic to $\mathfrak{sl}(2,\Rb)$, and this isomorphism identifies 
\[H_\alpha':=\frac{2H_\alpha}{\langle H_\alpha,H_\alpha\rangle}\in{\rm Span}_{\Rb}(E,\tau(E),H_\alpha)\quad\text{with}\quad\begin{pmatrix}1&0\\0&-1\end{pmatrix}\in\mathfrak{sl}(2,\Rb),\]
see \cite[Chap. VI, Prop. 6.52]{Knapp}. The element $H_\alpha'$ is called the \emph{coroot} associated to $\alpha$. If $\alpha\in\Delta$, the \emph{fundamental weight} associated to $\alpha$ is then the element $\omega_\alpha \in \mfa^*$ such that 
\[\omega_\alpha(H_\beta')=\left\{\begin{array}{ll}
1&\text{ if }\alpha=\beta,\\
0&\text{ if }\alpha\ne\beta
\end{array}\right.\]
for all $\beta\in\Delta$.

Given a subset $\theta \subset \Delta$, the \emph{partial Cartan subspace associated to $\theta$} is 
$$
\mfa_\theta := \{ H \in \mfa : \alpha(H) = 0 \text{ for all } \alpha \in \Delta \setminus \theta\}.
$$

Since $(\Delta\setminus\theta)\cup\{\omega_\alpha:\alpha\in\theta\}$ is a basis of $\mfa^*$, there is a unique projection 
\[p_\theta : \mfa \rightarrow \mfa_\theta\] 
such that $\omega_\alpha(X)=\omega_\alpha(p_\theta(X))$ for all $\alpha\in\theta$ and $X\in\mfa$. Then the \emph{partial Cartan projection associated to $\theta$} is
\[\kappa_\theta:=p_\theta\circ\kappa:\GG\to\mfa_\theta.\]
One can show that $\{ \omega_\alpha|_{\mfa_\theta} : \alpha \in \theta\}$ is a basis of $\mfa_\theta^*$ and hence we will identify 
$$
\mfa_\theta^*={\rm Span}\{ \omega_\alpha : \alpha \in \theta\}\subset\mfa^*.
$$
Note that $\omega_\alpha(\kappa_\theta(g))=\omega_\alpha(\kappa(g))$ for all $\alpha\in\theta$ and $g\in\GG$. So 
\begin{equation}
\label{eqn:stupid equality between projections} 
\phi(\kappa_\theta(g))=\phi(\kappa(g))
\end{equation} 
for all $\phi \in \mfa_\theta^*$ and $g\in\GG$.

Given $\phi \in \mfa_\theta^*$ we define the \emph{$\phi$-length} of an element $g \in \GG$ as 
$$
\ell^\phi(g) = \lim_{n \rightarrow \infty} \frac{1}{n} \phi(\kappa_\theta(g^n))
$$
(notice that this limit exists by Fekete's Subadditive Lemma). Equivalently, one can define the length using the Jordan projection. 

\subsubsection{The partial Iwasawa cocycle} Let $\Usf := \exp( \mfu_\Delta)$. The \emph{Iwasawa decomposition} states that the map 
$$
(k,a, u) \in \Ksf \times \exp(\mfa) \times \Usf \mapsto kau \in \GG
$$
is a diffeomorphism, see \cite[Chap.\ VI, Prop.\ 6.46]{Knapp}. Using this, Quint \cite {quint-ps} defined the Iwasawa cocycle
\[B : \GG \times \Fc_\Delta \rightarrow \mfa\] 
with the defining property that $gk \in \Ksf \cdot \exp(B(g,F)) \cdot \Usf$ for all $(g,F)\in \GG \times \Fc_\Delta$, where $k\in\Ksf$ is an element such that $F = k \Psf_\Delta$. The map $B$ is known as the \emph{Iwasawa cocycle}.

For any $\theta\subset\Delta$, note that $\Psf_\Delta\subset \Psf_\theta$, so the identity map on $\GG$ induces a surjection \hbox{$\Pi_\theta:\Fc_\Delta\to\Fc_\theta$.} The \emph{partial Iwasawa cocycle} is the map 
\[B_\theta: \GG \times \Fc_\theta \rightarrow \mfa_\theta\] 
defined by $B_\theta(g,F) = p_\theta( B(g,F') )$ for some (all) $F'\in\Pi_\theta^{-1}(F)$.
By~\cite[Lem.\ 6.1 and 6.2]{quint-ps}, this is a well-defined cocycle, that is 
$$
B_\theta(gh, F) = B_\theta(g, hF) + B_\theta(h, F)
$$
for all $g,h \in \GG$ and $F \in \Fc_\theta$. 

We will use two estimates from~\cite{quint-ps}. In the next two lemmas, let $\norm{\cdot}$ denote the norm of the inner product $\langle\cdot,\cdot\rangle$ on $\mfa$.

\begin{lemma}[{Quint~\cite[Lem.\ 6.5]{quint-ps}}]
\label{quintlemma}
For any $\epsilon>0$ and distance $\d_{\Fc_\theta}$ on $\Fc_{\theta}$ induced by a Riemannian metric there exists $C=C(\epsilon, \d_{\Fc_\theta})>0$ such that: if $g \in \GG$, $g=me^H \ell$ is a $\mathsf{KAK}$-decomposition, $F \in \Fc_\theta$ and $\d_{\Fc_\theta}\left( F, \Zc_{\ell^{-1} \Psf_\theta^-}\right) > \epsilon$, then 
$$
\norm{ B_\theta(g,F)- \kappa_\theta(g)}<C.
$$
\end{lemma}

\begin{lemma}[{Quint~\cite[Lem.\ 6.6]{quint-ps}}]
\label{differences}
For any $\epsilon>0$ and $g\in\GG$ there exists $C=C(\epsilon, g)>0$ such that: if $h\in\GG$ and $\min_{\alpha \in \theta} \alpha(\kappa(h))>C$, then 
$$
\norm{ \kappa_\theta(gh)-\kappa_\theta(h)-B_\theta(g,U_\theta(h))}<\epsilon.
$$
\end{lemma}

\subsection{When $\theta$ is symmetric}\label{sec:background when theta is symmetric} In this section, as in much of the paper, we will consider the case when $\theta \subset \Delta$ is \emph{symmetric},  that is $\iota^*(\theta)=\theta$. 

As before, let $k_0\in \mathsf{N}_{\mathsf{K}}(\mfa)$ be a representative of the longest element $w_0 \in W$. Then $k_0\Psf_\theta k_0^{-1}=k_0^{-1}\Psf_\theta k_0 =\Psf_\theta^-$, see Equation~\eqref{eqn:opposite conjugate form}. So we can identify $\Fc_\theta $ with $\Fc_\theta^-$ via the map 
$$
g \Psf_\theta^- \mapsto g k_0 \Psf_\theta=gk_0^{-1}\Psf_\theta.
$$
Using this identification, we can speak of two elements in $\Fc_\theta$ being transverse. More explicitly, the flags $g_1 \Psf_\theta$ and $g_2 \Psf_\theta$ in $\Fc_\theta$ are \emph{transverse} if and only if there exists $g \in \GG$ such that $gg_1 \in \Psf_\theta$ and $gg_2 k_0 \in \Psf_\theta^-$. With some abuse of the notation, for a flag $F\in\Fc_\theta$, we now let $\Zc_F \subset \Fc_\theta$ denote the set of flags that are not transverse to $F$. 

Following the notation in~\cite{GGKW}, we define a map 
\[U_\theta: \GG \rightarrow \Fc_\theta\] 
by fixing a $\mathsf{KAK}$-decomposition $g = m_g e^{\kappa(g)} \ell_g$ for each $g \in \GG$ and then letting $U_\theta(g) := m_g\Psf_\theta$. One can show that if $\alpha(\kappa(g)) > 0$ for all $\alpha \in \theta$, then $U_\theta(g)$ is independent of the choice of $\mathsf{KAK}$-decomposition, see \cite[Chap. IX, Thm. 1.1]{Helgason}, and hence $U_\theta$ is continuous on the set 
$$
\left\{ g \in \GG : \alpha(\kappa(g)) > 0 \text{ for all } \alpha \in \theta\right\}.
$$ 
Observation~\ref{obs: opposition involution 2} implies that ${\rm Ad}(k_0)(-\kappa(g)) = \kappa(g^{-1})$ and so 
$$
g^{-1} = \left(\ell_g^{-1} k_0^{-1} \right) e^{\kappa(g^{-1})} \left(k_0 m_g^{-1}\right)
$$
is a $\mathsf{KAK}$-decomposition of $g^{-1}$. So we may assume that $m_{g^{-1}} = \ell_g^{-1} k_0^{-1}$ and $\ell_{g^{-1}} = k_0 m_g^{-1}$ for all $g \in \GG$. Then 
$$
U_\theta(g^{-1}) = \ell_g^{-1} k_0^{-1} \Psf_\theta,
$$
which under our identification $\Fc_\theta^- = \Fc_\theta$ coincides with $\ell_g^{-1}\Psf_\theta^-$. 

Then, in the symmetric case, Proposition~\ref{prop:characterizing convergence in general} can be restated as follows. 

\begin{proposition}[{Proposition~\ref{prop:characterizing convergence in general} in the symmetric case}]\label{prop:characterizing convergence in general symmetric case} Suppose $\theta \subset \Delta$ is symmetric, $F^\pm \in \Fc_\theta$ and $\{g_n\}$ is a sequence in $\GG$. The following are equivalent:  
\begin{enumerate}
\item $U_\theta(g_n) \rightarrow F^+$, $U_\theta(g_n^{-1}) \rightarrow F^-$ and $\lim_{n \rightarrow \infty} \alpha(\kappa(g_n)) = \infty$ for every $\alpha \in \theta$,
\item $g_n(F) \to F^+$ for all $F \in \Fc_\theta \setminus \Zc_{F^-}$, and this convergence is uniform on compact subsets of $\Fc_\theta \setminus \Zc_{F^-}$. 
\item $g_n^{-1}(F) \to F^-$ for all $F \in \Fc_\theta \setminus \Zc_{F^+}$, and this convergence is uniform on compact subsets of $\Fc_\theta \setminus \Zc_{F^+}$. 
\item There are open sets $\mathcal{U}^\pm\subset\Fc_\theta$ such that $g_n(F) \to F^+$ for all $F \in \mathcal U^+$ and $g_n^{-1}(F) \to F^-$ for all $F \in \mathcal U^-$.
\end{enumerate}
\end{proposition}

\subsection{Discrete subgroups of semisimple Lie groups} 
Next, we discuss some terminology for discrete subgroups of $\GG$ and their basic properties.

\subsubsection{Critical exponents}
Let $\Gamma\subset\GG$ be any discrete subgroup and let $\theta\subset\Delta$. For any $\phi\in\mfa_\theta^*$, let $Q_\Gamma^\phi(s)$ denote the Poincar\'e series
$$Q_\Gamma^\phi(s)=\sum_{\gamma\in\Gamma} e^{-s\phi(\kappa_\theta(\gamma))}.$$
Let $\delta^\phi(\Gamma)$ be the critical exponent of $Q_\Gamma^\phi(s)$, i.e. 
\[
\delta^\phi(\Gamma)=\inf\{s > 0:Q_\Gamma^\phi(s)<+\infty\}.
\]
Equivalently, 
\[
\delta^\phi(\Gamma)=\limsup_{T\to\infty}\frac{1}{T}\log\#\{\gamma\in\Gamma:\phi(\kappa_\theta(\gamma))<T\}.
\]

The \emph{$\theta$-Benoist limit cone} of $\Gamma$ is the cone
\[
\mathcal B_\theta(\Gamma):=\left\{ H \in \mfa_\theta^+ : \text{ there exists } \{\gamma_n\} \subset \Gamma \text{ and } t_n \searrow 0 \text{ such that } t_n \kappa_\theta(\gamma_n) \rightarrow H\right\}.
\]
Set
\[
\mathcal B_\theta(\Gamma)^+:=\{\phi\in\mfa_\theta^*:\phi > 0 \text{ on } \Bc_\theta(\Gamma)-\{0\}  \}.
\]
We observe that for any $\phi\in\Bc_\theta(\Gamma)^+$, the critical exponent $\delta^\phi(\Gamma)$ is finite.

\begin{proposition}
\label{finite critexp}
Suppose $\Gamma\subset\GG$ is a discrete group and $\theta\subset\Delta$. If $\phi\in\mathcal B_\theta(\Gamma)^+$, then $\delta^\phi(\Gamma)<+\infty$. In particular, if  $\phi$ is positive on $\mfa_\theta^+$, then $\delta^\phi(\Gamma)<+\infty$. 
\end{proposition}

\begin{proof} Equation~\eqref{eqn:stupid equality between projections}  implies that $\phi$ is positive on $\mathcal B_\Delta(\Gamma)-\{0\}$. Then, since $\mathcal B_\Delta(\Gamma)$ is a cone and $\phi$ is linear, there exists $A>0$ such that 
\[ \phi(H)\ge A\norm{H}\]
for all $H\in \mathcal B_\Delta(\Gamma)$. 

We claim that there exists $B \ge 0$ such that 
\begin{equation}\label{eqn: bounds for phi}
\phi(\kappa(\gamma)) \ge \frac{A}{2} \norm{\kappa(\gamma)} - B
\end{equation}
for all $\gamma \in \Gamma$. Suppose not. Then for each $n \ge 1$ there exists $\gamma_n \in \Gamma$ where 
$$
\phi(\kappa(\gamma_n)) \le \frac{A}{2} \norm{\kappa(\gamma_n)} - n. 
$$
This implies that $\norm{\kappa(\gamma_n)} \rightarrow +\infty$. Passing to a subsequence we can suppose that $\frac{1}{\norm{\kappa(\gamma_n)} } \kappa(\gamma_n) \rightarrow H \in \mathcal B_\Delta(\Gamma)$. Then 
$$
A = A \norm{H} \le \phi(H) = \lim_{n \rightarrow \infty} \phi\left( \frac{1}{\norm{\kappa(\gamma_n)} } \kappa(\gamma_n)  \right) \le \frac{A}{2}. 
$$ 
So we have a contradiction and hence such a $B \ge 0$ exists.

Let $X :=\GG/\Ksf$ and $x_0 :=\Ksf \in X$.  Then endow $X$ with a $\GG$-invariant Riemannian symmetric metric scaled so that 
$$
\d_X(x_0, gx_0) = \norm{\kappa(g)}
$$
for all $g \in \GG$. Since $\phi(\kappa_\theta(g))=\phi(\kappa(g))$ for all $g\in\GG$, the inequality \eqref{eqn: bounds for phi} implies that
$$
\{\gamma\in\Gamma:\phi(\kappa_\theta(\gamma))<T\}\subset\left\{\gamma\in\Gamma:\d_X(x_0, \gamma x_0)<\frac{2T}{A}+\frac{2B}{A}\right\}.
$$
Thus, $\delta^\phi(\Gamma)\le \frac{2}{A}\delta_X(\Gamma)$, where 
\[
\delta_X(\Gamma):=\limsup_{T\to\infty}\frac{\log\#\left\{\gamma\in\Gamma:\d_X(x_0, \gamma x_0) <T\right\}}{T}.
\]

Recall that the volume growth entropy of $X$ is 
$$
h(X) : = \limsup_{T \rightarrow \infty} \frac{\log {\rm Vol}_X( B_T(x_0) )}{T} 
$$
where ${\rm Vol}_X$ is the Riemannian volume on $X$ and $B_T(x_0) \subset X$ is the open ball of radius $T > 0$ centered at $x_0$. Since $X$ has bounded sectional curvature, volume comparison theorems imply that $h(X) <+\infty$.

Fix $r_0 > 0$ and for $T > 0$ let $\Gamma_T : = \left\{\gamma\in\Gamma:\d_X(x_0, \gamma x_0) <T\right\}$. Then
$$
\#\Gamma_T = \frac{1}{{\rm Vol}_X(B_{r_0}(x_0))} \sum_{\gamma \in \Gamma_T} {\rm Vol}_X(B_{r_0}(\gamma x_0)) \le \frac{\#\Gamma_{2r_0} }{{\rm Vol}_X(B_{r_0}(x_0))} {\rm Vol}_X(B_{T+r_0}(x_0)). 
$$
Thus $\delta_X(\Gamma) \le h(X)< +\infty$.
\end{proof}

\subsubsection{$\Psf_\theta$-divergent groups}
\label{defs of groups}
A subgroup $\Gamma\subset\GG$ is \emph{$\Psf_\theta$-divergent} if $\alpha(\kappa(\gamma_n))\to\infty$ for any $\alpha\in\theta$ and any sequence $\{\gamma_n\}$ in $\Gamma$ of pairwise distinct elements. 
Notice that  by Observation \ref{obs: opposition involution 2}, a subgroup $\Gamma\subset\mathsf{G}$ is $\Psf_\theta$-divergent if and only if it is 
$\Psf_{\theta\cup\iota^*(\theta)}$-divergent.

The \emph{$\theta$-limit set} $\Lambda_\theta(\Gamma)$ of $\Gamma$ is the set of accumulation points in $\Fc_\theta$ of $\{U_\theta(\gamma):\gamma\in\Gamma\}$. Using Proposition~\ref{prop:characterizing convergence in general symmetric case}, one can verify that $\Lambda_\theta(\Gamma)$ is a closed, $\Gamma$-invariant subset of $\Fc_\theta$. We will say that $\Gamma$ is {\em non-elementary} if $\Lambda_\theta(\Gamma)$ is infinite. 

We note that in the literature, divergent groups are sometimes called regular groups (e.g. \cite{KLP1}).

\subsubsection{$\Psf_\theta$-transverse groups} In this subsection we assume that $\theta\subset\Delta$ is symmetric, i.e $\iota^*(\theta)=\theta$.  A $\Psf_\theta$-divergent subgroup $\Gamma\subset\GG$ is \emph{$\Psf_\theta$-transverse} if $\Lambda_\theta(\Gamma)$ is a transverse subset of $\Fc_\theta$, i.e. distinct pairs of flags in $\Lambda_\theta(\Gamma)$ are transverse. We note that in the literature, transverse groups are sometimes called antipodal groups (e.g. \cite{KLP1}). One crucial feature of $\Psf_\theta$-transverse groups is that  $\Gamma$ acts on $\Lambda_\theta(\Gamma)$ as a convergence group. 

We recall that the action, by homeomorphisms, of a  group $\Gamma_0$ on a 
compact metric space $X$ is said to be a (discrete)  {\em convergence group action} if whenever
$\{\gamma_n\}$ is a sequence of distinct elements in $\Gamma_0$, then there are points $x,y\in X$ and a subsequence,
still called $\{\gamma_n\}$, so that $\gamma_n(z)$ converges to  $x$ for all $z\in X\setminus\{y\}$ (uniformly on compact subsets of $X\setminus\{y\}$).

\begin{proposition}[{\cite[Section 5.1]{KLP1}, \cite[Proposition 3.3]{CZZ2}}]\label{prop: convergence group}
If $\Gamma$ is $\Psf_\theta$-transverse, then $\Gamma$ acts on $\Lambda_\theta(\Gamma)$ as a convergence group. 
In particular, if $\Gamma$ is non-elementary, then $\Gamma$ acts on $\Lambda_\theta(\Gamma)$ minimally, and $\Lambda_\theta(\Gamma)$ is perfect. 
\end{proposition}

If a group $\Gamma_0$ acts on a metric space $X$ as a convergence group, we say that a point $x\in X$  is a  {\em conical limit point} for the convergence group action if  there exist distinct $a,b\in X$ and a sequence
$\{\gamma_n\}$ in $\Gamma_0$ so that $\gamma_n(x)$ converges to $a$ and $\gamma_n(y)$ converges to  $b$ for all $y\in X\setminus\{x\}$.  

When $\Gamma \subset \GG$ is $\Psf_\theta$-transverse, the set of conical limit points for the action of $\Gamma$ on $\Lambda_\theta(\Gamma)$ is called the {\em $\theta$-conical limit set} and is denoted $\Lambda^{\rm con}_\theta(\Gamma)$.

\subsubsection{Anosov groups}

Anosov groups were introduced by Labourie \cite{labourie-invent} in his work on Hitchin representations and were further developed by Guichard-Wienhard \cite{guichard-wienhard} and others.
They are a natural generalization of the notion of a convex cocompact subgroup of a rank one Lie group into the higher rank setting. There are now many different equivalent definitions, and
we give a definition which is well-adapted to our setting.

Following~\cite{KLP1}, a $\Psf_\theta$-transverse subgroup $\Gamma\subset\GG$ is said to be $\Psf_\theta$-{\em Anosov} if $\Gamma$ is Gromov hyperbolic with Gromov boundary
$\partial\Gamma$ and there exists a  $\Gamma$-equivariant homeomorphism $\xi:\partial\Gamma \to \Lambda_\theta(\Gamma)$.

\subsection{A helpful reduction}\label{sec: a helpful reduction} Since $\GG$ is semisimple, we may decompose its Lie algebra  $\mfg = \bigoplus_{j=1}^m \mfg_j$ into a product of simple Lie algebras. For each $1 \le j \le m$, let $\GG_j \subset \GG$ denote the connected subgroup with Lie algebra $\mfg_j$. The subgroups $\GG_1, \dots, \GG_m$ are called the \emph{simple factors} of $\GG$. One can verify that each simple factor of $\GG$ is a closed, normal subgroup and 
$$
\GG = \GG_1 \cdots \GG_m
$$
is an almost direct product, i.e. any distinct pair of simple factors of $\GG$ commute, and the intersection between $\GG_j$ and $\GG_1\cdots\GG_{j-1}\GG_{j+1}\cdots\GG_m$ is finite for all $j$.

In this section we explain why one can often reduce to the case where $\GG$ has trivial center and the fixed parabolic subgroup contains no simple factors of $\GG$. The main construction needed for this reduction is a well-behaved quotient of $\GG$.

\begin{proposition}\label{prop: reduce}
 For any $\theta \subset \Delta$ symmetric, there is a semisimple Lie group $\GG'$ without compact factors and with trivial center, and a quotient $p:\GG\to\GG'$ with the following properties:
\begin{enumerate}
\item There exists a Cartan decomposition $\mfg'=\mfk'\oplus \mfp'$ of the Lie algebra $\mfg'$ of $\GG'$, a Cartan subspace $\mfa^\prime \subset \mathfrak p^\prime$, and a system of simple restricted roots $\Delta'\subset(\mfa')^*$, so that $(\d p)_{\id}:\mfg\to\mfg'$ sends $\mfk$, $\mfp$ and $\mfa$ to $\mfk'$, $\mfp'$ and $\mfa'$ respectively, and $(\d p)_{\id}^*:(\mfa')^*\to\mfa^*$ identifies $\Delta'$ with a subset of $\Delta$ that contains $\theta$.
\item The parabolic subgroup $\Psf_\theta'\subset\GG'$ corresponding to $\theta\subset\Delta'$ satisfies $p^{-1}(\Psf_\theta')=\Psf_\theta$, and does not contain any simple factors of $\GG'$. Furthermore, if $\Fc_\theta':=\GG'/\Psf_\theta'$, then the map $\xi:\Fc_\theta\to\Fc_\theta'$ given by $\xi:g\Psf_\theta\mapsto p(g)\Psf_\theta'$ is a $p$-equivariant diffeomorphism which preserves transverality.
\item Let $\kappa_\theta:\GG\to\mfa_\theta^+$ and $\kappa_\theta':\GG'\to(\mfa_\theta')^+$ be the partial Cartan projections, and let $B_\theta:\GG\times\Fc_\theta\to\mfa_\theta$ and $B_\theta':\GG'\times\Fc_\theta'\to\mfa_\theta'$ be the partial Iwasawa cocycles. Then $(\d p)_{\id}:\mfg\to\mfg'$ restricts to an isomorphism from $\mfa_\theta$ to $\mfa_\theta'$, and satisfies
\[
(\d p)_{\id}(\kappa_\theta(g)) = \kappa_{\theta}^\prime(p(g)) \quad \text{and}\quad
(\d p)_{\id}(B_\theta(g,F)) = B^\prime_{\theta}(p(g), \xi(F))
\]
for all $g \in \GG$ and $F \in \Fc_\theta$, 
\end{enumerate}
\end{proposition}

Once we have such a Lie group $\GG'$ and quotient map $p:\GG\to\GG'$ as in Proposition~\ref{prop: reduce}, then for any $\Psf_\theta$-transverse subgroup $\Gamma \subset \GG$ and any $\phi \in \mfa_{\theta}^*$, we may set $\Gamma':=p(\Gamma)$ and $\phi': = \phi \circ (\d p)_{\id}|_{\mfa_\theta}^{-1}$. By Proposition~\ref{prop: reduce}, it follows that 
\begin{enumerate}
\item[(I)] $p|_\Gamma$ has finite kernel, $\Gamma^\prime$ is $\Psf_{\theta}^\prime$-transverse, $\xi(\Lambda_\theta(\Gamma)) = \Lambda_{\theta}(\Gamma^\prime)$ and $\xi(\Lambda_\theta^{\rm con}(\Gamma)) = \Lambda_{\theta}^{\rm con}(\Gamma^\prime)$.
\item[(II)] $\phi(\kappa_\theta(\gamma)) = \phi^\prime(\kappa_{\theta}^\prime(p(\gamma)))$ for all $\gamma \in \Gamma$. 
\item[(III)] $\phi(B_\theta(\gamma,F)) = \phi'(B^\prime_{\theta}(p(\gamma), \xi(F)))$ for all $\gamma\in\Gamma$ and $F\in\Lambda_\theta(\Gamma)$.
\end{enumerate}
Thus, any result for $\Gamma\subset\GG$ and $\phi\in\mfa_\theta$ that depends only on $\Lambda_\theta(\Gamma)$, $\Lambda_\theta^{\rm con}(\Gamma)$, $\phi\circ\kappa_\theta$ and $\phi\circ B_\theta$ will hold if and only if they also hold for $\Gamma'\subset\GG'$ and $\phi'\in\mfa_\theta'$. In many situations, this allows us to assume without loss of generality that $\GG$ has trivial center and $\Psf_\theta$ does not contain any simple factors of $\GG$.

\begin{proof}[Proof of Proposition~\ref{prop: reduce}]

Let $\mfp_\theta\subset\mfg$ be the Lie subalgebra corresponding to $\Psf_\theta$. If we set 
\[J :=\{ j : \mfg_j\cap\mfp_\theta=0\}\quad\text{and}\quad J^c :=\{ j : \mfg_j\subset\mfp_\theta\},\]
then $J\cup J^c=\{1,\dots,m\}$. 

Let $\mathsf{H}:=Z(\GG) \prod_{j \in J^c} \GG_j \subset\GG$, $\GG':=\GG/\mathsf{H}$ and $p:\GG\to\GG'$ be the quotient map.
Then observe that via the map $(\d p)_{\rm id}$, we may identify:
\begin{equation}\label{eqn: mfg'}\mfg'=\bigoplus_{j\in J}\mfg_j.\end{equation}
In particular, $\GG'$ is semisimple without compact factors, and has trivial center.

First, we prove part (1). Observe that we may decompose
$$
\mathfrak k = \bigoplus_{j=1}^m \mathfrak k_j, \quad \mathfrak p = \bigoplus_{j=1}^m \mathfrak p_j,\quad \mfa = \bigoplus_{j=1}^m \mfa_j, \quad\Sigma = \bigcup_{j=1}^m \Sigma_j,\quad \Delta = \bigcup_{j=1}^m \Delta_j \quad\text{and} \quad  \mfa^+ = \bigoplus_{j=1}^m \mfa_j ^+, 
$$
where $\mathfrak g_j = \mathfrak k_j \oplus \mathfrak p_j$ is a Cartan decomposition of $\mathfrak g_j$, $\mfa_j \subset \mathfrak p_j$ is a Cartan subspace, $\Sigma_j$ is the set of restricted roots for $ \mfa_j$ and $\Delta_j \subset \Sigma_j$ is a system of simple restricted roots and $\mfa^+_j \subset \mfa_j$ is the positive Weyl chamber relative to $\Delta_j$. Hence, if we set
$$
\mathfrak k^\prime := \bigoplus_{j \in J} \mathfrak k_j, \quad \mathfrak p^\prime: = \bigoplus_{j \in J} \mathfrak p_j, \quad \mfa^\prime := \bigoplus_{j \in J} \mfa_j,\quad \Sigma^\prime : = \bigcup_{j \in J} \Sigma_j,\quad\Delta^\prime := \bigcup_{j \in J} \Delta_j \quad\text{and}\quad (\mfa^\prime)^+ = \bigoplus_{j\in J} \mfa_j ^+
$$
then via the identification \eqref{eqn: mfg'}, $\mathfrak g^\prime= \mathfrak k^\prime \oplus \mathfrak p^\prime$ is a Cartan decomposition of $\mathfrak g^\prime$, $\mfa^\prime \subset \mathfrak p^\prime$ is a Cartan subspace, $\Sigma^\prime$ is the set of restricted roots for $ \mfa^\prime$, $\Delta^\prime\subset\Sigma'$ is a system of simple restricted roots and $(\mfa^{\prime})^+$ is the positive Weyl chamber relative to $\Delta^\prime$. Furthermore, from the definition of $\Psf_\theta$, if $\GG_j$ is a simple factor of $\GG$ that lies in $\Psf_\theta$, then $\theta$ does not intersect $\Delta_j$. This proves part (1).

Next, we prove part (2). The fact that $\Psf_\theta=p^{-1}(\Psf'_\theta)$ is a straightforward verification from the definition of $\Psf_\theta$ and $\Psf'_\theta$. This fact, together with \eqref{eqn: mfg'} imply that $\Psf'_\theta$ does not contain any simple factors of $\GG'$. It is clear that $\xi$ is a $p$-equivariant diffeomorphism. To see that $\xi$ preserves transversality, simply note that the proof that $\Psf_\theta=p^{-1}(\Psf'_\theta)$ also verifies that $\Psf_\theta^-=p^{-1}((\Psf'_\theta)^-)$. Thus, part (2) holds. 

Part (3) holds because with our choice of $\mathfrak a'$, $p$ sends the Cartan and Iwasawa decompositions of $\GG$ to the Cartan and Iwasawa decompositions of $\GG'$ respectively. 
\end{proof}

\section{Patterson-Sullivan measures for divergent groups}
Patterson-Sullivan measures were first constructed by Patterson \cite{patterson} for Fuchsian groups. Subsequently they were constructed in many settings
where there is a natural boundary at infinity and some amount of Gromov hyperbolic behavior. Almost all these constructions mimic Patterson's original
constructions with technical modifications appropriate to the setting. 

Given $\theta \subset \Delta$ symmetric, we will now construct Patterson-Sullivan measures for $\Psf_\theta$-divergent subgroups, using the $\theta$-limit set of the group as the natural boundary. More precisely, given $\phi\in\mfa_\theta^*$ and a $\Psf_\theta$-divergent group $\Gamma \subset \GG$, a probability measure $\mu$ on $\mathcal F_\theta$ is called a \emph{$\phi$-conformal measure for $\Gamma$ of dimension $\beta$} if for any $\gamma \in \Gamma$, the measures $\mu$ and $\gamma_*\mu$ are absolutely continuous and 
$$
\frac{d\gamma_*\mu}{d\mu}(F)=e^{-\beta\phi(B_\theta(\gamma^{-1},F))}.
$$
If, in addition, ${\rm supp}(\mu) \subset \Lambda_\theta(\Gamma)$, then $\mu$ is a \emph{$\phi$-Patterson-Sullivan measure.}

\begin{remark} \,
\begin{enumerate}
\item Since the Radon-Nikodym derivative $\frac{d\gamma_*\mu}{d\mu}$ is only defined almost everywhere, the above equation should be understood to hold only almost everywhere. The same abuse of notation will be used throughout the paper.
\item Notice that in the definition of the partial Iwasawa cocycle $B_\theta$, we implicitly made a choice of a Cartan decomposition of $\mathfrak g$ (equivalently, a choice of maximal compact $\Ksf\subset\GG$) and a choice of a maximal abelian subspace $\mathfrak a\subset\mathfrak g$ that is orthogonal (in the Killing form) to the Lie subalgebra $\mathfrak k\subset\mathfrak g$ of $\Ksf$. In this paper, we fix once and for all a choice of $\Ksf$, and we only consider $\phi$-conformal measures with respect to this fixed $\Ksf$. The choice of $\Ksf$ is equivalent to
a choice of basepoint for $\mathbb H^n$ in the classical case.
\end{enumerate}
\end{remark}

Also, recall that $\delta^\phi(\Gamma)$ is the critical exponent of the Poincar\'e series 
$$
Q_\Gamma^\phi(s) = \sum_{\gamma \in \Gamma} e^{-s\phi(\kappa_\theta(\gamma))}. 
$$

\begin{proposition}\label{prop: Patterson-Sullivan}
If $\theta \subset \Delta$ is symmetric, $\Gamma\subset\GG$ is $\Psf_\theta$-divergent, $\phi\in \mfa^*_\theta$ and $\delta^\phi(\Gamma) < +\infty$, then there is a $\phi$-Patterson-Sullivan measure $\mu$ for $\Gamma$ of dimension $\delta^\phi(\Gamma)$.
\end{proposition}

In the case when $\Gamma$ is a $\Psf_\theta$-Anosov subgroup, Proposition \ref{prop: Patterson-Sullivan} is a consequence of the following theorem of Sambarino \cite{sambarino-dichotomy}, who completely classified the linear functionals which admit Patterson-Sullivan measures 
(see also Lee-Oh \cite{lee-oh-invariant} for the case when $\Gamma$ is Zariski dense and Anosov with respect to a minimal parabolic subgroup and Kapovich-Dey~\cite{dey-kapovich} for the case when $\phi$ is symmetric and positive on $\mfa_\theta^+$).

\begin{theorem}[{Sambarino \cite{sambarino-dichotomy}}]\label{thm:Anosov case of PS}
\label{anosov crit exp finite}
If $\theta \subset \Delta$ is symmetric, $\Gamma\subset\GG$ is $\Psf_\theta$-Anosov and $\phi\in\mfa_\theta^*$, then
the following are equivalent
\begin{enumerate}
\item
$\phi\in \mathcal B_\theta^+(\Gamma)$,
\item
$\delta^\phi(\Gamma)<+\infty$, and
\item
$\Gamma$ admits a $\phi$-Patterson-Sullivan measure of dimension $\delta^\phi(\Gamma)$.
\end{enumerate}
Moreover, if $\delta^\phi(\Gamma)<+\infty$, then $Q_\Gamma^\phi$ diverges at its critical exponent. 
\end{theorem}

The strategy to prove Proposition \ref{prop: Patterson-Sullivan} is to first observe that one can regard  $\Gamma\cup\Lambda_\theta(\Gamma)$ as a well-behaved 
compactification of $\Gamma$. Using this compactification one can simply repeat Patterson's construction verbatim. 

\begin{lemma}\label{lem: limit set compactifies} 
Suppose $\theta\subset\Delta$ is symmetric. If $\Gamma \subset \GG$ is $\Psf_\theta$-divergent, then the set $\Gamma \cup \Lambda_\theta(\Gamma)$ has a topology that makes it a compactification of $\Gamma$. More precisely:
\begin{enumerate}
\item $\Gamma \cup \Lambda_\theta(\Gamma)$ is a compact metrizable space.
\item If $\Gamma$ has the discrete topology, $\Gamma \hookrightarrow \Gamma \cup \Lambda_\theta(\Gamma)$ is an embedding.
\item If $\Lambda_\theta(\Gamma)$ has the subspace topology from $\Fc_\theta$, then $\Lambda_\theta(\Gamma) \hookrightarrow \Gamma \cup \Lambda_\theta(\Gamma)$ is an embedding.
\item A sequence $\{\gamma_n\}$ in $\Gamma$ converges to $F$ in $\Lambda_\theta(\Gamma)$ if and only if 
\[\min_{\alpha\in\theta}\alpha(\kappa(\gamma_n))\to\infty\quad\text{and}\quad U_\theta(\gamma_n) \rightarrow F.\] 
\item The natural left action of $\Gamma$ on $\Gamma \cup \Lambda_\theta(\Gamma)$ is by homeomorphisms. 
\end{enumerate}
Moreover, for any $\eta \in \Gamma$ the function $\bar{B}_\theta(\eta, \cdot) : \Gamma \cup \Lambda_\theta(\Gamma) \rightarrow \mfa_\theta$ defined by 
$$
\bar{B}_\theta(\eta, x) = \begin{cases}
 \kappa_\theta(\eta x) - \kappa_\theta(x)  & \text{ if } x \in \Gamma, \\
B_\theta(\eta, x) & \text{ if } x \in \Lambda_\theta(\Gamma), 
\end{cases}
$$
 is continuous, where the map 
$B_\theta: \GG \times \Fc_\theta \rightarrow \mfa_\theta$ is the partial Iwasawa cocycle.
\end{lemma}

\begin{proof} We will construct an explicit metric on $\Gamma \cup \Lambda_\theta(\Gamma)$. First let $\d_{\Gamma}$ denote the discrete metric on $\Gamma$, that is  
$$
\d_\Gamma(\gamma_1, \gamma_2) = 
\begin{cases} 
1 & \text{ if } \gamma_1 \neq \gamma_2, \\
0 & \text{ if } \gamma_1 = \gamma_2.
\end{cases}
$$
Second, fix a metric $\d_{\theta}$ on $\Fc_\theta$ which is induced by a Riemannian metric. By scaling we can assume that in the metric $\d_\theta$, the diameter of $\Fc_\theta$ is $1$. Finally, define $m_\theta : \Gamma \rightarrow (0,1]$ by 
$$
m_\theta(\gamma) = \exp\left(- \min_{\alpha \in \theta} \alpha(\kappa(\gamma)) \right).
$$

We now define a metric $\d$ on $\Gamma \cup \Lambda_\theta(\Gamma)$ as follows: 
\begin{itemize}
\item If $\gamma_1, \gamma_2 \in \Gamma$, then 
$$
\d(\gamma_1, \gamma_2)  = \max \{ m_\theta(\gamma_1),m_\theta(\gamma_2)\} \d_\Gamma(\gamma_1, \gamma_2)+ \d_{\theta}(U_\theta(\gamma_1), U_\theta(\gamma_2)).
$$
\item If $\gamma \in \Gamma$ and $F \in \Lambda_\theta(\Gamma)$, then 
$$
\d(\gamma, F)  = m_\theta(\gamma) + \d_{\theta}(U_\theta(\gamma), F).
$$
\item If $F_1,F_2 \in \Lambda_\theta(\Gamma)$, then 
$$
\d(F_1, F_2)  = \d_\theta(F_1, F_2). 
$$
\end{itemize}

It is straightforward to check that $\d$ defines a metric. Also, from the definition of $\d$, it is clear that the restriction of $\d$ to $\Gamma$ and $\Lambda_\theta(\Gamma)$ induce the discrete topology on $\Gamma$ and the usual topology on $\Lambda_\theta(\Gamma)$ respectively, so (2) and (3) holds. To see that (4) holds, note that $\gamma_n\to F$ if and only if $m_\theta(\gamma_n)\to 0$ and $\d_\theta(U_\theta(\gamma_n),F)\rightarrow0$, which is in turn equivalent to requiring $\min_{\alpha\in\theta}\alpha(\kappa(\gamma_n))\to\infty$ and $U_\theta(\gamma_n) \rightarrow F$.

Next we prove the compactness in (1) by taking a sequence $\{x_n\}$ in $\Gamma\cup\Lambda_\theta(\Gamma)$ and showing that it has a convergent subsequence. Observe that $\{x_n\}$ either has 
\begin{enumerate}
\item[(i)] a subsequence that lies in $\Lambda_\theta(\Gamma)$, 
\item[(ii)] a subsequence that lies in a finite subset of $\Gamma$, or 
\item [(iii)] a subsequence that lies in $\Gamma$, but does not lie in any finite subset of $\Gamma$. 
\end{enumerate}
If (i) or (ii) holds, then the compactness of $\Lambda_\theta(\Gamma)$ and the compactness of finite subsets of $\Gamma$ respectively imply $\{x_n\}$ has a convergent subsequence. If (iii) holds, then by taking a further subsequence $\{\gamma_j\}$ of $\{x_n\}$, we may assume that $U_\theta(\gamma_j)\to F$ for some $F\in\Fc_\theta$. Since the $\Psf_\theta$-divergence of $\Gamma$ implies that $\min_{\alpha\in\theta}\alpha(\kappa(\gamma_j))\to\infty$, we may apply (4) to deduce that $\gamma_j\to F$. So $\{x_n\}$ has a convergent subsequence.

Since the left $\Gamma$ action on $\Gamma$ and the $\Gamma$ action on $\Lambda_\theta(\Gamma)$ are both clearly continuous, to prove part (5) it suffices to show: if $\eta \in \Gamma$ and $\{\gamma_n\}$ is a sequence in $\Gamma$ converging to $F^+ \in \Lambda_\theta(\Gamma)$, then $\eta\gamma_n \rightarrow \eta(F^+)$. By compactness, it suffices to consider the case when $\eta\gamma_n \rightarrow F^\prime$ and show that $F^\prime = \eta(F^+)$. Notice that (4) implies that $\min_{\alpha\in\theta}\alpha(\kappa(\gamma_n))\to\infty$ and $U_\theta(\gamma_n)\to F^+$. Then using Proposition~\ref{prop:characterizing convergence in general} and passing to a subsequence we can suppose that there exists $F^- \in \Fc_\theta^-$ such that $\gamma_n(F) \rightarrow F^+$ for all $F \in \Fc_\theta \setminus \mathcal{Z}_{F^-}$ and this convergence is uniform on compact subsets of $\Fc_\theta \setminus \Zc_{F^-}$. Then $\eta\gamma_n(F) \rightarrow \eta(F^+)$ for all $F \in \Fc_\theta \setminus \mathcal{Z}_{F^-}$ and this convergence is uniform on compact subsets of $\Fc_\theta \setminus \Zc_{F^-}$. So Proposition~\ref{prop:characterizing convergence in general} implies that  $\min_{\alpha\in\theta}\alpha(\kappa(\eta\gamma_n))\to\infty$ and $U_\theta(\eta\gamma_n)\to \eta(F^+)$. So part (4) implies that $\eta \gamma_n \rightarrow \eta(F^+)$. So part (5) is true. 

Finally notice that Lemma~\ref{differences} and part (4) of this proposition imply the  ``moreover'' part.
\end{proof}

\begin{proof}[Proof of Proposition \ref{prop: Patterson-Sullivan}] Let $\delta: = \delta^\phi(\Gamma)$. Endow $\Gamma \cup \Lambda_\theta(\Gamma)$ with the topology from Lemma~\ref{lem: limit set compactifies} and for $x \in \Gamma \cup \Lambda_\theta(\Gamma)$ let $\mathcal D_{x}$ denote the Dirac measure centered at $x$. By~\cite[Lem.\ 3.1]{patterson} there exists a continuous non-decreasing function $h:\Rb^+\to\Rb^+$ such that:
\begin{enumerate}
\item The series
\[
\hat{Q}(s):=\sum_{\gamma\in \Gamma}h\left(e^{\phi(\kappa_\theta(\gamma))}\right)e^{-s\phi(\kappa_\theta(\gamma))}
\]
converges for $s > \delta$ and diverges for $s \le \delta$. 
\item For any $\epsilon > 0$ there exists $\lambda_0> 0$ such that: if $s > 1$ and $\lambda > \lambda_0$, then $h(\lambda s) \le s^{\epsilon} h(\lambda)$. 
\end{enumerate} 
(In the case when $Q_\Gamma^\phi$ diverges at its critical exponent, we can choose $h \equiv 1$.) Then for $s > \delta$ consider the probability measure 
$$
\mu_s:=\frac{1}{\hat{Q}(s)}\sum_{\gamma\in\Gamma} h\left(e^{\phi(\kappa_\theta(\gamma))}\right)e^{-s\phi(\kappa_\theta(\gamma))}\mathcal D_{\gamma}
$$
on $\Gamma \cup \Lambda_\theta(\Gamma)$. By compactness, the family of measures $\{ \mu_s\}_{s > \delta}$  admits a subsequential weak limit as $s\searrow\delta$ , i.e. there exists $\{s_n\}\subset (\delta,\infty)$ so that $\lim s_n=\delta$ and
\[
\mu:=\lim \mu_{s_n}
\]
exists. We will prove that $\mu$ is a Patterson-Sullivan measure of dimension $\delta$. 

Notice that if  $A \subset  \Gamma$ is a finite set, then 
$$
\mu(A) = \lim_{n \rightarrow \infty} \frac{1}{\hat{Q}(s_n)}\sum_{\gamma\in A} h\left(e^{\phi(\kappa_\theta(\gamma))}\right)e^{-s_n\phi(\kappa_\theta(\gamma))}= 0 \cdot \sum_{\gamma\in A} h\left(e^{\phi(\kappa_\theta(\gamma))}\right)e^{-\delta\phi(\kappa_\theta(\gamma))}=0.
$$
Hence ${\rm supp}(\mu) \subset \Lambda_\theta(\Gamma)$. 

To verify the remaining property, fix $\eta \in \Gamma$, let 
\[\bar{B}_\theta(\eta^{-1}, \cdot):\Gamma \cup \Lambda_\theta(\Gamma) \rightarrow \Rb\] 
be the continuous function defined in Lemma~\ref{lem: limit set compactifies}, and define the function $g_\eta : \Gamma \cup \Lambda_\theta(\Gamma) \rightarrow \Rb$
by 
$$
g_\eta(z) = \begin{cases} 
\frac{h\left(e^{\phi(\kappa_\theta(z)) + \phi(\bar{B}_\theta(\eta^{-1},z))}\right)}{h\left(e^{\phi(\kappa_\theta(z))}\right)} & \text{ if } z \in \Gamma, \\
1 & \text{ if } z \in \Lambda_\theta(\Gamma).
\end{cases}
$$
Notice that property (2) of $h$ implies that $g_\eta$ is continuous. 

For any continuous function $f : \Gamma \cup \Lambda_\theta(\Gamma) \rightarrow \Rb$ and $s > \delta$, we have 
\begin{align*}
\int f(z) d\eta_* \mu_s(z) &= \frac{1}{\hat{Q}(s)} \sum_{\gamma\in\Gamma} h\left(e^{\phi(\kappa_\theta(\gamma))}\right)e^{-s\phi(\kappa_\theta(\gamma))} f(\eta \gamma)  \\
&= \frac{1}{\hat{Q}(s)} \sum_{\gamma\in\Gamma} h\left(e^{\phi(\kappa_\theta(\eta^{-1}\gamma))}\right)e^{-s\phi(\kappa_\theta(\eta^{-1}\gamma))} f(\gamma)  \\
& =  \frac{1}{\hat{Q}(s)} \sum_{\gamma\in\Gamma} h\left(e^{\phi(\kappa_\theta(\gamma))}\right)e^{-s\phi(\kappa_\theta(\gamma))} e^{-s\phi(\bar{B}_\theta(\eta^{-1}, \gamma))}\frac{h\left(e^{\phi(\kappa_\theta(\gamma))+\phi(\bar{B}_\theta(\eta^{-1}, \gamma))}\right)}{h\left(e^{\phi(\kappa_\theta(\gamma))}\right)}f(\gamma) \\
& = \int f(z)e^{-s\phi(\bar{B}_\theta(\eta^{-1}, z))}g_\eta(z)  d\mu_s(z).
\end{align*} 
Then taking limits and recalling that $\mu$ is supported on $\Lambda_\theta(\Gamma)$, we obtain
$$\frac{d\eta_*\mu}{d\mu}(F)=e^{-\delta\phi(B_\theta(\eta^{-1},F))}.$$
So $\mu$ is a Patterson-Sullivan measure of dimension $\delta$.
\end{proof}

\section{Entropy drop}

It is natural to conjecture, in analogy with results of Coulon-Dal'bo-Sambusetti \cite{CDS}, that if $\Gamma_0$ is a subgroup of a $\Psf_\theta$-Anosov group $\Gamma$, $\phi \in \mfa_\theta^*$ and $\delta^\phi(\Gamma) < +\infty$, then $\delta^\phi(\Gamma)=\delta^\phi(\Gamma_0)$ if and only if $\Gamma_0$ is co-amenable in $\Gamma$. (Glorieux and Tapie \cite{glorieux-tapie} have studied this conjecture when $\Gamma_0$ is normal in $\Gamma$ and Zariski dense.)
We apply an argument of Dal'bo-Otal-Peign\'e \cite{DOP} to obtain a  criterion guaranteeing entropy drop for subgroups of transverse groups. As a consequence, 
we obtain generalization of a result of Brooks \cite{brooks} from the setting of geometrically finite hyperbolic 3-manifolds
into the setting of Anosov groups.

\begin{theorem}
\label{entropy drop}
Suppose $\theta \subset \Delta$ is symmetric, $\Gamma\subset\GG$ is a non-elementary $\Psf_\theta$-transverse subgroup, $\phi\in\mfa_\theta^*$ and $\delta^\phi(\Gamma) < +\infty$. 
If $\Gamma_0$ is a subgroup of $\Gamma$ such that $Q_{\Gamma_0}^\phi$ diverges at its critical exponent and $\Lambda_\theta(\Gamma_0)$ is a proper subset of $\Lambda_\theta(\Gamma)$, then
$$\delta^\phi(\Gamma)>\delta^\phi(\Gamma_0).$$
\end{theorem}

\begin{proof}
Let $\mu$ be a $\phi$-Patterson-Sullivan measure for $\Gamma$ of dimension $\delta^\phi(\Gamma)$. 

Fix an open subset $W\subset\Lambda_\theta(\Gamma)$ such that $\overline W \cap \Lambda_\theta(\Gamma_0)=\emptyset$. We claim that 
$$
N: = \#\{ \gamma \in \Gamma_0 : \gamma W \cap W \neq \emptyset\}
$$
is finite. Otherwise there would exist an infinite distinct sequence $\{\gamma_n\} \subset \Gamma_0$ with $\gamma_n W \cap W \neq \emptyset$. Then using Proposition~\ref{prop:characterizing convergence in general symmetric case} and passing to a subsequence we can suppose that there exist $F^+, F^- \in \Lambda_\theta(\Gamma_0)$ such that $\gamma_n(F) \rightarrow F^+$ uniformly on compact subsets of $\Fc_\theta \setminus \mathcal{Z}_{F^-}$. Since $\Lambda_\theta(\Gamma)$ is a transverse set, we see that $\overline W$ is a compact subset of $\Fc_\theta \setminus \mathcal{Z}_{F^-}$. Hence $\gamma_n W \cap W = \emptyset$ for $n$ large. So we have a contradiction and hence $N$ is finite. 

Fix a distance $\d_{\Fc_\theta}$ on $\Fc_\theta$ which is induced by a Riemannian metric. 
Since $\Lambda_\theta(\Gamma_0)$ is the set of accumulation points of $\{ U_\theta(\gamma) : \gamma \in \Gamma_0\}$, there is a finite subset $\Sc\subset \Gamma_0$ and $\epsilon>0$ so  that 
$$
\d_{\Fc_\theta}\left(F, \Zc_{U_\theta(\gamma)}\right)\ge\epsilon
$$
for all $F\in W$ and $\gamma\in \Gamma_0\setminus\Sc$. Then Lemma \ref{quintlemma}  implies that there exists $C>0$ such that
\begin{align}\label{eqn: drop1}
\phi(B_\theta (\gamma,F))\le \phi(\kappa_\theta(\gamma))+C
\end{align}
for all $F\in W$ and $\gamma\in \Gamma_0\setminus\Sc$ (recall that $\theta$ is symmetric and so we can identify $\Fc_\theta$ and $\Fc_\theta^-$, see Section~\ref{sec:background when theta is symmetric}).

Since $\Gamma_0\subset\Gamma$, it is immediate that $\delta^\phi(\Gamma)\ge\delta^\phi(\Gamma_0)$. Suppose for contradiction that  $\delta:=\delta^\phi(\Gamma)=\delta^\phi(\Gamma_0)$. Notice that 
$$\mu(\gamma(W))=\gamma^{-1}_*\mu(W)=\int_W  e^{-\delta \phi(B_\theta (\gamma,F))} d\mu(F),$$
so \eqref{eqn: drop1} implies that
$$\mu(\gamma(W))\ge e^{-\delta C}e^{-\delta \phi(\kappa_\theta (\gamma))} \mu(W)$$
for all $\gamma\in \Gamma_0\setminus\Sc$. 
Since $Q_{\Gamma_0}^\phi$ diverges at its critical exponent,
$$1=\mu(\Lambda_\theta(\Gamma))\ge\frac{1}{N}\sum_{\gamma\in \Gamma_0}\mu(\gamma(W))\ge \frac{e^{-\delta C}\mu(W)}{N}\sum_{\gamma\in \Gamma_0\setminus\Sc} e^{-\delta \phi(\kappa_\theta (\gamma))}=+\infty$$
which is a contradiction.
\end{proof}

One immediate consequence of our criterion is an entropy gap result for quasiconvex subgroups of Anosov groups. 
We recall that a subgroup $\Gamma_0$  of a hyperbolic group $\Gamma$ is {\em quasiconvex} if there exists $K>0$ such that any geodesic joining two points in $\Gamma_0$ in the Cayley graph of $\Gamma$ 
(with respect to some finite presentation of $\Gamma$) lies within distance $K$ of the vertices associated to $\Gamma_0$. 

\begin{corollary}
\label{entropy drop anosov}
Suppose $\theta\subset\Delta$ is symmetric, $\Gamma\subset\GG$ is a non-elementary $\Psf_\theta$-Anosov subgroup and
$\Gamma_0$  is an infinite index quasiconvex subgroup of $\Gamma$. If $\phi\in\mfa_\theta^*$ and $\delta^\phi(\Gamma) < +\infty$, then
$$\delta^\phi(\Gamma)>\delta^\phi(\Gamma_0).$$
\end{corollary}

\begin{proof}
Since $\Gamma\subset\GG$ is a non-elementary $\Psf_\theta$-Anosov subgroup and $\Gamma_0$  is a quasiconvex subgroup of $\Gamma$, 
Canary, Lee, Sambarino and Stover observed  (see \cite[Lem. 2.3]{CLSS}) that $\Gamma_0\subset\GG$ is also a $\Psf_\theta$-Anosov subgroup. Furthermore, since $\Gamma_0\subset\Gamma$ is infinite index we see that $\partial\Gamma_0$ is a proper subset of $\partial\Gamma$, so it follows that 
$\Lambda_\theta(\Gamma_0)$ is a proper subset of $\Lambda_\theta(\Gamma)$. 
Theorem~\ref{thm:Anosov case of PS} implies that  $Q_{\Gamma_0}^\phi$ diverges at its critical exponent, so the corollary now follows from Theorem \ref{entropy drop}. 
\end{proof}

\begin{remark}\
\begin{enumerate}
\item In Corollary \ref{entropy drop anosov} it is not enough to assume that $\Gamma_0$ is infinite index and finitely generated, since the results fails when $\Gamma\subset\mathsf{PO}(3,1)$
uniformizes a closed hyperbolic 3-manifold which fibers over the circle and $\Gamma_0$ is the fiber subgroup. In this case, if we set $\theta:=\{\alpha_1,\alpha_3\}$, then $\Gamma\subset\PGL(4,\Rb)$ is $\Psf_\theta$-Anosov 
and $\Gamma_0\subset\Gamma$ is an infinite index, finitely generated subgroup. However,  in this case, $\delta^{\alpha_1}(\Gamma)=\delta^{\alpha_1}(\Gamma_0)$, see \cite[Cor. 4.2]{canary-laplacian}.
\item Theorem \ref{entropy drop} also gives a new proof of \cite[Prop. 11.5]{CZZ2}.
\end{enumerate}
\end{remark}

\section{Projectively visible groups and their geodesic flows}
\label{visibility}

In this mostly expository section, we recall the definition of projectively visible groups from~\cite{IZ} and state some of their basic properties. Projectively visible groups are a class of transverse groups and we will see in the next section that every transverse group can be identified with a projectively visible group in a useful manner

\subsection{Properly convex domains} We briefly recall some properties of properly convex domains, the Hilbert metric, and the automorphism group of a properly convex domain. For a more detailed discussion we refer the reader to the survey article of Marquis \cite{marquis-around}.

Suppose $\Omega \subset \Pb(\Rb^d)$ is a properly convex domain, that is an open set which is convex and bounded in some affine chart of $\Pb(\Rb^{d})$. Then a \emph{supporting hyperplane} to $\Omega$ at a point $x\in\partial\Omega$ is a projective hyperplane $H\subset\Pb(\Rb^{d})$ (i.e. the projectivization of a codimension one linear subspace) that contains $x$ but does not intersect $\Omega$. By convexity, every boundary point of $\partial \Omega$ is contained in at least one supporting hyperplane and a boundary point which is contained in a unique supporting hyperplane is called a \emph{$C^1$-smooth point} of $\partial \Omega$. In the case when $x$ is a $C^1$-smooth point of $\partial \Omega$, we let $T_x \partial \Omega$ denote the unique supporting hyperplane at $x$. 

For any pair of points $x,y\in\overline{\Omega}$, let $[x,y]_\Omega$ denote the closed projective line segment in $\overline{\Omega}$ with $x$ and $y$ as its endpoints. 
Similarly, $(x,y)_\Omega:=[x,y]_{\Omega}-\{x,y\}$, $[x,y)_\Omega:=[x,y]_{\Omega}-\{y\}$ and $(x,y]_\Omega:=[x,y]_{\Omega}-\{x\}$.

A properly convex domain $\Omega$  admits a natural Finsler metric $\d_\Omega$, called the \emph{Hilbert metric}. Given a pair of points $p,q\in\Omega$, let $x,y\in\partial\Omega$ 
be the points such that that $x,p,q,y$ lie along $[x,y]_\Omega$ in that order. Then 
\[\d_\Omega(p,q):=\log\frac{|x-q||y-p|}{|x-p||y-q|},\]
where $\abs{\cdot}$ denotes some (any) norm on some (any) affine chart containing $x,p,q,y$. Observe that all projective line segments in $\Omega$ are geodesics of the Hilbert metric.

Although the Hilbert metric is rarely ${\rm CAT}(0)$, the distance function has the following well known and useful convexity  property, for a proof see for instance~\cite[Prop. 5.3]{IZ}.

\begin{proposition}
\label{hausdorff-convexity}
Suppose $\Omega\subset\Pb(\Rb^d)$ is a properly convex domain, $x \in \overline{\Omega}$ and $q_1,q_2\in\Omega$.
If $p\in [q_1,x)_\Omega$, then 
$$
\d_\Omega(p,[q_2,x)_\Omega)\le \d_\Omega(q_1,q_2).
$$
\end{proposition}

Given a properly convex domain $\Omega\subset\Pb(\Rb^{d})$, we denote by $\Aut(\Omega)\subset\PGL(d,\Rb)$ the subgroup that leaves $\Omega$ invariant. The group $\Aut(\Omega)$ preserves the Hilbert metric and acts properly on $\Omega$. 

The \emph{full orbital limit set} of a discrete infinite subgroup $\Gamma\subset \Aut(\Omega)$ is 
\[
\Lambda_\Omega(\Gamma):=\left\{x\in\partial\Omega:x=\lim_{n\to\infty}\gamma_n(p)\text{ for some }p\in\Omega\text{ and some }\{\gamma_n\}\subset\Gamma\right\}.
\]
We also let $\Lambda_\Omega^{\rm con}(\Gamma) \subset \Lambda_\Omega(\Gamma)$ denote the set of limit points $x \in  \Lambda_\Omega(\Gamma)$ where there exist $b_0\in\Omega$, a sequence $\{\gamma_n\}$ in $\Gamma$ and some $r>0$ such that $\gamma_n(b_0)\to x$ and $\d_\Omega(\gamma_n(b_0),[b_0,x)_\Omega)<r$ for all $n$.

\subsection{Properties of projectively visible groups}
If $\Omega$ is a properly convex domain, we say that a discrete subgroup $\Gamma\subset\Aut(\Omega)$ is \emph{projectively visible} if 
\begin{enumerate}
\item $(x,y)_\Omega \subset \Omega$ for any two points $x,y \in \Lambda_\Omega(\Gamma)$ and 
\item every point in $\Lambda_\Omega(\Gamma)$ is a $C^1$-smooth point of $\partial\Omega$.
\end{enumerate}

The following proposition collects elementary properties of  projectively visible groups and shows, in particular, that they are examples of transverse subgroups. 

\begin{proposition} \label{prop: limit set projectively visible}
If $\Omega\subset\Pb(\Rb^{d})$ is a properly convex domain and $\Gamma\subset\Aut(\Omega)$ is a projectively visible subgroup, then the following hold:
\begin{enumerate}
\item If $b_0\in\Omega$, then
$\Lambda_\Omega(\Gamma)=\overline{\Gamma(b_0)}\cap\partial\Omega$.
\item
If $\theta=\{\alpha_1,\alpha_{d-1}\}$, then $\Gamma\subset\PGL(d,\Rb)$ is a $\Psf_\theta$-transverse subgroup, with 
\[
\Lambda_\theta(\Gamma)=\left\{(x,T_x\partial\Omega):x\in\Lambda_\Omega(\Gamma)\right\}.
\]
In particular, $\Gamma$ acts as a convergence group on $\Lambda_\Omega(\Gamma)$.
\item If $\{\gamma_n\}$ is a sequence in $\Gamma$  and there exists $b_0\in\Omega$ such that $\gamma_n(b_0)\to x\in\Lambda_\Omega(\Gamma)$ and $\gamma_n\to T\in\Pb(\mathrm{End}(\Rb^d))$, 
then $T$ is the projectivization of a rank $1$ linear map whose image is $x$. Furthermore, if $\gamma_n^{-1}(b_0)\to y$, then $\ker(T)=T_y\partial\Omega$.
\item
$x\in\Lambda_\Omega(\Gamma)$ is a conical limit point (in the convergence group sense) if and only if $x \in \Lambda_\Omega^{\rm con}(\Gamma)$.

\end{enumerate}
\end{proposition}

\begin{proof}
(1): By definition $\overline{\Gamma(b_0)}\cap\partial\Omega \subset \Lambda_\Omega(\Gamma)$. To show the other inclusion, fix $x \in  \Lambda_\Omega(\Gamma)$. Then there is a sequence $\{\gamma_n\}$ in $\Gamma$ and $b_0'\in\Omega$ such that $ \gamma_n(b_0') \rightarrow x$. Passing to a subsequence we can suppose that $ \gamma_n(b_0) \rightarrow x'$. Since 
$$
\lim_{n \rightarrow \infty} \d_\Omega( \gamma_n(b_0'), \gamma_n(b_0) ) = \d_\Omega( b_0', b_0 ),
$$
the definition of the Hilbert metric implies that $[x,x']_\Omega \subset \partial \Omega$. Then, since $\Gamma$ is visible, we must have $x = x^\prime \in \overline{\Gamma(b_0)}\cap\partial\Omega$.

(2): This  was established as \cite[Prop. 3.5]{CZZ2}.

(3): First, $\ker(T) \cap \Omega = \emptyset$ by \cite[Prop. 5.6]{IZ}. Next, note that $T(\Omega)\subset\Lambda_\Omega(\Gamma)$; indeed, if $b\in\Omega$, then $b \notin \ker(T)$ and hence
$$
T(b)=\lim_{n \rightarrow \infty} \gamma_n(b)\in\Lambda_\Omega(\Gamma).
$$ 
Thus, if $b\in\Omega$, then 
\[[T(b),x]_\Omega=[T(b),T(b_0)]_\Omega=T([b,b_0]_\Omega)\subset\Lambda_\Omega(\Gamma),\]
so $T(b)=x$ because $T(b),x\in\Lambda_\Omega(\Gamma)$ and $\Gamma\subset\Aut(\Omega)$ is projectively visible. 
Since $\Omega\subset\Pb(\Rb^{d})$ is open and $T(\Omega)=\{x\}$, it follows that $T$ is the projectivization of a rank $1$ map whose image is $x$. 
By \cite[Prop. 5.6]{IZ}, if $\gamma_n^{-1}(b_0)\to y$, then $y$ lies in the kernel of $T$. Since $\ker(T) \cap \Omega = \emptyset$ and $y$ is a $C^1$-smooth point, we have $\ker(T)=T_y\partial\Omega$.  

(4): This was established as \cite[Lem.~3.6]{CZZ2}.
\end{proof}

\subsection{The geodesic flow}\label{sec:geodesic flow definition} 
Following earlier work of Benoist \cite{benoist-divisible1}, Bray~\cite{Bray} and Blayac \cite{blayac,blayac2}, we now develop the theory of the geodesic flow of a projectively visible group.

First given a properly convex domain $\Omega \subset \Pb(\Rb^d)$,  let $T^1 \Omega \subset T \Omega$ denote the unit tangent bundle with respect to the infinitesimal Hilbert metric. Given $v \in T^1 \Omega$, let $\gamma_v: \Rb \rightarrow \Omega$ denote the unique geodesic line with $\gamma_v^\prime(0)=v$ and whose image is a projective line segment. Also, let
$$
v^\pm : = \lim_{t \rightarrow \pm \infty} \gamma_v(t) \in \partial \Omega. 
$$
The subspace $T^1 \Omega$ has a natural flow, called the \emph{geodesic flow}, which is defined by $\varphi_t(v) = \gamma_v^\prime(t)$. Using this flow, we may define a metric $\d_{T^1 \Omega}$ on $T^1\Omega$ by
\[
\d_{T^1 \Omega}(v,w):=\max_{t\in[0,1]}\d_{\Omega}\big(\pi(\varphi_t(v)),\pi(\varphi_t(w))\big)
\]
where $\pi:T^1\Omega\to\Omega$ takes a vector to its basepoint. It is well-known (see \cite[Lem. 3.4]{benoist-divisible1} for a proof) that two geodesic rays that  end at the same $C^1$-smooth point in the boundary are asymptotic.

\begin{lemma}
\label{asymptotic at smooth point}
Suppose $\Omega\subset\Pb(\Rb^d)$ is a properly convex domain, $v,w \in T^1 \Omega$ and $v^+=w^+$. If $v^+=w^+$ is a $C^1$-smooth point of $\partial\Omega$, then
there exists $T \in \Rb$ such that 
$$
\lim_{t \rightarrow \infty} \d_{T^1 \Omega}\big(\varphi_{t+T}(v),\varphi_t(w)\big)=0. 
$$
\end{lemma}

Next, given a projectively visible subgroup $\Gamma\subset\Aut(\Omega)$, let $\Usf(\Gamma)\subset T^1\Omega$ denote the space of all unit tangent vectors $v$ where $v^+, v^- \in \Lambda_\Omega(\Gamma)$.  Note that $\Usf(\Gamma)$ is  $\varphi_t$-invariant and $\Gamma$-invariant, further 
the $\Gamma$-action on $\Usf(\Gamma)$ is properly discontinuous, and the $\varphi_t$-action on $\Usf(\Gamma)$ commutes with the $\Gamma$-action. As such, $\varphi_t$ descends to a flow, 
still denoted $\varphi_t$, on the quotient
\[
\widehat{\Usf}(\Gamma):=\Gamma \backslash \Usf(\Gamma).
\]
Since the Hilbert metric on $\Omega$ is a length metric, we can define a metric $\d_{\Gamma \backslash \Omega}$ on $\Gamma \backslash \Omega$ by 
$$
\d_{\Gamma \backslash \Omega}(a,b) = \inf\{ \d_\Omega(\tilde{a}, \tilde{b}) : p(\tilde{a}) = a \text{ and } p(\tilde{b})=b\}
$$
where $p : \Omega \rightarrow \Gamma \backslash \Omega$ is the natural projection. Then we may define a metric on $\Gamma \backslash T^1\Omega$ by
\[
\d_{\Gamma \backslash T^1\Omega}(v,w):=\max_{t\in[0,1]}\d_{\Gamma \backslash \Omega}\big(\pi(\varphi_t(v)),\pi(\varphi_t(w))\big)
\]
where $\pi:\Gamma \backslash T^1\Omega\to\Gamma \backslash \Omega$ takes a vector to its basepoint. Notice that if $p : T^1\Omega \rightarrow \Gamma \backslash T^1\Omega$ is the natural projection, then 
\begin{equation}
\label{eqn:the projection between T1's is Lipschitz} 
\d_{\Gamma \backslash T^1\Omega}(p(v),p(w)) \le \d_{T^1\Omega}(v,w)
\end{equation}
for all $v,w \in T^1\Omega$. 

Let $\Lambda_\Omega(\Gamma)^{(2)}$ denote the set of distinct pairs in $\Lambda_\Omega(\Gamma)^2$. Since $\Gamma$ is a projectively visible group, $\Usf(\Gamma)$ is homeomorphic to $\Lambda_\Omega(\Gamma)^{(2)} \times \Rb$. Using horofunctions, this homeomorphism can be made explicit. Bray \cite[Lem.\ 3.2]{Bray} showed that if $y$ is  a $C^1$-smooth point of $\partial\Omega$, there is a well-defined  \emph{horofunction at $y$}
\[
h_y:\Omega\times\Omega\to\Rb
\]
given by 
$$
h_y(a,b):=\lim_{x\to y}\d_\Omega(x,a)-\d_\Omega(x,b),
$$
where the limit is taken over all sequences of points $x$ in $\Omega$ that converge to $y$. 
Since $\Gamma\subset\Aut(\Omega)$ is projectively visible, every point in $\Lambda_\Omega(\Gamma)$ is a $C^1$-smooth point of $\partial\Omega$, so $h_y$ is well-defined for all $y\in\Lambda_\Omega(\Gamma)$. 

For every $b_0\in\Omega$, the \emph{Hopf parameterization} of $\Usf(\Gamma)$ determined by $b_0$ is the identification
$$\Usf(\Gamma)\cong\Lambda_\Omega(\Gamma)^{(2)}\times\mathbb R,$$
where $v \in \Usf(\Gamma)$ is identified with $(v^-, v^+, h_{v^+}(b_0,\pi(v)))$. In this parameterization, the flow $\varphi_t$ on $\Usf(\Gamma)$ is given by 
\[\varphi_t(x,y,s)=(x,y,s+t),\]
and the $\Gamma$ action on $\Usf(\Gamma)$ is given by
\[
\gamma(x,y,s)=(\gamma(x),\gamma(y),s+h_{y}(\gamma^{-1}(b_0),b_0)).
\]

\section{Transverse representations and Bowen-Margulis-Sullivan measures}
\label{BMS}

By results from \cite{CZZ2} and Appendix~\ref{app:proving transverse repn theorems}, we deduce that any $\Psf_\theta$-transverse subgroup $\Gamma \subset \GG$ is the image of a well-behaved representation of a projectively visible subgroup $\Gamma_0 \subset \Aut(\Omega)$. Then, given $\phi\in\mfa_\theta^*$ with $\delta^\phi(\Gamma) < +\infty$, we produce a geodesic flow-invariant measure $m_\phi$ on the unit tangent bundle of $\Omega$, which we call the Bowen-Margulis-Sullivan measure. Later, we will use this measure in our proof of the ergodicity properties of the Patterson-Sullivan measure.

\subsection{Transverse representations} If $\theta \subset \Delta$ is symmetric, $\Omega\subset\Pb(\Rb^{d})$ is a properly convex domain and $\Gamma_0\subset\Aut(\Omega)$ is a projectively visible subgroup, a representation $\rho:\Gamma_0\to\GG$ is said to be
\emph{$\Psf_\theta$-transverse} if there exists a continuous $\rho$-equivariant embedding
\[
\xi:\Lambda_\Omega(\Gamma_0)\to\Fc_\theta
\]
with the following properties:
\begin{enumerate}
\item $\xi(\Lambda_\Omega(\Gamma_0))$ is a transverse subset of $\Fc_\theta$,
\item if $\{\gamma_n\}$ is a sequence in $\Gamma_0$ so that $\gamma_n(b_0)\to x\in\Lambda_\Omega(\Gamma_0)$ and $\gamma_n^{-1}(b_0)\to y\in\Lambda_\Omega(\Gamma_0)$ 
for some (any) $b_0\in\Omega$, then $\rho(\gamma_n)(F)\to\xi(x)$ for all $F\in\Fc_\theta \setminus \Zc_{\xi(y)}$.
\end{enumerate}
We refer to $\xi$ as the \emph{limit map} of $\rho$.

The following observation is a consequence of Proposition~\ref{prop:characterizing convergence in general}. 

\begin{observation}\label{obs: basic properties of transverse repns}
 If $\rho:\Gamma_0\to\GG$ is a $\Psf_\theta$-transverse representation, then $\Gamma:=\rho(\Gamma_0)$ is a $\Psf_\theta$-transverse subgroup and the limit map $\xi$ induces a homeomorphism $\Lambda_\Omega(\Gamma_0) \rightarrow \Lambda_\theta(\Gamma)$. Moreover, 
 \begin{enumerate}
 \item  $\xi(\Lambda_{\Omega}^{\rm con}(\Gamma_0))=\Lambda_{\theta}^{\rm con}(\Gamma)$. 
 \item  If $\{\gamma_n\}$ is a sequence in $\Gamma_0$ so that $\gamma_n(b_0)\to x\in\Lambda_\Omega(\Gamma_0)$ for some $b_0 \in \Omega$, then $U_\theta(\rho(\gamma_n)) \rightarrow \xi(x)$ and $\alpha(\kappa(\rho(\gamma_n))) \rightarrow \infty$ for all $\alpha \in \theta$.  
 \end{enumerate} 
 \end{observation} 
 
 \begin{proof} We begin by proving (2). Fix a sequence $\{\gamma_n\}$ in $\Gamma_0$ so that $\gamma_n(b_0)\to x\in\Lambda_\Omega(\Gamma_0)$ for some $b_0 \in \Omega$. By compactness it suffices to consider the case where $F^+ : = \lim_{n \rightarrow \infty} U_\theta(\rho(\gamma_n))$ and 
 $$
 L:= \lim_{n \rightarrow \infty} \min_{\alpha \in \theta} \alpha(\kappa(\rho(\gamma_n)))  \in \Rb_{\ge 0} \cup \{+ \infty\}
 $$
 both exist, then show that $\xi(x) = F^+$ and $L=+\infty$. Passing to a subsequence we can suppose that $\gamma_n^{-1}(b_0) \rightarrow y$. Then by definition $\rho(\gamma_n)(F)\to\xi(x)$ for all $F\in\Fc_\theta \setminus \Zc_{\xi(y)}$ and $\rho(\gamma_n^{-1})(F)\to\xi(y)$ for all $F\in\Fc_\theta \setminus \Zc_{\xi(x)}$. Since $\Fc_\theta \setminus \Zc_{\xi(y)}$ and $\Fc_\theta \setminus \Zc_{\xi(x)}$ are both open, Proposition~\ref{prop:characterizing convergence in general symmetric case} implies that $\xi(x) = F^+$ and $L=+\infty$. Thus (2) is true. 
 
Then  $\Gamma:=\rho(\Gamma_0)$ is a $\Psf_\theta$-divergent subgroup and $\xi$ induces a homeomorphism $\Lambda_\Omega(\Gamma_0) \rightarrow \Lambda_\theta(\Gamma)$. Further, by definition, $\Lambda_\theta(\Gamma)=\xi(\Lambda_\Omega(\Gamma_0))$ is a transverse subset and hence $\Gamma$ is $\Psf_\theta$-transverse. Finally, Proposition~\ref{prop: limit set projectively visible}(4) implies that $\xi(\Lambda_{\Omega}^{\rm con}(\Gamma_0))=\Lambda_{\theta}^{\rm con}(\Gamma_0)$. 
 \end{proof} 

The next two results were established in~\cite{CZZ2} in the special case when $\GG = \PSL(d,\Rb)$. In Appendix~\ref{app:proving transverse repn theorems} we explain how to reduce the general case to this special case. 

The first result states that under mild conditions on $\GG$ and $\theta$, see Section~\ref{sec: a helpful reduction}, every transverse group is the image of a transverse representation. 

\begin{theorem}\label{thm:transverse image of visible general}
Suppose $Z(\GG)$ is trivial, $\theta \subset \Delta$ is symmetric and $\Psf_\theta$ contains no simple factors of $\GG$. If $\Gamma \subset \GG$ is $\Psf_\theta$-transverse, then there exist $d \in \Nb$, a properly convex domain $\Omega \subset \Pb(\Rb^{d})$, a projectively visible subgroup $\Gamma_0 \subset \mathrm{Aut}(\Omega)$ and a faithful $\Psf_\theta$-transverse representation
$\rho: \Gamma_0 \rightarrow \GG$ with limit map $\xi:\Lambda_{\Omega}(\Gamma_0)\to\mathcal F_\theta$ such that $\rho(\Gamma_0)=\Gamma$ and $\xi(\Lambda_{\Omega}(\Gamma_0)) = \Lambda_{\theta}(\Gamma)$. 
\end{theorem}

It will be useful throughout the paper, to understand how the Cartan projection behaves under multiplication of group elements. The next lemma assures that when two elements translate a basepoint $b_0\in\Omega$ in roughly the same direction, then
the Cartan projection is coarsely additive.

\begin{proposition}\label{prop:multiplicative along geodesics} 
Suppose $\theta \subset \Delta$ is symmetric, $\Omega \subset \Pb(\Rb^{d})$ is a properly convex domain, $\Gamma_0 \subset \Aut(\Omega)$ is  a projectively visible subgroup and
$\rho : \Gamma_0 \rightarrow \GG$ a $\Psf_\theta$-transverse representation. For any $b_0\in\Omega$ and $r > 0$, there exist $C>0$ such that  if $\gamma, \eta \in \Gamma_0$ and 
\begin{align*}
\d_{\Omega}\left(\gamma(b_0),[b_0, \eta(b_0)]_\Omega \right) \le r, 
\end{align*}
then 
\begin{align*}
\norm{\kappa_\theta(\rho(\eta))-\kappa_\theta(\rho(\gamma))-\kappa_\theta(\rho(\gamma^{-1}\eta))} \le C.
\end{align*}
\end{proposition}

\subsection{The Bowen-Margulis-Sullivan measure} Suppose $\theta\subset\Delta$ is symmetric, $\Omega \subset \Pb(\Rb^{d})$ is a properly convex domain, $\Gamma_0 \subset \Aut(\Omega)$ is  a non-elementary projectively visible subgroup and $\rho : \Gamma_0 \rightarrow \GG$ is a $\Psf_\theta$-transverse representation with limit map $\xi:\Lambda_\Omega(\Gamma_0)\to\Fc_\theta$. Let $\Gamma:=\rho(\Gamma_0)$. 

As in Section~\ref{sec:background}, let $\iota : \mathfrak{a} \rightarrow \mathfrak{a}$ denote the opposite involution.  Then fix $\phi\in\mathfrak{a}_\theta^*$ with $\delta:=\delta^\phi(\Gamma) < +\infty$ and let 
\[
\bar\phi:=\phi\circ\iota\in\mathfrak{a}_\theta^*.
\] 
Notice that $\bar\phi(\kappa_\theta (g))=\phi(\kappa_\theta(g^{-1}))$ for all $g\in\GG$, and so $\delta^{\bar{\phi}}(\Gamma) =\delta^\phi(\Gamma)< +\infty$. Finally, suppose $\mu$ is a $\phi$-Patterson-Sullivan measure for $\Gamma$ and $\bar{\mu}$ is a $\bar{\phi}$-Patterson-Sullivan measure for $\Gamma$, both with dimension $\beta$. 

The goal of this section is to construct, using $\rho$, $\mu$ and $\bar{\mu}$, a measure $m$ on $\wh{\Usf}(\Gamma_0)$ that is $\varphi_t$-invariant. We will call this measure the \emph{Bowen-Margulis measure associated to $\rho$, $\mu$ and $\bar{\mu}$}.

Let $\Fc_\theta^{(2)}$ denote the space of pairs of transverse flags in $\Fc_\theta$. Then there exists a continuous function \[[\cdot,\cdot]_\theta:\mathcal F_\theta^{(2)}\to\mathfrak{a}_\theta,\] called the \emph{Gromov product} such that 
\begin{align}\label{eqn: Gromov product}
[g(F),g(G)]_\theta-[F,G]_\theta=-\iota \circ B_\theta(g,F)-B_\theta(g,G)
\end{align}
for all $g\in\GG$ and $(F,G)\in\Fc_\theta^{(2)}$, see \cite[Lem.\ 4.12]{sambarino15}.

 Identify $\Usf(\Gamma_0) = \Lambda_{\Omega}(\Gamma_0)^{(2)} \times \Rb$ via the Hopf parametrization based at a point $b_0 \in \Omega$. Then define a measure $\tilde{m}$ on $\Usf(\Gamma_0)$ by
$$ 
\d \tilde{m}(x,y,s)=e^{-\beta\phi([\xi(x),\xi(y)]_\theta)} \d\bar{\mu}(\xi(x))\otimes \d\mu(\xi(y))\otimes \d t(s)
$$
where $\d t$ is the Lebesgue measure on $\Rb$. This measure is clearly $\varphi_t$-invariant. Furthermore, Equation~\eqref{eqn: Gromov product} and the quasi-invariance property of $\mu$ and $\bar{\mu}$, imply that $\tilde{m}$ is $\Gamma_0$-invariant. Therefore, $\tilde{m}$ descends to a measure $m$ on $\widehat\Usf(\Gamma_0)$ that is $\varphi_t$-invariant.

\section{A shadow lemma for transverse representations} 

Sullivan's shadow lemma, originally proven in the setting of convex cocompact Kleinian groups \cite{sullivan-density}, is a central
tool in the analysis of Patterson-Sullivan measures in many settings. It gives  estimates from above and below on the measure of a shadow in the sphere at infinity of a ball about an orbit point from a light
based at the basepoint.

In the setting of properly convex domains, shadows can be defined as follows: If $\Omega$ is a properly convex domain, $b,p \in \Omega$ and $r > 0$, one defines the shadow
$$
\Oc_r(b,p) := \{ x \in \partial \Omega :  \d_\Omega(p, [b, x)_\Omega) < r\}.
$$

Our version of Sullivan's shadow then has the following form.

\begin{proposition}\label{prop:shadow estimates}
Suppose $\theta \subset \Delta$ is symmetric, $\Omega \subset \Pb(\Rb^{d})$ is a properly convex domain, $\Gamma_0 \subset \Aut(\Omega)$ is  a non-elementary projectively visible subgroup, 
$\rho : \Gamma_0 \rightarrow \GG$ a 
$\Psf_\theta$-transverse representation with limit map $\xi : \Lambda_\Omega(\Gamma_0) \rightarrow \Fc_\theta$, $\Gamma:=\rho(\Gamma_0)$, $\phi \in \mfa^*_\theta$
and $\mu$ is a $\phi$-Patterson-Sullivan measure for $\Gamma$ of dimension $\beta$. For any $b_0 \in \Omega$, there exists $R_0$ such that: if  $r>R_0$, then
there exists $C=C(b_0,r)> 1$ so that 
\begin{align*}
C^{-1} e^{-\beta \phi(\kappa_\theta(\rho(\gamma)))} \le \mu\Big(\xi\big( \Oc_r(b_0,\gamma(b_0)) \cap \Lambda_\Omega(\Gamma_0) \big) \Big)  \le Ce^{-\beta \phi(\kappa_\theta(\rho(\gamma)))}
\end{align*}
for all $\gamma \in \Gamma_0$.
\end{proposition}

\begin{proof}
For notational convenience, we let $\nu$ be the measure on $\partial\Omega$ defined by 
$$
\nu(A) = \mu\big(\xi(A \cap \Lambda_\Omega(\Gamma_0))\big).
$$

By Proposition \ref{prop: convergence group}, the action of
$\Gamma$ on $\Lambda_\theta(\Gamma)$ is minimal, so
the support of $\mu$ is $\Lambda_\theta(\Gamma)$. Also, since $\Lambda_\theta(\Gamma)=\xi(\Lambda_\Omega(\Gamma_0))$, it follows that $\Lambda_\Omega(\Gamma_0)$ is the support of $\nu$. This observation, together with a compactness argument, yields a lower bound on the measure of (large enough) shadows of $b_0$ based at any point in $\Gamma(b_0)$.

\begin{lemma} 
\label{shadow lower bound}
For any $b_0\in\Omega$, there exist $\epsilon_0, R_0 > 0$ such that 
$$
\nu(\Oc_{R_0}(z,b_0)) \ge \epsilon_0
$$
for all $z \in \Gamma_0(b_0)$. 
\end{lemma} 

\begin{proof} Suppose not. Then for every $n \ge 1$ there exists $z_n \in \Gamma_0(b_0)$ such that 
$$
\nu(\Oc_{n}(z_n,b_0)) \le 2^{-n}.
$$
Passing to a subsequence we can suppose that $z_n \rightarrow z \in \Gamma_0(b_0) \cup \Lambda_\Omega(\Gamma_0)$. If $z \in  \Gamma_0(b_0)$, then 
$$
\bigcup_{n =N}^\infty \Oc_{n}(z_n,b_0) = \partial \Omega
$$
for every $N \ge 1$. On the other hand, if $z \in\Lambda_\Omega(\Gamma_0)$, then by assumption, $(z,y)_\Omega\subset\Omega$ for every $y\in\Lambda_\Omega(\Gamma_0)\setminus\{z\}$. This implies that $\d_\Omega(b_0,(z,y)_\Omega)<+\infty$, so
$$
\bigcup_{n = N}^\infty \Oc_{n}(z_n,b_0) \supset \Lambda_\Omega(\Gamma_0)-\{z\}
$$
for every $N \ge 1$. Thus, in either case 
$$
\nu( \Lambda_\Omega(\Gamma_0)-\{z\}) \le \lim_{N \rightarrow \infty} \sum_{n \ge N} \nu(\Oc_{n}(z_n,b_0))  = 0.
$$
Since $ \Lambda_\Omega(\Gamma_0)-\{z\}$ is open in $ \Lambda_\Omega(\Gamma_0)$, which is the support of $\nu$, this is impossible.
\end{proof} 

Next we  use Proposition~\ref{prop:multiplicative along geodesics}  to show that if $x\in\Lambda_\Omega(\Gamma_0)$ lies in the shadow $\Oc_r(b_0, \gamma(b_0))$ for some $\gamma\in\Gamma_0$, 
then $B_\theta(\rho(\gamma)^{-1},\xi(x))$ can be approximated by $\kappa_\theta(\rho(\gamma))$.

\begin{lemma}\label{lem cocycle on shadows} 
For any $r > 0$, there exists $C_1>0$ such that 
$$
\abs{\phi\big(B_\theta(\rho(\gamma)^{-1},\xi(x))+\kappa_\theta(\rho(\gamma))\big)} \le C_1
$$
for all $\gamma \in \Gamma_0$ and $x \in \Oc_r(b_0, \gamma(b_0)) \cap \Lambda_\Omega(\Gamma_0)$.
\end{lemma}

\begin{proof} 
Since $x \in \Lambda_\Omega(\Gamma_0)$, by Proposition~\ref{prop: limit set projectively visible}(1), there exists a sequence $\{\eta_n\}$ in $\Gamma_0$ such that
\hbox{$\eta_n(b_0) \rightarrow x$.}  Since $x\in\mathcal O_r(b_0,\gamma(b_0))$, we have
$$
\d_\Omega(\gamma(b_0), [b_0, x)_\Omega) < r
$$
and hence 
\begin{align*}
\d_\Omega\left( \gamma(b_0), [b_0, \eta_n(b_0)]_\Omega \right)  < r
\end{align*}
for sufficiently large $n$. So, by Proposition~\ref{prop:multiplicative along geodesics}, there exists $C_1>0$ which depends on $r$ and $\phi$, so that
\begin{align*}
\abs{\phi\big(\kappa_\theta(\rho(\gamma))+\kappa_\theta(\rho(\gamma^{-1}\eta_n))-\kappa_\theta(\rho(\eta_n))\big)} \le C_1
\end{align*}
for sufficiently large $n$. Further, Observation~\ref{obs: basic properties of transverse repns} implies that  $U_\theta(\rho(\eta_n)) \rightarrow \xi(x)$. So, by the ``moreover'' part of Lemma~\ref{lem: limit set compactifies}, 
\begin{align*}
& \abs{\phi\big(B_\theta(\rho(\gamma)^{-1},\xi(x))+\kappa_\theta(\rho(\gamma))\big)} = \lim_{n \rightarrow \infty} \abs{\phi\big(B_\theta(\rho(\gamma)^{-1},U_\theta(\rho(\eta_n)))+\kappa_\theta(\rho(\gamma))\big)} \\
&\quad \quad = \lim_{n \rightarrow \infty} \abs{\phi\big(\kappa_\theta(\rho(\gamma^{-1}\eta_n))-\kappa_\theta(\rho(\eta_n))+\kappa_\theta(\rho(\gamma))\big)} \le C_1. \qedhere
\end{align*}

\end{proof}

Now we can complete the proof of Proposition~\ref{prop:shadow estimates}. 

Let $\epsilon_0,R_0>0$ be the constants given by Lemma \ref{shadow lower bound} (which depend on $b_0$).
For any $r \ge R_0$ and $\gamma\in\Gamma_0$, 
\begin{align*}
\nu\Big(\Oc_r(\gamma^{-1}(b_0),b_0)\Big)&=\gamma_*\nu\Big(\Oc_r(b_0,\gamma(b_0))\Big)=\int_{\Oc_r(b_0,\gamma(b_0))}e^{-\beta\phi(B_\theta(\rho(\gamma)^{-1},\xi(x)))}d\nu(x).
\end{align*}
So Lemma \ref{lem cocycle on shadows} implies that there is some $C_1>0$ (which depends on $r$) such that
\[
e^{\beta\phi(\kappa_\theta(\rho(\gamma))))-\beta C_1} \le \frac{\nu\Big(\Oc_r(\gamma^{-1}(b_0),b_0)\Big)}{\nu\Big(\Oc_r(b_0,\gamma(b_0))\Big) } \le e^{\beta\phi(\kappa_\theta(\rho(\gamma))))+\beta C_1}.
\]
Since $r\ge R_0$, Lemma \ref{shadow lower bound} implies that $\epsilon_0\le {\nu\Big(\Oc_r(\gamma^{-1}(b_0),b_0)\Big)}\le 1$, so
\[\epsilon_0e^{-\beta C_1} e^{-\beta \phi(\kappa_\theta(\rho(\gamma)))}\le \nu\Big(\Oc_r(b_0,\gamma(b_0))\Big) \le e^{\beta C_1} e^{-\beta \phi(\kappa_\theta(\rho(\gamma)))}.
\]
Hence the lemma holds with $C := e^{\beta C_1} \epsilon_0^{-1}$. 
\end{proof} 

\section{Consequences of the shadow lemma}

In this section, we collect several standard consequences of the shadow lemma. Most importantly, we see that 
conical limit points cannot be atoms for any Patterson-Sullivan measure and that if the $\phi$-Poincar\'e series converges in the dimension
of the measure, then the conical limit set has measure zero. Later, we will see that if the $\phi$-Poincar\'e series diverges at its critical exponent,
then the conical limit set has full measure in the $\phi$-Patterson-Sullivan measure associated to the critical exponent.

\begin{proposition}\label{prop: consequences of shadow lemma} 
Suppose $\theta \subset \Delta$ is symmetric, $\Gamma\subset\GG$ is a non-elementary $\Psf_\theta$-transverse subgroup, $\phi\in\mfa_\theta^*$ and $\mu$ is a $\phi$-Patterson-Sullivan measure with dimension $\beta$. 
\begin{enumerate}
\item $\beta\ge\delta^\phi(\Gamma)$.
\item If $y\in\Lambda_{\theta}^{\rm con}(\Gamma)$, then $\mu(\{y\})=0$.
\item If $Q^\phi_\Gamma(\beta)<+\infty$, then $\mu(\Lambda_{\theta}^{\rm con}(\Gamma))=0$.
\item If $\{\Gamma_n\}$ is a sequence of increasing subgroups with $\Gamma = \cup \Gamma_n$, then 
$$
\lim_{n \rightarrow \infty} \delta^\phi(\Gamma_n) = \delta^\phi(\Gamma).
$$ 
\end{enumerate}
\end{proposition}

The rest of the section is devoted to the proof of the proposition. Fix a non-elementary, $\Psf_\theta$-transverse group $\Gamma\subset\GG$, $\phi\in\mfa_\theta^*$ and a $\phi$-Patterson-Sullivan measure $\mu$ with dimension $\beta$. 

Using the discussion in Section~\ref{sec: a helpful reduction} we may assume that $\GG$ has trivial center and that $\Psf_\theta$ does not contain any simple factors of $\GG$. By Theorem~\ref{thm:transverse image of visible general}, there is a properly convex domain \hbox{$\Omega\subset\Pb(\Rb^{d})$,} a projectively visible subgroup $\Gamma_0\subset\Aut(\Omega)$ 
and a faithful $\Psf_\theta$-transverse representation $\rho:\Gamma_0\to\GG$ with limit map $\xi:\Lambda_\Omega(\Gamma_0)\to\Fc_\theta$ so that $\rho(\Gamma_0)=\Gamma$ 
and $\xi(\Lambda_\Omega(\Gamma_0))=\Lambda_\theta(\Gamma)$. Further, $\xi(\Lambda_{\Omega}^{\rm con}(\Gamma_0))=\Lambda_{\theta}^{\rm con}(\Gamma)$, see Observation \ref{obs: basic properties of transverse repns}. Define a probability measure $\nu$ on $\partial \Omega$ by 
\begin{equation*}
\nu(A) := \mu\big( \xi(A \cap \Lambda_\Omega(\Gamma_0))\big).
\end{equation*}

Fix $b_0 \in \Omega$. By the Shadow Lemma (Proposition~\ref{prop:shadow estimates}) there is some $R_0 > 0$ such that for every $r \ge R_0$  there exists a constant $C_1=C_1(r) \ge  1$ where
\begin{align}\label{eqn: shadow lemma 8}
C_1^{-1} e^{- \beta\phi(\kappa_\theta(\rho(\gamma)))} \le \nu\big( \Oc_r(b_0,\gamma(b_0)) \big)  \le C_1e^{- \beta\phi(\kappa_\theta(\rho(\gamma)))}
\end{align}
for all $\gamma \in \Gamma$.

\begin{proof}[Proof of part (1)] We will make use of a subdivision of the group into sets of the form
$$
\Ac_n := \{\gamma \in \Gamma_0 :   n < \phi(\kappa_\theta(\rho(\gamma))) \le n+1 \}.
$$
We observe that if elements in a single $\mathcal A_n$ have overlapping shadows, then they are nearby.

\begin{lemma}\label{lem: claim}
For any $r>0$, there exists $C_2=C_2(r)>0$ such that:
if $\gamma_1, \gamma_2 \in \Ac_n$ and $\Oc_r(b_0, \gamma_1(b_0)) \cap \Oc_r(b_0, \gamma_2(b_0)) \neq \emptyset$, then 
$$
\d_\Omega(\gamma_1(b_0), \gamma_2(b_0)) \le C_2.
$$ 
 \end{lemma}

\begin{proof}
Fix  $ x \in \Oc_r(b_0, \gamma_1(b_0)) \cap \Oc_r(b_0, \gamma_2(b_0)) \neq \emptyset$. 
Then for $j=1,2$, there exists $p_j \in [b_0,x)$ such that $\d_\Omega(p_j, \gamma_j(b_0)) < r$. After possibly relabelling we may assume that $p_1 \in [b_0, p_2]$. By Proposition \ref{hausdorff-convexity},
$$
\d_\Omega(\gamma_1(b_0), [b_0, \gamma_2(b_0)])\le \d_\Omega(\gamma_1(b_0), p_1)+\d_\Omega(p_1, [b_0, \gamma_2(b_0)]) \le r +  \d_\Omega(p_2, \gamma_2(b_0)) \le 2r. 
$$
Then by Proposition~\ref{prop:multiplicative along geodesics} there exists a constant $C > 0$ (which depends on $r$) such that 
\begin{align*}
\abs{\phi\Big(\kappa_\theta(\rho(\gamma_1))+\kappa_\theta(\rho(\gamma_1^{-1}\gamma_2))-\kappa_\theta(\rho(\gamma_2))\Big)} \le C.
\end{align*}
Since $\gamma_1, \gamma_2 \in \Ac_n$, it follows that 
$$
\phi(\kappa_\theta(\rho(\gamma_1^{-1}\gamma_2))) \le C+1.
$$
Thus, if we choose 
$$C_2:=\max\{ \d_\Omega(b_0,\gamma(b_0)) : \gamma\in\Gamma_0 \text{ and } \phi(\kappa_\theta(\rho(\gamma))) \le C+1\},$$ 
then
\[
\d_\Omega(\gamma_1(b_0), \gamma_2(b_0)) = \d_\Omega(b_0,\gamma_1^{-1}\gamma_2(b_0)) \le C_2.\qedhere
\]
\end{proof}

Fix $r \ge R_0$, and let $C_2 > 0$ be the constant given by Lemma \ref{lem: claim} for $r$. For each $n$, let $\Ac_n'\subset\Ac_n$ be a maximal collection of elements such that 
\[
\d_\Omega(\gamma_1(b_0),\gamma_2(b_0))>C_2
\] 
for all distinct $\gamma_1,\gamma_2\in\Ac_n'$. Observe that if 
$$
N := \#\{ \gamma \in \Gamma_0 : \d_\Omega(\gamma(b_0), b_0) \le C_2\},
$$
then $\#\Ac_n^\prime \ge \frac{1}{N} \#\Ac_n$. 

By Lemma \ref{lem: claim}, 
$$
\Oc_r(b_0, \gamma_1(b_0)) \cap \Oc_r(b_0, \gamma_2(b_0)) = \emptyset
$$
for all $\gamma_1, \gamma_2 \in \Ac_n'$. Thus, by \eqref{eqn: shadow lemma 8},
\begin{align*}
1=\nu\big(\Lambda_\Omega(\Gamma_0)\big)  &\ge \sum_{\gamma \in \Ac_n^\prime} \nu\big(\Oc_r(b_0, \gamma(b_0))\big) \ge  \frac{1}{C_1} \sum_{\gamma \in \Ac_n^\prime} e^{-\beta \phi(\kappa_\theta(\rho(\gamma)))}   \ge \frac{1}{C_1} \#\Ac_n^\prime e^{-\beta(n+1)}.
\end{align*}
This implies that $\# \Ac_n\le N\#\Ac_n^\prime \le C_1N e^{\beta(n+1)}$. Then
\[
\delta^{\phi}(\Gamma) = \limsup_{n \rightarrow \infty} \frac{1}{n} \log \# \Ac_n \le \beta. \qedhere
\]
\end{proof} 

\begin{proof}[Proof of part (2)]
We first observe that $\beta$ is positive. If this were not the case, then part (1) implies that $\beta = 0$, or equivalently, that $\mu$ is a $\Gamma$-invariant measure on $\Lambda_\theta(\Gamma)$. However, this is impossible because $\Gamma$ acts as a non-elementary convergence group on $\Lambda_\theta(\Gamma)$. 

Let $y\in \Lambda_{\theta}^{\rm con}(\Gamma)$. By Observation \ref{obs: basic properties of transverse repns}(1), $x:=\xi^{-1}(y) \in\Lambda_{\Omega}^{\rm con}(\Gamma_0)$. Then, by definition, there is some $r>0$ and a sequence $\{\gamma_n\}$ in $\Gamma_0$ such that $\gamma_n(b_0) \rightarrow x$ and $\d_\Omega( \gamma_n(b_0), [b_0, x) )<r$ for all $n$. We may assume that $r \ge R_0$. 

By part (1), $\delta^\phi(\Gamma)\le\beta < +\infty$, so $Q_\Gamma^\phi(s)$ converges for $s$ sufficiently large. This implies that 
$$
\lim_{n \rightarrow \infty} \phi(\kappa_\theta(\rho(\gamma_n)))=+\infty.
$$
Since $x\in\mathcal O_{r}(b_0,\gamma_n(b_0))$ for all $n$ and $\beta > 0$, it follows from \eqref{eqn: shadow lemma 8} that
\[
\mu(\{y\})\le\liminf_{n\to\infty}\nu\big(\Oc_r(b_0, \gamma_n(b_0))\big)\le C_1\liminf_{n\to\infty} e^{-\beta\phi(\kappa_\theta(\rho(\gamma_n)))}=0. \qedhere
\]
\end{proof} 

\begin{proof}[Proof of part (3)] For $r > 0$ let $\Lambda_{\Omega,b_0,r}(\Gamma_0) \subset \Lambda_\Omega(\Gamma_0)$ denote the set of  limit points $x$  where there is a sequence $\{\gamma_n\}$ in $\Gamma_0$ such that $\gamma_n(b_0) \rightarrow x$ and $\d_\Omega( \gamma_n(b_0), [b_0, x) )<r$ for all $n$. Notice that $\Lambda_{\Omega}^{\rm con}(\Gamma_0)=\bigcup_{n\in\mathbb N}\Lambda_{\Omega,b_0,n}(\Gamma_0)$. Therefore, it suffices to show that $\mu(\xi(\Lambda_{\Omega,b_0,r}(\Gamma_0)))=0$ for all $r \ge R_0$.

Fix $r\ge R_0$, fix an enumeration $\Gamma = \{ \gamma_1,\gamma_2,\dots\}$ and let $F_n : = \{ \gamma_1, \dots, \gamma_n\}$. Then for any $n$, 
$$
\Lambda_{\Omega,b_0,r}(\Gamma_0) \subset \bigcup_{\gamma\in\Gamma -F_n} \mathcal O_r(b_0,\gamma(b_0)),
$$
so by \eqref{eqn: shadow lemma 8},
$$
\nu\big(\Lambda_{\Omega,b_0,r}(\Gamma_0)\big)\le \sum_{\gamma\in\Gamma -F_n}\nu \big(\mathcal O_r(b_0,\gamma(b_0))\big)\le C_1\sum_{\gamma\in\Gamma-F_n} e^{-\beta\phi(\kappa_\theta(\rho(\gamma)))}.
$$
However, since $Q_\Gamma^\phi(\beta)<+\infty$,
$$
\lim_{n \rightarrow \infty} \sum_{\gamma\in\Gamma-F_n} e^{-\beta\phi(\kappa_\theta(\rho(\gamma)))} = 0.
$$
Therefore, $\nu\big(\Lambda_{\Omega,b_0,r}(\Gamma_0)\big)=0$ for $r \ge R_0$.
\end{proof}

\begin{proof}[Proof of part (4)] Since $\{\Gamma_n\}$ is a sequence of increasing subgroups, $\delta^{\phi}(\Gamma_1) \le \delta^{\phi}(\Gamma_2) \le \dots$
and hence $\delta : = \lim_{n \rightarrow \infty} \delta^\phi(\Gamma_n) \in \Rb \cup \{+\infty\}$ exists. Further, $\delta \le \delta^{\phi}(\Gamma)$. If $\delta = +\infty$, then 
$$
\delta^{\phi}(\Gamma)=+ \infty  = \lim_{n \rightarrow \infty} \delta^\phi(\Gamma_n).
$$
If $\delta < +\infty$, then for each $n$ there exists a $\phi$-Patterson-Sullivan measure $\mu_n$ for $\Gamma_n$ with dimension $\delta^{\phi}(\Gamma_n)$. If $\mu$ is a weak-$*$ limit point of $\{\mu_n\}$, then $\mu$ is a $\phi$-Patterson-Sullivan measure for $\Gamma$ with dimension $\delta$. Hence by part (1) we have $\delta \ge \delta^\phi(\Gamma)$.  
\end{proof}

\section{The conical limit set has full measure in the divergent case}

In this section we show that the Patterson-Sullivan measure is supported on the conical limit set  in case when the associated Poincar\'e series diverges at its critical exponent. The proof is similar to Roblin's~\cite{roblin} argument  for the analogous result in ${\rm CAT}(-1)$ spaces -- in that we use a variant of the Borel-Cantelli Lemma. However, we use a different variant of the lemma and apply it to a different collection of sets. This seems to simplify the argument and this approach was developed during discussions between the authors and Pierre-Louis Blayac.

\begin{proposition} \label{prop: support on conical limit set}
Suppose $\theta \subset \Delta$ is symmetric, $\Gamma\subset\GG$ is a non-elementary $\Psf_\theta$-transverse subgroup, $\phi \in \mfa_\theta^*$, $\delta^\phi(\Gamma) < +\infty$ and $\mu$ is a $\phi$-Patterson-Sullivan measure for $\Gamma$ with dimension $\delta:=\delta^\phi(\Gamma)$. If $Q_\Gamma^\phi(\delta)=+\infty$, then  $\mu(\Lambda_{\theta}^{ \rm con}(\Gamma))=1$. In particular, $\mu$ has no atoms.
\end{proposition}

We will use the following variant of the Borel-Cantelli Lemma, sometimes called the Kochen-Stone Lemma. 

\begin{lemma}[{Kochen-Stone Lemma~\cite{KochenStone}}]\label{lem: KS BC lemma}
Let $(X,\mu)$ be a finite measure space. If $\{ A_n\}$ is a sequence of measurable sets where $\sum_{n =1}^\infty \mu(A_n) = +\infty$ and 
$$
\liminf_{N \rightarrow\infty} \frac{ \sum_{1 \le m,n \le N} \mu(A_n \cap A_m)}{\left(\sum_{n=1}^N \mu(A_n)\right)^2} < +\infty,
$$
then  the set $\{ x \in X: x \text{ is in infinitely many of } A_1,A_2,\dots\}$ has positive $\mu$ measure. 
\end{lemma}

For the rest of the section fix $\Gamma$, $\phi$ and $\mu$ as in the statement of Proposition~\ref{prop: support on conical limit set}. 
Using the discussion in Section~\ref{sec: a helpful reduction} we may assume that $\GG$ has trivial center and that $\Psf_\theta$ does not contain any simple factors of $\GG$. Then by Theorem~\ref{thm:transverse image of visible general}, there is a properly convex domain $\Omega\subset\Pb(\Rb^{d})$, a projectively visible subgroup $\Gamma_0\subset\Aut(\Omega)$
and a faithful $\Psf_\theta$-transverse representation $\rho:\Gamma_0\to\GG$ with limit map $\xi:\Lambda_\Omega(\Gamma_0)\to\Fc_\theta$ so that $\rho(\Gamma_0)=\Gamma$ 
and $\xi(\Lambda_\Omega(\Gamma_0))=\Lambda_\theta(\Gamma)$. Define a measure $\nu$ on $\partial \Omega$ by 
\begin{equation*}
\nu(A) = \mu\big( \xi(A \cap \Lambda_\Omega(\Gamma_0))\big).
\end{equation*}

Fix $b_0 \in \Omega$. Then using  Proposition~\ref{prop:shadow estimates} we may fix $C,r > 0$ such that 
\begin{equation}
\label{eqn:shadow lemma in proof using BC lemma} 
\frac{1}{C} e^{-\delta \phi(\kappa_\theta(\rho(\gamma)))} \le \nu\left(\Oc_r(b_0, \gamma(b_0))  \right) \le C e^{-\delta \phi(\kappa_\theta(\rho(\gamma)))}
\end{equation} 
for all $\gamma \in \Gamma_0$. Fix an enumeration $\Gamma_0 = \{ \gamma_1,\gamma_2,\dots\}$ and let $T_n : = \d_\Omega(b_0, \gamma_n(b_0))$. By reordering we may assume that 
$$
T_1 \le T_2 \le T_3 \le \cdots. 
$$
Then let $A_n : = \Oc_r(b_0, \gamma_n(b_0))$. We will verify that the sets $\{A_n\}$ satisfy the hypotheses of Lemma~\ref{lem: KS BC lemma}.

The first hypothesis in Lemma~\ref{lem: KS BC lemma}  is easy to check. Directly from Equation~\eqref{eqn:shadow lemma in proof using BC lemma} we obtain 
$$
\sum_{n=1}^\infty \nu(A_n) \ge \frac{1}{C} \sum_{n=1}^\infty e^{-\delta \phi(\kappa_\theta(\rho(\gamma_n)))}=\frac{1}{C}Q_{\Gamma}^\phi(\delta)=+\infty. 
$$

Verifying the second hypothesis in Lemma~\ref{lem: KS BC lemma}  is slightly more involved. We require the following technical result, which informally says that the ``boundaries'' of sums of the form $\sum_{n=1}^N e^{-\delta \phi(\kappa_\theta(\rho(\gamma_n)))}$ are controlled by their ``interiors.''

For $N \in \Nb$, set
$$
N^\prime: = \max\{ n \in \Nb : T_n \le T_N + 2r\}.
$$ 

\begin{lemma}\label{lem: estimates on annuli like sets} 
There exists $C_1>1$ such that: if $N \ge 1$, then 
$$
\sum_{n=1}^{N^\prime} e^{-\delta \phi(\kappa_\theta(\rho(\gamma_n)))}\le C_1 \sum_{n=1}^N e^{-\delta \phi(\kappa_\theta(\rho(\gamma_n)))}.
$$
\end{lemma} 

\begin{proof} Note that if $T_n, T_m \in [T_N, T_N+2r]$ and  $\Oc_r(b_0, \gamma_n(b_0)) \cap \Oc_r(b_0, \gamma_m(b_0)) \neq \emptyset$, then 
$$
\d_\Omega( b_0, \gamma_n^{-1}\gamma_m(b_0))=\d_\Omega( \gamma_n(b_0), \gamma_m(b_0)) \le 6r.
$$
Thus, if we set
$$
M : =  \# \{ \gamma \in \Gamma_0 : \d_\Omega(b_0, \gamma(b_0)) \le 6r\},
$$
then every point in $\partial\Omega$ lies in at most $M$ different sets of the form $\Oc_r(b_0, \gamma_n(b_0))$ such that $T_n\in  [T_N, T_N+2r]$. This implies that
$$
\sum_{n=N+1}^{N^\prime} \nu\left(A_n \right)=\sum_{n=N+1}^{N^\prime} \nu\left( \Oc_r(b_0, \gamma_n(b_0))  \right) \le M \nu(\partial \Omega) = M. 
$$
Then by Equation~\eqref{eqn:shadow lemma in proof using BC lemma},
\begin{align*}
\sum_{n=1}^{N^\prime}& e^{-\delta \phi(\kappa_\theta(\rho(\gamma_n)))}=\sum_{n=1}^{N} e^{-\delta \phi(\kappa_\theta(\rho(\gamma_n)))}+\sum_{n=N+1}^{N^\prime} e^{-\delta \phi(\kappa_\theta(\rho(\gamma_n)))} \le \sum_{n=1}^{N} e^{-\delta \phi(\kappa_\theta(\rho(\gamma_n)))}+CM \\
&\le  \left(1+\frac{C M }{e^{-\delta \phi(\kappa_\theta(\rho(\gamma_1)))}} \right) \sum_{n=1}^N e^{-\delta \phi(\kappa_\theta(\rho(\gamma_n)))}
\end{align*}
for all $N \ge 1$. The lemma now holds with $C_1:=1+\frac{C M }{e^{-\delta \phi(\kappa_\theta(\rho(\gamma_1)))}}$.
\end{proof} 

The next lemma verifies that the sequence $\{A_n\}$ satisfy the second hypothesis of Lemma~\ref{lem: KS BC lemma}.

\begin{lemma}\label{lem:the 2nd main estimate in the divergence proof} There exists $C_2 > 0$ such that: if $N \ge 1$, then 
\begin{align*}
\sum_{1 \le n,m \le N} & \nu(A_n \cap A_m) \le C_2 \left(\sum_{n=1}^N \nu(A_n)\right)^2.
\end{align*}

\end{lemma} 

\begin{proof}  Let 
$$
\Delta_N:=\{(m,n):1 \le n \le m \le N \text{ and } A_m\cap A_n\neq\emptyset\}.
$$
One can show (see the proof of Lemma \ref{lem: claim}) that if $(m,n)\in\Delta_N$, then 
\begin{align*}
\d_\Omega( \gamma_n(b_0), [b_0, \gamma_m(b_0)]_\Omega) \le 2r.
\end{align*}
Then Proposition~\ref{prop:multiplicative along geodesics} implies
\begin{equation}
\label{eqn: almost multiplicative on the triangle set}
\sup_{(m,n)\in\Delta_N} \norm{\kappa_\theta(\rho(\gamma_n))+\kappa_\theta(\rho(\gamma_n^{-1}\gamma_m))-\kappa_\theta(\rho(\gamma_m))} < +\infty,
\end{equation} 
and so by Equation~\eqref{eqn:shadow lemma in proof using BC lemma}, there exists a constant $C^\prime > 0$ such that
\begin{align*}
\nu(A_n \cap A_m) & \le \nu(\Oc_r(b_0, \gamma_m(b_0))) \le C^\prime e^{-\delta \phi(\kappa_\theta(\rho(\gamma_n)))}e^{-\delta \phi(\kappa_\theta(\rho(\gamma_n^{-1}\gamma_m)))}
\end{align*}
for all $(m,n)\in\Delta_N$. Also,
\begin{align*}
\d_\Omega(b_0, \gamma_n^{-1}\gamma_m(b_0)) & = \d_\Omega(\gamma_n(b_0), \gamma_m(b_0)) \le \d_\Omega(\gamma_n(b_0),[b_0, \gamma_m(b_0)]_\Omega ) + \d_\Omega(b_0, \gamma_m(b_0))\\
& \le 2r+T_m \le 2r + T_N
\end{align*}
for all $(m,n)\in\Delta_N$. In particular, if $(m,n)\in\Delta_N$, then $\gamma_n^{-1}\gamma_m = \gamma_k$ for some $k \le N'$. 

These observations, Lemma~\ref{lem: estimates on annuli like sets} and Equation~\eqref{eqn:shadow lemma in proof using BC lemma} imply that if $N\ge 1$, then
\begin{align*}
\sum_{1 \le n,m \le N}  & \nu(A_n \cap A_m)  \le 2 \sum_{(m,n) \in \Delta_N}\nu(A_n \cap A_m)   \le 2C^\prime\sum_{(m,n) \in \Delta_N} e^{-\delta \phi(\kappa_\theta(\rho(\gamma_n)))}e^{-\delta \phi(\kappa_\theta(\rho(\gamma_n^{-1}\gamma_m)))}\\
&  \le 2C^\prime \sum_{k=1}^{N^\prime} \sum_{n=1}^N e^{-\delta \phi(\kappa_\theta(\rho(\gamma_n)))}e^{-\delta \phi(\kappa_\theta(\rho(\gamma_k)))} \le 2C^\prime C_1\left(\sum_{n=1}^N e^{-\delta \phi(\kappa_\theta(\rho(\gamma_n)))}\right)^2 \\
& \le 2C^\prime C_1C^2 \left(\sum_{n=1}^N \nu(A_n)\right)^2. \qedhere
\end{align*}
\end{proof} 

We may now apply Lemma \ref{lem: KS BC lemma} to the finite measure space $\left(\partial\Omega, \nu\right)$ and the sequence $\{A_n\}$ to finish the proof of Proposition \ref{prop: support on conical limit set}.

\begin{proof}[Proof of Proposition \ref{prop: support on conical limit set}]
We first show that $\mu(\Lambda_{\theta}^{\rm con}(\Gamma))>0$. By Lemma \ref{lem: KS BC lemma}, if we set 
$$
Y:=\{ x \in \partial\Omega : x \text{ is in infinitely many of } A_1,A_2,\dots\},
$$
then $\nu(Y) > 0$. Notice that if $x\in  Y$,  then there is a sequence $\{\gamma_n\}$ in $\Gamma_0$ such that $\gamma_n(b_0) \rightarrow x$ and 
$$
 \d_\Omega( \gamma_n(b_0), [b_0, x) ) < r 
$$
for all $n \ge 1$. Thus $Y \subset \Lambda_\Omega^{\rm con}(\Gamma_0)$. By Observation~\ref{obs: basic properties of transverse repns}(1), $\xi(Y)\subset\Lambda_{\theta}^{\rm con}(\Gamma)$, so 
\[\mu(\Lambda_{\theta}^{\rm con}(\Gamma))\ge\mu(\xi(Y))=\nu(Y)>0.\]

Now suppose for contradiction that $\mu(\Lambda_{\theta}^{\rm con}(\Gamma))<1$. If we set $S:=\Lambda_\theta(\Gamma)-\Lambda_{\theta}^{\rm con}(\Gamma)$, then $\mu(S)>0$, so we may define a probability measure $\mu_S$ on $\Lambda_\theta(\Gamma)$ by
\[
\mu_S(A):=\frac{1}{\mu(S)}\mu(A\cap S).
\]
By definition, $\mu_S(\Lambda_\theta(\Gamma))=0$.
On the other hand, since $S$ is $\Gamma$-invariant, $\mu_S$ is a $\phi$-Patterson-Sullivan measure for $\Gamma$ of dimension $\delta$, so the above argument implies that  $\mu_S(\Lambda_{\theta}^{\rm con}(\Gamma))>0$, which is  a contradiction. Therefore, $\mu(\Lambda_{\theta}^{\rm con}(\Gamma))=1$.

By Proposition~\ref{prop: consequences of shadow lemma}, $\mu$ has no atoms in $\Lambda_{\theta}^{\rm con}(\Gamma)$. Since $\mu(\Lambda_{\theta}^{\rm con}(\Gamma))=1$, we conclude that $\mu$ has no atoms.
\end{proof}

\section{Non-ergodicity of the flow in the convergent case}

In this section, we prove that the geodesic flow of a transverse representation is dissipative and non-ergodic if its image is in the convergent case of our Hopf-Sullivan-Tsuji dichotomy.

\begin{proposition}\label{prop: dissipative}
Let $\Omega\subset\Pb(\Rb^{d})$ be a properly convex domain, let $\Gamma_0\subset\Aut(\Omega)$ be a non-elementary projectively visible subgroup and let $\rho:\Gamma_0\to\GG$ be a $\Psf_\theta$-transverse representation for some symmetric $\theta\subset\Delta$. Suppose $\phi\in\mfa_\theta^*$ satisfies $\delta:=\delta^\phi(\rho(\Gamma_0))<+\infty$. Let $\mu$ and $\bar{\mu}$ respectively be $\phi$ and $\bar{\phi}$-Patterson-Sullivan measures for $\rho(\Gamma_0)$ of dimension $\beta$, and let $m$ be the Bowen-Margulis measure on $\wh\Usf(\Gamma_0)$ associated to $\rho$, $\mu$ and $\bar\mu$.
If $Q^\phi_\Gamma(\delta)<+\infty$, then 
\begin{enumerate}
\item the $\Gamma_0\times\Rb$-action on $\left(\Usf(\Gamma_0),\widetilde{m}\right)$ is dissipative,
\item the action of the geodesic flow on $\left(\wh\Usf(\Gamma_0),m\right)$ is dissipative, and
\item the action of the geodesic flow on $\left(\wh\Usf(\Gamma_0),m\right)$ is non-ergodic. 
\end{enumerate}
\end{proposition}

Before proving Proposition \ref{prop: dissipative}, we briefly discuss the notions of dissipative and conservative dynamical systems. Suppose that $X$ is a standard Borel space, $H$ is a locally compact, second countable, unimodular group that acts measurably on $X$, ${\rm d}h$ is a Haar measure on $H$, and $m$ a $H$-quasi-invariant, $\sigma$-finite measure on $X$. If $A\subset X$ has positive $m$-measure, we say that $A\subset X$ is \emph{wandering} if for $m$-almost every $x\in A$, $\int_{h\in H}\mathbf{1}_A(h(x)) {\rm d}h<+\infty$. Then  let $\mathcal D\subset\Omega$ be the union of all wandering sets, and let $\mathcal C:=\Omega-\mathcal D$. We say that $H$-action on $(X,m)$ is \emph{conservative} (resp. \emph{dissipative}) if $m(\mathcal D)=0$ (resp. $m(\mathcal C)=0$).

Given a $m$-integrable, positive function $f:X\to (0,\infty)$, we may decompose $X$ into
\[\mathcal C_f:=\left\{x\in X:\int_H f(h(x)){\rm d}h=+\infty\right\}\quad\text{ and }\quad\mathcal D_f:=\left\{x\in X:\int_H f(h( x)){\rm d}h<+\infty\right\}.\]
In the case when the measure $m$ is $H$-invariant, it is known (see for instance~\cite[Fact 2.27]{blayac}) that $\Cc_f=\Cc$ and $\mathcal D_f=\mathcal D$ up to measure zero sets. 

\begin{proof}[Proof of Proposition \ref{prop: dissipative}] 
Proof of (1). Suppose for contradiction that the $\Gamma_0\times\Rb$-action on $\left(\Usf(\Gamma_0),\widetilde{m}\right)$ is not dissipative. Then there is a $\widetilde{m}$-integrable, positive function $f:\Usf(\Gamma_0)\to (0,\infty)$ and a compact set 
\[K\subset\mathcal C_f:=\left\{v\in \Usf(\Gamma_0):\sum_{\gamma\in\Gamma_0}\int_{\Rb} f(\gamma\cdot \varphi_t(v)){\rm d}t=+\infty\right\}\] 
such that $\widetilde{m}(K)>0$. For any $R>0$, let 
\[K_R:=\left\{v\in K:\sum_{\gamma\in\Gamma_0}\int_{\Rb}1_K(\gamma\cdot \varphi_t(v)){\rm d}t\le R\right\}.\]
Since $K_R\subset\mathcal C_f$, the integral
\[\int_{\Usf(\Gamma_0)}\sum_{\gamma\in\Gamma_0}\int_{\Rb}f(\gamma\cdot\varphi_t(v))1_{K_R}(v)\d t\,\d\widetilde{m}(v)=\int_{K_R}\sum_{\gamma\in\Gamma_0}\int_{\Rb}f(\gamma\cdot\varphi_t(v))\d t\,\d\widetilde{m}(v).\]
is infinite if $\widetilde{m}(K_R)>0$. On the other hand, since $\widetilde{m}$ is $\Gamma_0\times\Rb$-invariant, 
\begin{align*}
\int_{\Usf(\Gamma_0)}\sum_{\gamma\in\Gamma_0}\int_{\Rb}f(\gamma\cdot\varphi_t(v))1_{K_R}(v)\d t\,\d\widetilde{m}(v)&=\int_{\Usf(\Gamma_0)}f(v)\sum_{\gamma\in\Gamma_0}\int_{\Rb}1_{K_R}(\gamma\cdot \varphi_t(v))\d t\,\d\widetilde{m}(v)\\
&\le R\int_{\Usf(\Gamma_0)}f(v)\d\widetilde{m}(v)<+\infty.
\end{align*}
It follows that $\widetilde{m}(K_R)=0$ for all $R>0$, or equivalently, that 
\[\sum_{\gamma\in\Gamma_0}\int_{\Rb}1_K(\gamma\cdot \varphi_t(v)){\rm d}t=+\infty\] 
for $\widetilde{m}$-almost every $v\in K$. This in turn implies that for almost every $v\in K$, there are diverging sequences $\{t_n\}$ in $\Rb$ and $\{\gamma_n\}$ in $\Gamma_0$ such that $\gamma_n\varphi_{t_n}(v)\in K$, and so at least one of the forward endpoint $v^+$ or backward endpoint $v^-$ of $v$ is in $\Lambda_\Omega^{\rm con}(\Gamma_0)$. Thus, 
$$
\mu(\xi(\Lambda_\Omega^{\rm con}(\Gamma_0)))>0,
$$
since $\widetilde{m}(K)>0$. However, by Proposition \ref{prop: consequences of shadow lemma}, $\mu(\xi(\Lambda_\Omega^{\rm con}(\Gamma_0)))=0$, which is a contradiction.

Proof of (2). Let $f:\Usf(\Gamma_0)\to(0,\infty)$ be a $\widetilde{m}$-integrable, positive function. By part (1), we may define an $m$-integrable, positive function $F:\wh{\Usf}(\Gamma_0)\to(0,\infty)$ by $F([v]):=\sum_{\gamma\in\Gamma_0}f(\gamma\cdot v)$. Furthermore, for $m$-almost every $[v]\in\wh{\Usf}(\Gamma_0)$, 
\[\int_{\Rb}F(\varphi_t([v]))\d t=\sum_{\gamma\in\Gamma_0}\int_{\Rb}f(\gamma\cdot \varphi_t(v))\d t<+\infty.\]

Proof of (3). Pick a compact set $K\subset\wh{\Usf}(\Gamma_0)$ with non-empty interor. Let $f:\widehat\Usf(\Gamma_0)\to(0,\infty)$ be a $m$-integrable, positive function that takes the value $1$ on the compact set $\varphi_{[0,1]}(K)$. Part (2) implies that for $m$-almost every $v\in\wh{\Usf}(\Gamma_0)$, we have
\[\int_{\Rb}f(\varphi_t(v)){\rm d}t<+\infty,\] 
so there is some $T_v>0$ such that $\varphi_t(v)\notin K$ for all $t\notin [-T_v,T_v]$. 

Suppose for contradiction that the action of the geodesic flow on $\left(\wh\Usf(\Gamma_0),m\right)$ is ergodic. Then for $m$-almost every $v\in \wh\Usf(\Gamma_0)$, 
the flow line of $v$ is dense in $\wh\Usf(\Gamma_0)$. Thus, there is some $v_0\in\wh\Usf(\Gamma_0)$ and some $T:=T_{v_0}>0$ such that $\wh\Usf(\Gamma_0)= \overline{\varphi_{\Rb}(v_0)}$ and 
$$
\varphi_{(-\infty,-T)}(v_0)\cup\varphi_{(T,\infty)}(v_0)\subset\wh{\Usf}(\Gamma_0)-K.
$$
It follows that the interior
$K^0$ of $K$ lies in $\varphi_{[-T,T]}(v_0)$. However, it is easy to see that no open set in $\wh{\Usf}(\Gamma_0)$ is homeomorphic to a subset of the interior of a line segment.
\end{proof}

\section{Ergodicity of the flow in the divergent case}

In this section, we prove that the geodesic flow of a transverse representation is conservative and ergodic if its image is in the divergent case of our Hopf-Sullivan-Tsuji dichotomy.

\begin{theorem}\label{thm: ergodicity}
Let $\Omega\subset\Pb(\Rb^{d})$ be a properly convex domain, let $\Gamma_0\subset\Aut(\Omega)$ be a non-elementary projectively visible subgroup and let $\rho:\Gamma_0\to\GG$ be a $\Psf_\theta$-transverse representation for some symmetric $\theta\subset\Delta$. Let $\phi\in\mfa_\theta^*$, let $\mu$ and $\bar{\mu}$ respectively be $\phi$ and $\bar{\phi}$-Patterson-Sullivan measures for $\rho(\Gamma_0)$ of dimension $\delta:=\delta^\phi(\rho(\Gamma_0))$, let $m$ be the Bowen-Margulis measure on $\wh\Usf(\Gamma_0)$ associated to $\rho$, $\mu$ and $\bar\mu$, and let $\widetilde{m}$ be the lift of $m$ to $\Usf(\Gamma_0)$. If $Q_{\rho(\Gamma_0)}^\phi(\delta)=+\infty$, then 
\begin{enumerate}
\item  the action of the geodesic flow on $\left(\wh{\Usf}(\Gamma_0),m\right)$ is conservative, and
\item the action of the geodesic flow on $\left(\wh\Usf(\Gamma_0),m\right)$ is ergodic.
\end{enumerate} 
\end{theorem}

Before starting the proof of Theorem \ref{thm: ergodicity}, we recall a result of Coud\`ene. Suppose $\{ \varphi_t\}$ is a continuous flow on a metric space $X$ which preserves a Borel measure $m$. The \emph{strong stable manifold} of $x \in X$ is 
$$
W^{ss}(x) := \left\{ y \in X : \lim_{t \rightarrow \infty} \d(\varphi_t(x), \varphi_t(y)) = 0 \right\}
$$
and the \emph{strong unstable manifold} of $x \in X$ is 
$$
W^{su}(x) := \left\{ y \in X : \lim_{t \rightarrow -\infty} \d(\varphi_t(x), \varphi_t(y)) = 0 \right\}. 
$$
A measurable function $f : X \rightarrow \Rb$ is \emph{$W^{ss}$-invariant} if there exists a full measure set $X^\prime \subset X$ where $f(x) = f(y)$ whenever $x,y \in X^\prime$ and $y \in W^{ss}(x)$. Similarly, a measurable function $f : X \rightarrow \Rb$ is \emph{$W^{su}$-invariant} if there exists a full measure set $X^\prime \subset X$ where $f(x) = f(y)$ whenever $x,y \in X^\prime$ and $y \in W^{su}(x)$.

\begin{theorem}[{Coud\`ene~\cite{coudeneHopf}}]\label{thm: Coudene}
Let $X$ be a metric space, $\{\varphi_t\}$ a continuous flow on $X$ and $m$ a $\{\varphi_t\}$-invariant Borel measure on $X$ such that $(X,m,\{\varphi_t\})$ is conservative. 
Suppose  that there is a full measure subset of $X$ that is covered by a countable family of open sets with finite m-measure. Then every flow-invariant, $m$-measurable function on $X$ is $W^{ss}$-invariant and $W^{su}$-invariant. 
\end{theorem}

\begin{proof}[Proof of Theorem \ref{thm: ergodicity}] 
Proof of (1). Fix a $m$-integrable, positive, continuous function $f:\wh\Usf(\Gamma_0)\to (0,\infty)$. Then let 
\[\widetilde{f}:=f\circ p:\Usf(\Gamma_0)\to\Rb,\] 
where $p:\Usf(\Gamma_0)\to\wh \Usf(\Gamma_0) $ is the quotient map. To show that the action of the geodesic flow on $\left(\wh{\Usf}(\Gamma_0),m\right)$ is conservative, it suffices to show that 
\[\int_{\Rb} \widetilde f(\varphi_t(v))\d t=\int_{\Rb} f(\varphi_t(p(v)))\d t\]
is infinite for $\tilde m$-almost every $v \in \Usf(\Gamma_0)$. 

By Proposition \ref{prop: support on conical limit set} and Observation~\ref{obs: basic properties of transverse repns}(1), the set 
$$
\Rc: = \left\{ v \in  \Usf(\Gamma_0) : v^+ \in \Lambda_\Omega^{\rm con}(\Gamma_0)\right\}
$$
has full $\widetilde{m}$-measure. 

Fix $v \in \Rc$. Then there is some $r>0$ and a sequence $\{\gamma_n\}$ in $\Gamma_0$ such that $\gamma_n(\pi(v))\to v^+$ and 
$$
\d_\Omega\big(\gamma_n(\pi(v)),[\pi(v),v^+)_\Omega\big)<r
$$
for all $n$, where $\pi:T^1\Omega\to\Omega$ is the projection map. In particular, there exists a compact subset $K \subset \Usf(\Gamma_0)$, which depends on $r > 0$, such that 
$$
\{ t \in \Rb : \varphi_t(v) \in \Gamma_0 \cdot K\}
$$
has infinite Lebesgue measure. Since $\tilde f$ is $\Gamma_0$-invariant and continuous, 
$$
\inf_{w \in \Gamma_0 \cdot K} f(w) = \min_{w \in K} f(w) > 0. 
$$
Hence 
\[\int_{\Rb} \widetilde f(\varphi_t(v))\d t=+\infty.\]
Since $v \in \Rc$ was arbitrary and $\Rc$ has full $\widetilde{m}$-measure,
\[\int_{\Rb} \widetilde f(\varphi_t(v))\d t=+\infty\]
for $\tilde m$-almost every $v \in \Usf(\Gamma_0)$. 

Proof of (2). 
Notice that Equation~\eqref{eqn:the projection between T1's is Lipschitz} implies that
$$
p(W^{ss}(v)) \subset W^{ss}(p(v)) \quad \text{and} \quad p(W^{su}(v)) \subset W^{su}(p(v))
$$
for all $v \in \Usf(\Gamma_0)$, so the lift of a $W^{ss}$-invariant (respectively $W^{su}$-invariant) function on $\widehat{\Usf}(\Gamma_0)$ is a $W^{ss}$-invariant (respectively $W^{su}$-invariant) function on $\Usf(\Gamma_0)$. By definition, $\wh\Usf(\Gamma_0)$ is  covered by a countable family of open sets of with finite $m$-measure, so by Theorem \ref{thm: Coudene},  it suffices to show that if $f:\Usf(\Gamma_0)\to\mathbb R$ is a $\tilde{m}$-measurable, $\Gamma$-invariant, 
$\{\varphi_t\}$-invariant, $W^{ss}$-invariant and $W^{su}$-invariant function, then $f$ is constant on a set of full $\tilde{m}$-measure. 

Since $f$ is $W^{ss}$-invariant and $W^{su}$-invariant, by definition there exists a full $\tilde{m}$-measure set $Y_0 \subset \Usf(\Gamma_0)$ such that $f(v) = f(w)$ whenever $v,w \in Y_0$ and  $v \in W^{ss}(w) \cup W^{su}(w)$. Since $f$ is $\{\varphi_t\}$-invariant, we can assume that $Y_0$ is also $\{\varphi_t\}$-invariant. Let $\nu$ and $\bar\nu$ be measures on $\partial \Omega$ given by 
\begin{equation*}
\nu(A) = \mu\left( \xi(A \cap \Lambda_\Omega(\Gamma_0))\right)\quad \text{and}\quad\bar\nu(A) = \bar\mu\left( \xi(A \cap \Lambda_\Omega(\Gamma_0))\right),
\end{equation*}
where $\xi$ is the limit map of $\rho$. By the definition of $\tilde m$, we see that $Y_0 = Y_0^\prime \times \Rb$ for some set $Y_0^\prime \subset \Lambda_\Omega(\Gamma_0)^{(2)}$ of full $\bar{\nu} \otimes \nu$-measure. Set
\begin{align*}
X^+& : = \{ y \in \Lambda_\Omega(\Gamma_0) : (x,y) \in Y_0^\prime \text{ for $\bar{\nu}$-almost every $x\in\Lambda_\Omega(\Gamma_0)$} \},
\end{align*}
and note that $\nu(X^+) = 1$ by Fubini's theorem. Hence, if we fix $(v_0^-, v_0^+) \in (\Lambda_\Omega(\Gamma_0) \times X^+) \cap Y_0^\prime$, then the set
$$
Y^\prime :=\left\{ (x,y)\in Y_0' : (x,v_0^+) \in Y_0^\prime \right\}
$$
has full $\bar{\nu} \otimes \nu$-measure, so $Y:=Y'\times\Rb\subset \Usf(\Gamma_0)$ has full $\tilde{m}$-measure.

Let $(x,y,t)\in Y$. By Lemma \ref{asymptotic at smooth point}, there is some $s\in\Rb$ such that $(x,y,t)\in W^{su}(x,v_0^+,s)$, and there is some $r\in\Rb$ such that $(x,v_0^+,s)\in W^{ss}(v_0^-,v_0^+,r)$. By definition, $(x,y,t)$, $(x,v_0^+,s)$, and $(v_0^-,v_0^+,r)$ lie in $Y_0$, so 
\[f(x,y,t)=f(x,v_0^+,s)=f(v_0^-,v_0^+,r)=f(v_0^-,v_0^+,0).\]
This proves that $f$ is constant on $Y$.\end{proof}

\section{Consequences of ergodicity}\label{sec:consequences of ergodicity} 

In this section we record some consequences of Theorem \ref{thm: ergodicity}.
The first two corollaries complete the proof of our Hopf-Sullivan-Tsuji dichotomy. We also show, in the divergent case, that there is some $R$ so that the uniformly $R$-conical limit
set has full measure for the unique Patterson-Sullivan measure of critical dimension and establish a rigidity result for pairs of transverse representations with mutually
non-singular BMS measures.

\begin{corollary}\label{cor: ergodicity}
Suppose $\Gamma\subset\GG$ is a non-elementary $\Psf_\theta$-transverse subgroup for some symmetric $\theta\subset\Delta$, $\phi\in\mfa_\theta^*$ and $\delta:=\delta^\phi(\Gamma)<+\infty$. Let $\mu$ and $\bar{\mu}$ respectively be $\phi$ and $\bar{\phi}$-Patterson-Sullivan measures for $\Gamma$ of dimension $\beta$. 
\begin{enumerate}
\item If $Q_\Gamma^\phi(\delta)=+\infty$ and $\beta=\delta$, then the $\Gamma$-action on $(\Lambda_\theta(\Gamma)^{(2)},\bar{\mu}\otimes\mu)$ is conservative, and the $\Gamma$ actions on
 $(\Lambda_\theta(\Gamma)^{(2)},\bar{\mu}\otimes\mu)$ and on $(\Lambda_\theta(\Gamma),\mu)$ are ergodic.
 \item If $Q_\Gamma^\phi(\delta)<+\infty$, then the $\Gamma$ action on 
 $(\Lambda_\theta(\Gamma)^{(2)},\bar{\mu}\otimes\mu)$ is dissipative and non-ergodic.
 \end{enumerate}
\end{corollary}

\begin{proof} 
Using the discussion in Section~\ref{sec: a helpful reduction} we may assume that $\GG$ has trivial center and that $\Psf_\theta$ does not contain any simple factors of $\GG$. Then by Theorem \ref{thm:transverse image of visible general}, there is a properly convex domain $\Omega \subset \Pb(\Rb^{d})$, a projectively visible subgroup $\Gamma_0 \subset \Aut(\Omega)$
and a $\Psf_\theta$-transverse representation $\rho : \Gamma_0 \rightarrow \GG$ such that $\rho(\Gamma_0)=\Gamma$. Let $\xi:\Lambda_\Omega(\Gamma_0)\to\Lambda_\theta(\Gamma)$ 
be the $\rho$-equivariant boundary map and let $m$ be the Bowen-Margulis measure on $\wh\Usf(\Gamma_0)$ associated to $\rho$, $\mu$ and $\bar{\mu}$. 

Proof of (1). Theorem \ref{thm: ergodicity} part (3) implies that the geodesic flow on $(\widehat \Usf(\Gamma_0),m)$ is ergodic. Any $\Gamma$-invariant subset of either  $(\Lambda_\theta(\Gamma)^{(2)},\bar{\mu}\otimes\mu)$ or $(\Lambda_\theta(\Gamma),\mu)$ that has positive but not full
measure, gives rise to a flow-invariant subset of $(\widehat \Usf(\Gamma_0),m)$ that has positive but not full measure. Therefore,
the  actions of $\Gamma$ on  $(\Lambda_\theta(\Gamma)^{(2)},\bar{\mu}\otimes\mu)$ and $(\Lambda_\theta(\Gamma),\mu)$ are both ergodic.

Next, suppose for the purpose of contradiction that the action of $\Gamma$ on $(\Lambda_\theta(\Gamma)^{(2)},\bar{\mu}\otimes\mu)$ is not conservative. Since the measure $e^{-\delta\phi([\cdot,\cdot]_\theta)}\bar{\mu}\otimes\mu$ on $\Lambda_\theta(\Gamma)^{(2)}$ is $\Gamma$-invariant, there is a positive, continuous function $f:\Lambda_\theta(\Gamma)^{(2)}\to (0,\infty)$ such that 
\[\mathcal D_f:=\left\{(x,y)\in\Lambda_\theta(\Gamma)^{(2)}:\sum_{\gamma\in\Gamma}f(\gamma\cdot (x,y))<+\infty\right\}\]
has positive $\bar{\mu}\otimes\mu$-measure. Since $\Gamma$ acts minimally on $\Lambda_\theta(\Gamma)$, each open set in $\Lambda_\theta(\Gamma)^{(2)}$ has positive $\bar{\mu}\otimes\mu$-measure. This, together with the fact that the action of $\Gamma$ on $(\Lambda_\theta(\Gamma)^{(2)},\bar{\mu}\otimes\mu)$ is ergodic, implies that almost every orbit is dense. Thus, there exists $(x_0, y_0) \in \mathcal D_f$ with $\Lambda_\theta(\Gamma)^{(2)} = \overline{\Gamma \cdot (x_0,y_0)}$, but this is a contradiction since $f$ is positive and 
$$
\sum_{\gamma\in\Gamma}f(\gamma\cdot (x_0,y_0)) < +\infty. 
$$

Proof of (2). 
Let $f:\Usf(\Gamma_0)\to (0,\infty)$ be a $\widetilde{m}$-integrable, positive function. By Proposition~\ref{prop: dissipative} part (1), we may define the $\bar\mu\otimes\mu$-integrable, positive function
\[F:\Lambda_\theta(\Gamma)^{(2)}\to\Rb\quad\text{by}\quad F(\xi(v^-),\xi(v^+)):=\int_{\Rb}f(\varphi_t (v)).\] 
Furthermore, for $\bar\mu\otimes\mu$-almost every $(\xi(v^-),\xi(v^+))\in\Lambda_\theta(\Gamma)^{(2)}$, we have
\[\sum_{\rho(\gamma)\in\Gamma}F(\rho(\gamma)\xi(v^-),\rho(\gamma)\xi(v^+))=\sum_{\gamma\in\Gamma_0}\int_{\Rb}f(\gamma(\varphi_t(v)))\d t<+\infty.\]
It follows that the $\Gamma$ action on $(\Lambda_\theta(\Gamma)^{(2)},\bar{\mu}\otimes\mu)$ is dissipative.

Proposition \ref{prop: dissipative} part (3) implies that there is some subset of $\Usf(\Gamma_0)$ that is invariant under the $\Gamma_0\times\Rb$-action, with positive but not full $\widetilde{m}$-measure. This defines a subset of $\Lambda_\theta(\Gamma)^{(2)}$ that is invariant under the $\Gamma$-action, with positive but not full $\bar\mu\otimes\mu$-measure. Thus, the $\Gamma$ action on $(\Lambda_\theta(\Gamma)^{(2)},\bar{\mu}\otimes\mu)$ is non-ergodic.
\end{proof}

It follows, 
from a standard argument (see for instance~\cite[pg.\ 181]{sullivan-density}), that the Patterson-Sullivan measure  in the critical dimension is unique
 in the divergent case.

\begin{corollary}\label{cor:uniqueness} Suppose $\Gamma\subset\GG$ is a non-elementary $\Psf_\theta$-transverse subgroup for some symmetric $\theta \subset \Delta$, $\phi\in\mfa_\theta^*$ and $\delta:=\delta^\phi(\Gamma)<+\infty$. If $Q_\Gamma^\phi(\delta)=+\infty$, then there is a unique $\phi$-Patterson-Sullivan measure $\mu_\phi$ for $\Gamma$ of dimension $\delta$. 
 \end{corollary}

For the next two results let $\Omega\subset\Pb(\Rb^{d})$ be a properly convex domain, let $\Gamma_0 \subset \Aut(\Omega)$ be a projectively visible subgroup and let $b_0 \in \Omega$. For any $R > 0$, we denote by $\Lambda_{\Omega, b_0,R}^{\rm con}(\Gamma_0)$ the set of points $x \in \Lambda_\Omega^{\rm con}(\Gamma_0)$ for which there exists a sequence
$\{\gamma_n\}$ in $\Gamma_0$ such that $\gamma_n(b_0) \rightarrow x$ and 
$$
\d_\Omega(\gamma_n(b_0), [b_0, x)_\Omega) < R
$$
for all $n$. The next corollary proves that if the image of a transverse representation is in the divergent case, then there is an $R> 0$ such that the set of $R$-conical limit points have full measure. 

\begin{corollary}\label{cor:uniformly conical set has full measure} Suppose $\rho: \Gamma_0 \rightarrow \GG$ is a $\mathsf{P}_{\theta}$-transverse representation for some symmetric $\theta \subset \Delta$, 
$\phi \in \mfa_{\theta}^*$, $\delta:=\delta^{\phi}(\rho(\Gamma_0)) <+\infty$ and $\mu$ is the $\phi$-Patterson-Sullivan measure  for $\rho(\Gamma_0)$ of dimension $\delta$. If $Q_{\rho(\Gamma_0)}^\phi(\delta)=+\infty$, then for any $b_0 \in \Omega$, there exists $R > 0$ such that 
$$
\mu\left( \xi(\Lambda_{\Omega, b_0,R}^{\rm con}(\Gamma_0) ) \right) = 1.
$$
\end{corollary}

\begin{proof}
The following argument is standard, see for instance~\cite[pg.\ 190]{sullivan-density}. Define a measure $\nu$ on $\partial \Omega$ by 
$$
\nu(A) = \mu\big( \xi(A \cap \Lambda_\Omega(\Gamma_0))\big).
$$
Since $Q_{\rho(\Gamma_0)}^\phi(\delta)=+\infty$, by Proposition \ref{prop: support on conical limit set},
$$
1 = \nu\left( \Lambda_{\Omega}^{\rm con}(\Gamma_0) \right) = \lim_{n \rightarrow \infty}  \nu\left( \Lambda_{\Omega, b_0, n}^{\rm con}(\Gamma_0)\right).
$$
Hence there exists $R_0 > 0$ such that $\nu\left( \Lambda_{\Omega, b_0, R_0}^{\rm con}(\Gamma_0)\right) > 0$. 

Let $L $ be the set of points $x \in \Lambda_\Omega(\Gamma_0)$ for which there exist $b \in \Gamma_0(b_0)$ and a sequence
$\{\gamma_n\}$ in $\Gamma_0$ such that $\gamma_n(b) \rightarrow x$ and 
$$
\d_\Omega(\gamma_n(b), [b, x)_\Omega) \le R_0
$$
for all $n$. Observe that $L$ is $\Gamma_0$-invariant, and $\nu(L)>0$ because $\Lambda_{\Omega, b_0, R_0}^{\rm con}(\Gamma_0)\subset L$. Hence by Corollary \ref{cor: ergodicity}, $\nu(L) = 1$. 

It now suffices to show that $L \subset \Lambda_{\Omega, b_0, R_0+1}^{\rm con}(\Gamma_0)$. Fix $x \in L$. Then there exist $b \in \Gamma_0(b_0)$, a sequence
$\{\gamma_n\}$ in $\Gamma_0$, and a sequence $\{b_n\}$ in $[b,x)_\Omega$ where $\gamma_n(b) \rightarrow x$ and 
$$
\d_\Omega(\gamma_n(b),b_n) \le R_0
$$
for all $n$. By Lemma \ref{asymptotic at smooth point}, there exists a sequence $\{b_n^\prime\}$ in $[b_0, x)_\Omega$ such that 
$$
\lim_{n \rightarrow \infty} \d_\Omega(b_n, b_n^\prime) =0. 
$$
Since $b \in \Gamma_0(b_0)$, we can write $b = \gamma(b_0)$ for some $\gamma \in \Gamma_0$. Then 
$$
\d_\Omega(\gamma_n\gamma(b_0), [b_0, x)_\Omega) \le R_0+1
$$
for all $n$ sufficiently large. So $x \in \Lambda_{\Omega, b_0, R_0+1}^{\rm con}(\Gamma_0)$.
\end{proof}

Finally, we prove the following rigidity result for length functions which have non-singular Bowen-Margulis-Sullivan measures.

\begin{corollary}\label{cor:non-singular BM measures} For $j=1,2$, suppose $\rho_j: \Gamma_0 \rightarrow \GG_j$ is a $\mathsf{P}_{\theta_j}$-transverse representation for some symmetric $\theta_j \subset \Delta_j$, $\phi_j \in \mfa_{\theta_j}^*$ and $\delta_j:=\delta^{\phi_j}(\rho_j(\Gamma_0)) <+\infty$. For $\psi\in\{\phi_j,\bar\phi_j\}$, let $\mu_{\psi}$ be the $\psi$-Patterson-Sullivan measure for $\rho_j(\Gamma_0)$ of dimension $\delta_j$ and let $m_j$ denote the Bowen-Margulis-Sullivan measure associated to 
$\rho_j$, $\mu_{\phi_j}$ and $\mu_{\bar\phi_j}$. If $Q_{\rho_j(\Gamma_0)}^{\phi_j}(\delta_j)=+\infty$ for $j=1,2$ and $m_1$ is non-singular with respect to $m_2$, then:
\begin{enumerate}[itemsep=0.3cm]
\item  $m_1=cm_2$ for some $c > 0$.
\item\label{magnitude rigidity}  $\sup_{\gamma \in \Gamma_0} \abs{ \delta_1\phi_1(\kappa_{\theta_1}(\rho_1(\gamma))) - \delta_2\phi_2(\kappa_{\theta_2}(\rho_2(\gamma)))} < +\infty$.
\item\label{length rigidity} $\delta_1\ell^{\phi_1}(\rho_1(\gamma)) = \delta_2\ell^{\phi_2}(\rho_2(\gamma))$ for all $\gamma \in \Gamma_0$. 
\end{enumerate} 
If, in addition, $\GG_j$ is simple, $Z(\GG_j)$ is trivial and $\rho_j$ has Zariski-dense image for $j=1,2$, then there is an isomorphism $\Psi:\GG_1\to\GG_2$ such that $\rho_2=\Psi\circ\rho_1$.
\end{corollary}

The proof of Corollary~\ref{cor:non-singular BM measures} requires the following lemma. 

\begin{lemma}\label{lem: shadows get smaller} 
Suppose $\Omega\subset\Pb(\Rb^{d})$ is a properly convex domain and $\Gamma\subset\Aut(\Omega)$ is a projectively visible subgroup. Let $\d_{\Pb}$ be a distance on $\Pb(\Rb^{d})$ induced by a Riemannian metric. If $r > 0$, $b_0 \in \Omega$ and $\{\gamma_n\}$ is a sequence of distinct elements in $\Gamma$, then 
$$
\lim_{n \rightarrow \infty}  \mathrm{diam} \left(\Oc_r(b_0, \gamma_n(b_0)) \right) =0,
$$
where the diameter is computed using $\d_{\Pb}$.
\end{lemma} 

\begin{proof} Fix a subsequence $\{\gamma_{n_j}\}$ such that 
$$
\limsup_{n \rightarrow \infty}  \mathrm{diam} \left(\Oc_r(b_0, \gamma_n(b_0)) \right) = \lim_{j \rightarrow \infty}  \mathrm{diam} \left(\Oc_r(b_0, \gamma_{n_j}(b_0)) \right).
$$
Passing to a further subsequence we can suppose that $\gamma_{n_j}(b_0) \rightarrow x \in \Lambda_\Omega(\Gamma)$ and $\gamma_{n_j} \rightarrow T \in \Pb({\rm End}(\Rb^d))$. To show that $\mathrm{diam} \left(\Oc_r(b_0, \gamma_{n_j}(b_0)) \right)$ converges to 0, it suffices to fix a sequence $\{y_j\}$ where $y_j \in \Oc_r(b_0, \gamma_{n_j}(b_0))$ for all $j \ge 1$ and show that $y_j \rightarrow x$. By definition, for each $j \ge 1$ there exists $y_j^\prime \in [b_0, y_j)$ such that $\d_\Omega(y_j^\prime, \gamma_{n_j}(b_0)) < r$. Then the sequence $\{ \gamma_{n_j}^{-1}(y_j^\prime)\}$ is relatively compact in $\Omega$. So by Proposition~\ref{prop: limit set projectively visible}(3)
$$
x = T\left( \lim_{j \rightarrow \infty} \gamma_{n_j}^{-1}(y_j^\prime) \right) = \lim_{j \rightarrow \infty} \gamma_{n_j}\gamma_{n_j}^{-1}(y_j^\prime) = \lim_{j \rightarrow \infty} y_j^\prime. 
$$
Since $y_j^\prime \in [b_0, y_j)$ for all $j \ge 1$, this implies that $y_j \rightarrow x$. 
\end{proof}

\begin{proof}[Proof of Corollary~\ref{cor:non-singular BM measures}] 
By the ergodicity of the flow $\{\varphi_t\}$ (see Theorem \ref{thm: ergodicity}) and the assumption that $m_1$ is non-singular with respect to $m_2$, there exists $c>0$ such that $m_{1}=c m_{2}$.

Note that for $j=1,2$ and $\gamma \in \Gamma_0$, 
$$ 
\ell^{\phi_j}(\rho_j(\gamma)) = \lim_{n \rightarrow \infty} \frac{1}{n} \phi_j(\kappa_{\theta_j}(\rho_j(\gamma^n))).
$$
Thus, to prove part \eqref{length rigidity}, it suffices to prove part \eqref{magnitude rigidity}.

For all $\psi \in \left\{ \phi_1,\phi_2,\bar{\phi}_1,\bar{\phi}_2\right\}$, let $\nu_\psi$ be the measure on $\partial \Omega$ given by 
$$
\nu_\psi(A) = \mu_\psi\big( \xi(A \cap \Lambda_\Omega(\Gamma_0))\big).
$$ 
Fix $r > 0$ sufficiently large so that the Shadow lemma (Proposition \ref{prop:shadow estimates}) holds for the probability measures $\nu_{\phi_1}$ and $\nu_{\phi_2}$. Then there is some $C>0$ such that
\begin{align}\label{eqn: shadow lemma application 2}
\frac{1}{C}\frac{e^{\delta_2\phi_2(\kappa_{\theta_2}(\rho_2(\gamma)))}}{e^{\delta_1\phi_1(\kappa_{\theta_1}(\rho_1(\gamma)))}}\le\frac{\nu_{\phi_1}(\Oc_r(b_0, \gamma(b_0)))}{\nu_{\phi_2}(\Oc_r(b_0, \gamma(b_0)))}\le C\frac{e^{\delta_2\phi_2(\kappa_{\theta_2}(\rho_2(\gamma)))}}{e^{\delta_1\phi_1(\kappa_{\theta_1}(\rho_1(\gamma)))}}
\end{align}
for all $\gamma\in\Gamma_0$.

Fix a distance $\d_{\Pb}$ on $\Pb(\Rb^{d})$ induced by a Riemannian metric, fix $x_1, x_2 \in \Lambda_\Omega(\Gamma_0)$ distinct and let $\epsilon := \frac{1}{6} \d_{\Pb}(x_1, x_2)$.
Lemma~\ref{lem: shadows get smaller}  implies that there exists a finite set $S \subset \Gamma_0$ such that 
 $$
 \mathrm{diam} \left(\Oc_r(b_0, \gamma(b_0)) \right) \le \epsilon
 $$
 for all $\gamma \in \Gamma_0 - S$. Hence, for each $\gamma \in \Gamma_0 - S$, there is some $i \in \{1,2\}$  so that
 $$
B_i \times  \Oc_r(b_0, \gamma(b_0)) \subset \left\{ (x,y) \in \Lambda_\Omega(\Gamma_0)^2  : \d_{\Pb}(x,y) \ge \epsilon\right\}=:K,
$$
where $B_i := \{ y \in \Lambda_\Omega(\Gamma_0) : \d_{\Pb}(y,x_i) \le \epsilon\}$. From the definitions of $\d m_1$ and $\d m_2$, and the fact that $m_1 = cm_2$, we see that if we set 
\[C_0:=c\max_{(x,y)\in K}\frac{e^{\phi_1([\xi_1(x),\xi_1(y)]_{\theta_1})}}{e^{\phi_2([\xi_2(x),\xi_2(y)]_{\theta_2})}},\] 
then
 \begin{equation*}
\frac{1}{C_0} (\nu_{\bar{\phi}_2} \otimes\nu_{\phi_2} )(A) \le (\nu_{\bar{\phi}_1} \otimes \nu_{\phi_1})(A) \le C_0 (\nu_{\bar{\phi}_2} \otimes \nu_{\phi_2})(A)
 \end{equation*}
 for all Borel measurable sets $A \subset K$. Hence, if we set
\[C_1:=C_0\max\left\{\frac{\nu_{\bar{\phi}_2}(B_1)}{\nu_{\bar{\phi}_1}(B_1)},\frac{\nu_{\bar{\phi}_2}(B_2)}{\nu_{\bar{\phi}_1}(B_2)}\right\},\]
then 
\begin{align}\label{eqn: ratio bound}
\frac{1}{C_1} \le \frac{\nu_{\phi_1}(\Oc_r(b_0, \gamma(b_0)))}{\nu_{\phi_2}(\Oc_r(b_0, \gamma(b_0)))} \le C_1
\end{align}
for all $\gamma\in\Gamma_0-S$. 

Since $S$ is finite, \eqref{eqn: shadow lemma application 2} and \eqref{eqn: ratio bound} imply that
$$
\sup_{\gamma \in \Gamma_0} \abs{\delta_1 \phi_1(\kappa_{\theta_1}(\rho_1(\gamma))) -\delta_2 \phi_2(\kappa_{\theta_2}(\rho_2(\gamma)))}<+\infty.
$$

To prove the last claim of the corollary, we use the following argument of Dal'bo and Kim~\cite{dalbo-kim}. Consider the product representation $\rho_1\times\rho_2:\Gamma\to\GG_1\times\GG_2$, let $\Delta$ denote the set of simple roots of $\GG_1\times\GG_2$, and let $\mfa$ denote the Cartan subspace of $\GG_1\times\GG_2$.
Corollary~\ref{cor:non-singular BM measures}\eqref{length rigidity} implies that the $\Delta$-Benoist limit cone $\mathcal B(\rho_1\times\rho_2)$ lies in a hyperplane in
$\mfa$. A theorem of Benoist \cite{benoist-limit-cone} then  implies that the Zariski closure $Z$ of $(\rho_1\times\rho_2)(\Gamma)$ is properly contained in $\GG_1\times\GG_2$.

Let $\pi_j:\GG_1\times\GG_2\to\GG_j$ be the projection map.
Then  the kernel $\pi_{3-j}|_Z$ is a normal subgroup of $\GG_j$, which is not all of $\GG_j$. Since $\GG_j$ is simple and $Z(\GG_j)$ is trivial, we conclude
that $\pi_{3-j}|_Z$ is injective. Since $\rho_j$ has Zariski dense image, $\pi_{3-j}|_Z$ is also surjective. Hence, $\Psi:=\pi_2|_Z\circ \pi_1|_Z^{-1}$ is an isomorphism
such that  $\rho_2 = \Psi \circ \rho_1$.
 \end{proof}

\section{A Manhattan curve theorem}
\label{manhattan}

Sambarino \cite{sambarino-dichotomy} showed that when $\Gamma$ is Anosov, the entropy functional is concave and characterizes
when it is not strictly concave (see also Potrie-Sambarino  \cite{potrie-sambarino}). One may view this as an analogue of Burger's Manhattan Curve Theorem \cite{burger}, since in
this setting both are consequences of the convexity of the pressure function and rigidity results for equilibrium measures. However, in our setting
we do not have access to thermodynamic formalism, so we must adapt other methods.
 
\begin{theorem} 
\label{manhattan curve}
Suppose $\theta \subset \Delta$ is symmetric, $\Gamma$ is a non-elementary $\mathsf{P}_{\theta}$-transverse subgroup of $\GG$ and $\phi_1,\phi_2 \in \mfa_\theta^*$ satisfy $\delta^{\phi_1}(\Gamma) = \delta^{\phi_2}(\Gamma)=1$. If $\phi = \lambda \phi_1 + (1-\lambda) \phi_2$ for some $\lambda \in (0,1)$, then  
$$
\delta:=\delta^\phi(\Gamma) \le 1.
$$
Moreover, if  $\delta^{\phi}(\Gamma) = 1$ and $Q_\Gamma^\phi$ diverges at its critical exponent, then $\ell^{\phi_1}(\gamma) = \ell^{\phi_2}(\gamma)$ 
for all $\gamma \in \Gamma$.
\end{theorem} 

As a consequence of Theorem \ref{manhattan curve}, we use  a result of Benoist \cite{benoist-limit-cone} to show that equality never occurs when $\Gamma$ is Zariski dense.

\begin{corollary}\label{cor:manhattan curve Z dense case}
Suppose $\theta \subset \Delta$ is symmetric, $\Gamma$ is a Zariski dense $\Psf_\theta$-transverse subgroup of $\GG$, and $\phi_1, \phi_2 \in \mfa^*_\theta$ are distinct and satisfy $\delta^{\phi_1}(\Gamma)=\delta^{\phi_2}(\Gamma) =1$.
If $\phi = \lambda \phi_1+(1-\lambda)\phi_2$ for some $\lambda \in (0,1)$ and $Q_\Gamma^\phi$ diverges at its critical exponent, then $\delta^{\phi}(\Gamma) <1$.
\end{corollary}

\medskip\noindent
{\em Proof of Corollary \ref{cor:manhattan curve Z dense case}.}
For $g \in \GG$ define 
$$
\nu(g) : = \lim_{n \rightarrow \infty} \frac{1}{n} \kappa(g^n) \in \mfa^+\,\,\text{ and }\,\,\nu_\theta(g) : = \lim_{n \rightarrow \infty} \frac{1}{n} \kappa_\theta(g^n) \in \mfa_\theta^+
$$
(these limit exists by Fekete's Subadditive Lemma). Note that via the identification of $\mfa_\theta^*$ as a subspace of $\mfa^*$ described in Section~\ref{sec:background}, we have
\[\ell^{\phi_j}(\gamma)=\phi_j(\nu_\theta(\gamma))=\phi_j(\nu(\gamma))\]
for both $j=1,2$ and all $\gamma\in\Gamma$. 

Suppose for a contradiction that $\delta^{\phi}(\Gamma)=1$. By Theorem \ref{manhattan curve},
$\phi_1(\nu(\gamma))=\phi_2(\nu(\gamma))$ for all $\gamma \in \Gamma$, which implies that $\phi_1 = \phi_2$ on $$
\Cc : = \overline{\cup_{\gamma \in \Gamma} \Rb_{>0} \nu(\gamma)}
$$
Since $\Gamma$ is Zariski dense,
a result of Benoist \cite{benoist-limit-cone} implies that $\Cc$ is a convex subset of $\mfa$ with non-empty interior, so $\phi_1 = \phi_2$, and we obtain a contradiction. 
\qed

\begin{proof}[Proof of Theorem \ref{manhattan curve}] The general strategy of our proof is inspired by the proof of Theorem 1(a) in~\cite{burger}.

The first part follows immediately from the definition and H\"older's inequality which gives that, for all $s$,
\begin{equation*}
Q^\phi_\Gamma(s)\le Q^{\phi_1}_\Gamma(s)^{\lambda}  Q^{\phi_2}_\Gamma(s)^{1-\lambda}.
\end{equation*} 
So our main work is to establish the ``moreover'' part of the theorem.

Suppose that $\delta^\phi(\Gamma) =1$ and $Q_{\Gamma}^\phi(1) = +\infty$. For $\psi \in \left\{ \phi_1, \phi_2,  \phi, \bar{\phi}_1, \bar{\phi}_2,\bar{\phi}\right\}$, let $\mu_\psi$ denote a $\psi$-Patterson-Sullivan measure for $\Gamma$ of dimension $1$.

Using the discussion in Section~\ref{sec: a helpful reduction} we may assume that $\GG$ has trivial center and that $\Psf_\theta$ does not contain any simple factors of $\GG$. Then by Theorem~\ref{thm:transverse image of visible general}, there is a properly convex domain $\Omega\subset\Pb(\Rb^{d})$, a projectively visible subgroup $\Gamma_0\subset\Aut(\Omega)$ and a faithful $\Psf_\theta$-transverse representation $\rho:\Gamma_0\to\GG$ with limit map $\xi:\Lambda_\Omega(\Gamma_0)\to\Fc_\theta$ so that $\rho(\Gamma_0)=\Gamma$ and $\xi(\Lambda_\Omega(\Gamma_0))=\Lambda_\theta(\Gamma)$. 

For $\psi \in \left\{ \phi_1, \phi_2,  \phi, \bar{\phi}_1, \bar{\phi}_2,\bar{\phi}\right\}$, define a measure $\nu_\psi$ on $\partial \Omega$ by 
$$
\nu_\psi(A) = \mu_\psi\left( \xi(A \cap \Lambda_\Omega(\Gamma_0))\right).
$$
Fix $b_0 \in \Omega$. Recall, from  Section~\ref{sec:consequences of ergodicity},  that $\Lambda_{\Omega, b_0,R}^{\rm con}(\Gamma_0) \subset \Lambda_\Omega(\Gamma_0)$ denotes the set of limit points which are $R$-conical. By Corollary~\ref{cor:uniformly conical set has full measure} we can fix $R > 0$ sufficiently large so that 
$$
\nu_\phi\left( \Lambda_{\Omega, b_0,R}^{\rm con}(\Gamma_0)\right) = 1\quad\text{and}\quad \nu_{\bar\phi}\left( \Lambda_{\Omega, b_0,R}^{\rm con}(\Gamma_0)\right) = 1.
$$
Using the Shadow Lemma (Proposition~\ref{prop:shadow estimates}) and possibly increasing $R$, 
we can also assume that for every $r \ge R$ there exists a constant $C_r \ge  1$ such that 
\begin{align}\label{eqn: shadow lemma application}
C_r^{-1} e^{- \psi(\kappa_\theta(\rho(\gamma)))} \le \nu_\psi\big(\Oc_r(b_0,\gamma(b_0))\big)  \le C_re^{- \psi(\kappa_\theta(\rho(\gamma)))}
\end{align}
for all $\gamma \in \Gamma$ and $\psi \in \left\{ \phi_1, \phi_2,  \phi, \bar{\phi}_1, \bar{\phi}_2,\bar{\phi}\right\}$.

For all $\alpha, \beta \in \Gamma_0$ and $r > 0$, let 
$$
\Rc_r(\alpha, \beta) := \Oc_r(b_0, \alpha(b_0)) \times \Oc_r(b_0, \beta(b_0)).
$$

The following lemma is the crucial place where we use the fact that $\delta^\phi(\Gamma)=\delta^{\phi_1}(\Gamma)=\delta^{\phi_2}(\Gamma)$.

\begin{lemma} If $r\ge R$ and $\alpha, \beta \in \Gamma_0$, then 
\begin{align}
\label{eqn:inequality on product of shadows}
(\nu_{\bar{\phi}} \otimes \nu_{\phi})& \big(\Rc_r(\alpha, \beta) \big)  \le C_r^4\left(\nu_{\bar{\phi}_1} \otimes \nu_{\phi_1} + \nu_{\bar{\phi}_2} \otimes \nu_{\phi_2}\right)\big(\Rc_r(\alpha, \beta) \big).
\end{align}
\end{lemma} 

\begin{proof} By repeated applications of the Shadow Lemma \eqref{eqn: shadow lemma application}, we see that if $\alpha, \beta \in \Gamma_0$, then
\begin{align*}
(\nu_{\bar{\phi}} \otimes \nu_{\phi}) & \big(\Rc_r(\alpha, \beta) \big)  \le C_r^2 e^{-\bar\phi(\kappa_\theta(\rho(\alpha)))}e^{-\phi(\kappa_\theta(\rho(\beta)))} \\
 & =  C_r^2 e^{-\big(\lambda\bar\phi_1(\kappa_\theta(\rho(\alpha)))+(1-\lambda)\bar\phi_2(\kappa_\theta(\rho(\alpha)))+\lambda\phi_1(\kappa_\theta(\rho(\beta)))+(1-\lambda)\phi_2(\kappa_\theta(\rho(\beta)))\big)}\\
 &\le C_r^4 \nu_{\bar{\phi}_1}\left(\Oc_r(b_0, \alpha(b_0))\right)^\lambda \nu_{\bar{\phi}_2}\left(\Oc_r(b_0, \alpha(b_0))\right)^{1-\lambda} \nu_{\phi_1}\left(\Oc_r(b_0, \beta(b_0))\right)^\lambda \nu_{\phi_2}\left(\Oc_r(b_0, \beta(b_0))\right)^{1-\lambda}\\
 & = C_r^4(\nu_{\bar{\phi}_1} \otimes \nu_{\phi_1})\big(\Rc_r(\alpha, \beta) \big)^{\lambda}(\nu_{\bar{\phi}_2} \otimes \nu_{\phi_2})\big(\Rc_r(\alpha, \beta) \big)^{1-\lambda}.
\end{align*}
We may then apply the weighted Arithmetic Mean-Geometric Mean Inequality to see that 
\begin{align*}
(\nu_{\bar{\phi}} \otimes \nu_{\phi}) \big(\Rc_r(\alpha, \beta) \big)  \le C_r^4 \left(\nu_{\bar\phi_1} \otimes \nu_{\phi_1} + \nu_{\bar\phi_2} \otimes \nu_{\phi_2}\right)\big(\Rc_r(\alpha, \beta) \big)
\end{align*}
for all  $\alpha, \beta \in \Gamma_0$. 
\end{proof} 

Our goal is to upgrade the inequality in Equation~\eqref{eqn:inequality on product of shadows} to all Borel measurable sets in $\Lambda_\Omega(\Gamma_0)^{2}$. 
We first show that shadows form a neighborhood basis of every point in $\Lambda_{\Omega, b_0, R}(\Gamma_0)$. 

\begin{lemma}\label{lem:nbhd basis} If $x \in \Lambda_{\Omega, b_0, R}(\Gamma_0)$ and $U$ is a neighborhood of $x$ in $\partial \Omega$, then there exists $\gamma \in \Gamma$ such that 
$$
x \in \Oc_R(b_0, \gamma(b_0)) \subset U.
$$
\end{lemma}

\begin{proof} Fix a sequence $\{\gamma_n\}$ in $\Gamma$ such that $\gamma_n(b_0) \rightarrow x$ and $\dist_\Omega( \gamma_n(b_0), [b_0, x)) < R$ for all $n \ge 1$. 
Then $x \in \Oc_R(b_0, \gamma_n(b_0))$ for all $n \ge 1$ and Lemma~\ref{lem: shadows get smaller} implies that $\Oc_R(b_0, \gamma_n(b_0)) \subset U$ when $n$ is sufficiently large.
\end{proof}

Next, by the argument in~\cite[pg. 23]{roblin}, we observe that the shadows satisfy a version of the Vitali covering lemma. 

\begin{lemma}\label{lem:1->5covering} If $I \subset \Gamma_0$ and $r>0$,  then there exists $J \subset I$ such that the sets 
\hbox{$\{\Oc_r(b_0, \gamma(b_0)) : \gamma \in J\}$} are pairwise disjoint and 
$$
\bigcup_{\gamma \in I} \Oc_r(b_0, \gamma(b_0)) \subset \bigcup_{\gamma \in J} \Oc_{5r}(b_0, \gamma(b_0)).
$$
\end{lemma}

We now leverage our covering lemma to upgrade Equation \eqref{eqn:inequality on product of shadows} to all measurable subsets of $\Lambda(\Gamma_0)^2$.

\begin{lemma}\label{lem:inequality for general Borel sets} 
There exists $C>0$ such that: if $A \subset \Lambda_\Omega(\Gamma_0)^2$ is a Borel measurable set, then
$$
(\nu_{\bar\phi} \otimes \nu_{\phi})(A) \le C \left(\nu_{\bar\phi_1} \otimes \nu_{\phi_1} +\nu_{\bar\phi_2} \otimes \nu_{\phi_2}\right)(A).
$$
\end{lemma}

\begin{proof} 
It suffices to prove the lemma in the case when $A=A_1\times A_2$ for some $A_1,A_2\subset\Lambda_\Omega(\Gamma_0)$. Fix $\epsilon > 0$. By the outer regularity of the measures, for both $j=1,2$, there exists an open set $U_j \supset A_j$ with 
\begin{align*}
\left(\nu_{\bar\phi_1} \otimes \nu_{\phi_1}+\nu_{\bar\phi_2} \otimes \nu_{\phi_2}\right) \big(U_1 \times U_2\big)  \le \left(\nu_{\bar{\phi}_1} \otimes \nu_{\phi_1}+\nu_{\bar\phi_2} \otimes \nu_{\phi_2}\right)\big(A_1 \times A_2\big)  +\epsilon.
\end{align*}
If we let $I_j := \{ \alpha \in \Gamma_0 : \Oc_R(b_0, \alpha(b_0)) \subset U_j\}$, then by Lemma~\ref{lem:nbhd basis}
$$
(A_1\times A_2) \cap \Lambda_{\Omega, b_0, R}^{\rm con}(\Gamma_0)^2\subset \bigcup_{(\alpha,\beta) \in I_1\times I_2}   \Rc_{R}(\alpha,\beta)\subset U_1\times U_2.
$$
By Lemma~\ref{lem:1->5covering}, we can find a subset $J_j \subset I_j$ such that the sets $\{ \Oc_R(b_0, \alpha(b_0)): \alpha \in J_j\}$ are pairwise disjoint and 
$$
\bigcup_{\alpha \in I_j}  \Oc_R(b_0, \alpha(b_0)) \subset \bigcup_{\alpha \in J_j}  \Oc_{5R}(b_0, \alpha(b_0)).
$$

Since we chose $R>0$ such that $\nu_{\phi}\big(\Lambda_{\Omega,b_0,R}^{\rm con}(\Gamma_0)\big)=\nu_{\bar\phi}\big(\Lambda_{\Omega,b_0,R}^{\rm con}(\Gamma_0)\big)=1$, it follows that
\begin{align*}
(\nu_{\bar\phi} \otimes \nu_{\phi})  (A_1 \times A_2)& =(\nu_{\bar\phi} \otimes \nu_{\phi})  \big((A_1 \times A_2)\cap \Lambda_{\Omega, b_0, R}(\Gamma_0)^2\big) \\
&   \le  \sum_{(\alpha,\beta) \in J_1\times J_2}(\nu_{\bar\phi} \otimes \nu_{\phi})\left( \Rc_{5R}(\alpha,\beta) \right).
\end{align*}
Then by repeated applications of Equations \eqref{eqn: shadow lemma application} and \eqref{eqn:inequality on product of shadows},
\begin{align*}
 \sum_{(\alpha,\beta) \in J_1\times J_2} (\nu_{\bar\phi} \otimes \nu_{\phi})\big( \Rc_{5R}(\alpha,\beta) \big) &  \le C_R^4C_{5R}^4  \sum_{(\alpha,\beta) \in J_1\times J_2}  (\nu_{\bar\phi} \otimes \nu_{\phi})\big( \Rc_{R}(\alpha,\beta) \big) \\
& \le C_R^8C_{5R}^4 \sum_{(\alpha,\beta) \in J_1\times J_2}  \left(\nu_{\bar\phi_1} \otimes \nu_{\phi_1}+\nu_{\bar\phi_2} \otimes \nu_{\phi_2}\right)\big(\Rc_R(\alpha, \beta) \big) \\
& \le C_R^8C_{5R}^4\left(\nu_{\bar\phi_1} \otimes \nu_{\phi_1}+\nu_{\bar\phi_2} \otimes \nu_{\phi_2}\right)\big(U_1 \times U_2\big)\\
&  \le C_R^8C_{5R}^4\left(\nu_{\bar\phi_1} \otimes \nu_{\phi_1}+\nu_{\bar\phi_2} \otimes \nu_{\phi_2}\right)\big(A_1 \times A_2\big)  +C_R^6C_{5R}^2\epsilon.
\end{align*}
Since $\epsilon > 0$ was arbitrary, it follows that
\[
\left(\nu_{\bar\phi} \otimes \nu_{\phi}\right)(A_1 \times A_2) \le C_R^6C_{5R}^2 \left(\nu_{\bar\phi_1} \otimes \nu_{\phi_1}+
\nu_{\bar\phi_2} \otimes \nu_{\phi_2}\right)(A_1 \times A_2).\qedhere
\]
\end{proof}

Lemma~\ref{lem:inequality for general Borel sets} implies that $\nu_{\bar\phi} \otimes \nu_{\phi}$ is absolutely continuous with respect to 
$\nu_{\bar\phi_1} \otimes \nu_{\phi_1}+\nu_{\bar\phi_2} \otimes \nu_{\phi_2}$. Therefore,  after possibly relabelling, we can assume that $\nu_{\bar\phi} \otimes \nu_{\phi}$ is non-singular with respect to $\nu_{\bar\phi_1} \otimes \nu_{\phi_1}$. 

We claim that $Q^{\phi_1}_\Gamma(1)=+\infty$. Otherwise, Proposition~\ref{prop: consequences of shadow lemma} would imply that 
$$
\nu_{\phi_1}(\Lambda^{\rm con}_\Omega(\Gamma_0))=\nu_{\bar\phi_1}(\Lambda^{\rm con}_\Omega(\Gamma_0))=0,
$$
which is impossible since 
$$
\nu_{\phi}(\Lambda^{\rm con}_\Omega(\Gamma_0))=\nu_{\bar\phi}(\Lambda^{\rm con}_\Omega(\Gamma_0))=1
$$
by Proposition~\ref{prop: support on conical limit set}.

Since $\nu_{\bar\phi} \otimes \nu_{\phi}$ is non-singular with respect to $\nu_{\bar\phi_1} \otimes \nu_{\phi_1}$, the associated Bowen-Margulis measures are non-singular. Hence by Corollary~\ref{cor:non-singular BM measures} we have $\ell^{\phi}(\gamma) = \ell^{\phi_1}(\gamma)$ 
for all $\gamma \in \Gamma$. Thus 
$\ell^{\phi_1}(\gamma) = \ell^{\phi_2}(\gamma)$ 
for all $\gamma \in \Gamma$. 
\end{proof}

Notice that the H\"older inequality similarly proves a statement which is of the same form as Burger's Manhattan Curve Theorem. However, we are not
able to give an analogous characterization of when equality occurs.

\begin{theorem} Suppose $\theta \subset \Delta$ is symmetric, $\Gamma_1, \Gamma_2 \subset \GG$ are $\Psf_\theta$-transverse subgroups and
there exists an isomorphism $\rho:\Gamma_1\to \Gamma_2$. If $\phi\in\mfa_\theta^*$ and $\delta^\phi(\Gamma_1)=\delta^\phi(\Gamma_2)=1$,
then for any $\lambda\in (0,1)$ the weighted Poincar\'e series
$$\sum_{\gamma\in\Gamma_1}e^{-s\big(\lambda\phi(\kappa_\theta(\gamma))+(1-\lambda)\phi(\kappa_\theta(\rho(\gamma)))\big)}$$
has critical exponent $\delta\le 1$.
\end{theorem}

\appendix

\section{Proof of Proposition~\ref{prop:characterizing convergence in general}}\label{appendix: proof of result about contraction} 

In this section we prove Proposition~\ref{prop:characterizing convergence in general} which we restate here. 

\begin{proposition}\label{prop:characterizing convergence in general in appendix} Suppose $F^\pm \in \Fc_\theta^\pm$, $\{g_n\}$ is a sequence in $\GG$ and $g_n = m_n e^{\kappa(g_n)} \ell_n$ is a $\mathsf{KAK}$-decomposition for each $n \ge 1$. The following are equivalent:  
\begin{enumerate}
\item $m_n\Psf_\theta \rightarrow F^+$, $\ell_n^{-1}\Psf_\theta^- \rightarrow F^-$ and $\alpha(\kappa(g_n)) \rightarrow +\infty$ for every $\alpha \in \theta$,
\item $g_n(F) \to F^+$ for all $F \in \Fc_\theta^+ \setminus \Zc_{F^-}$, and this convergence is uniform on compact subsets of $\Fc_\theta^+ \setminus \Zc_{F^-}$. 
\item $g_n^{-1}(F) \to F^-$ for all $F \in \Fc_\theta^- \setminus \Zc_{F^+}$, and this convergence is uniform on compact subsets of $\Fc_\theta^- \setminus \Zc_{F^+}$. 
\item There are open sets $\mathcal{U}^\pm\subset\Fc_\theta^\pm$ such that $g_n(F) \to F^+$ for all $F \in \mathcal U^+$ and $g_n^{-1}(F) \to F^-$ for all $F \in \mathcal U^-$.
\end{enumerate}
\end{proposition}

 It is well-known that 
\[
\exp:\mathfrak u_\theta^-\to \mathsf{U}_\theta^-:=\exp(\mathfrak{u}_\theta^-)
\] 
is a diffeomorphism. Furthermore, the \emph{Langlands decomposition} (see for instance \cite[Thm.\ 1.2.4.8]{Warner}) of parabolic subgroups states that the map
\[(u,\ell)\in\mathsf \Usf_\theta^-\times L_\theta\mapsto u\ell \in \Psf_\theta^-\]
is a diffeomorphism, where $\mathsf L_\theta:=\Psf_\theta\cap\Psf_\theta^-$. It follows that $\mathsf{U}_\theta^-$ acts simply transitively on $\Fc_\theta \setminus \Zc_{\Psf_\theta^-}$. Thus, the map
\[
T:\mathfrak{u}_\theta^-\to\Fc_\theta \setminus \Zc_{\Psf_\theta^-}
\]
given by $T(X)=e^X\mathsf{P}_\theta$ is a diffeomorphism. 

Note that if $H\in\mathfrak{a}$ and $X\in \mathfrak u_\theta^-$, then
\begin{align}\label{eqn: shrink ball Lie group}
e^HT(X)=e^He^X\mathsf{P}_\theta=e^He^Xe^{-H}\mathsf{P}_\theta=e^{{\rm Ad}(e^H)(X)}\mathsf{P}_\theta=T\left({\rm Ad}(e^H)(X)\right).
\end{align}
Furthermore, if we decompose 
\[
X=\sum_{\alpha\in\Sigma_\theta^+}X_{-\alpha}\in \mathfrak u_\theta^-,
\] 
where $X_{-\alpha}\in\mathfrak g_{-\alpha}$ for all $\alpha\in\Sigma_\theta^+$, then
\begin{align}\label{eqn: shrink ball Lie algebra}
{\rm Ad}(e^H)(X)=\sum_{\alpha\in\Sigma_\theta^+}{\rm Ad}(e^H)(X_{-\alpha})=\sum_{\alpha\in\Sigma_\theta^+}e^{-\alpha(H)}X_{-\alpha}.
\end{align}

Together, Equations~\eqref{eqn: shrink ball Lie group} and~\eqref{eqn: shrink ball Lie algebra} imply the following observation.

\begin{lemma}\label{uniform convergence}
Let $\{H_n\}$ be a sequence in $\mathfrak a^+$. If $\alpha(H_n)\to+\infty$ for all $\alpha\in\theta$, then $e^{H_n} F \to \mathsf{P}_\theta$ for all $F \in \Fc_\theta \setminus \Zc_{\Psf_\theta^-}$, and this convergence is uniform on compact subsets of $\Fc_\theta \setminus \Zc_{\Psf_\theta^-}$.
\end{lemma}

Using Equations~\eqref{eqn: shrink ball Lie group} and~\eqref{eqn: shrink ball Lie algebra}, we can also prove the following lemma.

\begin{lemma}\label{weak version}
Let $g_n = m_n e^{\kappa(g_n)} \ell_n$ be as in the statement of Proposition~\ref{prop:characterizing convergence in general in appendix}. 
\begin{enumerate}
\item If there is an open set $\mathcal{U}\subset\Fc_\theta^+$ such that $g_n(F) \to F^+$ for all $F \in \mathcal U$, then $m_n\Psf_\theta \rightarrow F^+$ and $\alpha(\kappa(g_n)) \to +\infty$ for every $\alpha \in \theta$.
\item If there is an open set $\mathcal{U}\subset\Fc_\theta^-$ such that $g_n^{-1}(F) \to F^-$ for all $F \in \mathcal U$, then $\ell_n^{-1}\Psf_\theta^- \rightarrow F^-$ and $\alpha(\kappa(g_n)) \to +\infty$ for every $\alpha \in \theta$.
\end{enumerate} 
\end{lemma}

\begin{proof} By compactness, it suffices to consider the case where  $m_n\to m\in \mathsf{K}$ and $\ell_n\to\ell\in\mathsf{K}$. 

(1): We first prove that $\alpha(\kappa(g_n))\to+\infty$ for all $\alpha\in\theta$. If this is not the case, then by taking a subsequence, we may assume that there is some $\alpha_0\in\theta$ such that $\alpha_0(\kappa(g_n))\to c\in [0,\infty)$. Choose $F,F'\in\mathcal U$ such that $\ell(F),\ell(F')\in\Fc_\theta\setminus\mathcal Z_{\mathsf{P}_\theta^-}$, and if we decompose 
\[
T^{-1}(\ell(F))=\sum_{\alpha\in\Sigma_\theta^+}X_{-\alpha} \quad \text{and} \quad T^{-1}(\ell(F'))=\sum_{\alpha\in\Sigma_\theta^+}X'_{-\alpha},
\] 
where $X_{-\alpha},X'_{-\alpha}\in\mathfrak g_{-\alpha}$ for all $\alpha\in \Sigma_\theta^+$, then $X_{-\alpha_0}\neq X'_{-\alpha_0}$. Then by \eqref{eqn: shrink ball Lie group} and \eqref{eqn: shrink ball Lie algebra},
\[
\lim_{n\to\infty}T^{-1}(e^{\kappa(g_n)}\ell_n(F))=\lim_{n\to\infty}{\rm Ad}(e^{\kappa(g_n)})T^{-1}(\ell_n(F))=e^{-c}X_{-\alpha_0}+\lim_{n\to\infty}\sum_{\alpha\in\Sigma_\theta^+-\{\alpha_0\}}e^{-\alpha(\kappa(g_n))}X_{-\alpha}.
\]
Similarly, 
\[\lim_{n\to\infty}T^{-1}(e^{\kappa(g_n)}\ell_n(F'))=e^{-c}X'_{-\alpha_0}+\lim_{n\to\infty}\sum_{\alpha\in\Sigma_\theta^+-\{\alpha_0\}}e^{-\alpha(\kappa(g_n))}X'_{-\alpha},\]
so $\lim_{n\to\infty}me^{\kappa(g_n)}\ell_n(F)\neq \lim_{n\to\infty}me^{\kappa(g_n)}\ell_n(F')$, which implies that
\[\lim_{n\to\infty}g_n(F)\neq \lim_{n\to\infty}g_n(F').\]
This is a contradiction because $F,F'\in\mathcal U$.

Next, we prove that $m_n\Psf_\theta \rightarrow F^+$, or equivalently, $m\mathsf{P}_\theta=F^+$. Let $F\in\Fc_\theta$ such that $F$ is transverse to $F^-$ and $\ell(F)$ is transverse to $\mathsf P_\theta^-$. Then there is some compact subset $K\subset\Fc_\theta\setminus\mathcal Z_{\mathsf{P}_\theta^-}$ such that $\ell_n(F)\in K$ for all sufficiently large $n$. Since $\alpha(\kappa(g_n))\to+\infty$ for all $\alpha\in\theta$, Lemma~\ref{uniform convergence} implies that  
\[e^{\kappa(g_n)}\ell_n(F)\to \mathsf{P}_\theta,\]
which implies that
\[g_n(F)=m_ne^{\kappa(g_n)}\ell_n(F)\to m\mathsf{P}_\theta.\]
It follows that $m\mathsf{P}_\theta=F^+$.

(2): As in Section~\ref{sec:background}, let $k_0\in \mathsf{N}_{\mathsf{K}}(\mfa)$ be a representative of the longest element $w_0 \in W$. Observation~\ref{obs: opposition involution 2} implies that ${\rm Ad}(k_0)(-\kappa(g)) = \kappa(g^{-1})$ for all $g\in\GG$, and so 
$$
g_n^{-1} = \left(\ell_n^{-1} k_0^{-1} \right) e^{\kappa(g_n^{-1})} \left(k_0 m_n^{-1}\right)
$$
is a $\mathsf{KAK}$-decomposition of $g_n^{-1}$. 

Further, $\Psf_{\iota^*(\theta)} = k_0\Psf_\theta^-k_0^{-1}$, see Equation~\eqref{eqn:opposite conjugate form}, so we can define a $\GG$-equivariant diffeomorphism 
\[\Phi_\theta:\Fc_\theta^-\to \Fc_{\iota^*(\theta)}\] 
by $\Phi_\theta(g \Psf_\theta^-) = g k_0 \Psf_{\iota^*(\theta)}$. Then $g_n^{-1}(F) \rightarrow \Phi_\theta(F^-)$ for all $F \in \Phi_\theta(\mathcal{U})$. So by part (1), we see that $\ell_n^{-1} k_0^{-1}  \Psf_{\iota^*(\theta)} \rightarrow \Phi_\theta(F^-)$ and $\alpha(\kappa(g_n^{-1}))\to+\infty$ for all $\alpha\in\iota^*(\theta)$. Since $\Phi_\theta(\ell_n^{-1}\Psf_\theta^-) = \ell_n^{-1} k_0^{-1} \Psf_{\iota^*(\theta)}$ this implies that $\ell_n^{-1}\Psf_\theta^- \rightarrow F^-$. Further, by Observation \ref{obs: opposition involution 2}, 
\[\alpha( \kappa(g)) = \iota^*(\alpha)(\kappa(g^{-1}))\] 
for all $g \in \GG$ and all $\alpha\in\theta$. So we see that $\alpha(\kappa(g_n))\to+\infty$ for all $\alpha\in\theta$.
\end{proof}

\begin{proof}[Proof of Proposition~\ref{prop:characterizing convergence in general in appendix}]
It follows immediately from Lemma \ref{weak version} that (4) implies (1), and it is obvious that (2) and (3) together imply (4). It thus suffices to show that (2) and (3) are both individually equivalent to (1). By compactness, it suffices to consider the case where  $m_n\to m\in \mathsf{K}$ and $\ell_n\to\ell\in\mathsf{K}$. 

We first prove that (1) implies (2).  Since $\alpha(\kappa(g_n))\to+\infty$ for all $\alpha\in\theta$, Lemma~\ref{uniform convergence} implies
$$
\lim_{n \rightarrow \infty} e^{\kappa(g_n)} F = \Psf_\theta
$$ 
for all $F \in \Fc_\theta \setminus \Zc_{\Psf_\theta^-}$, and this convergence is uniform on compact subsets of $\Fc_\theta \setminus \Zc_{\Psf_\theta^-}$. Since, $m \Psf_\theta =F^+$ and $\ell^{-1}\Psf_\theta^- = F^-$, it follows that
$$
\lim_{n \rightarrow \infty} g_n(F) = F^+
$$ 
for all $F \in \Fc_\theta \setminus \Zc_{F^-}$, and this convergence is uniform on compact subsets of $\Fc_\theta \setminus \Zc_{F^-}$. 

Next, we prove (2) implies (1). By Lemma \ref{weak version}, $m_n \Psf_\theta \rightarrow F^+$ and $\alpha(\kappa(g_n)) \to +\infty$ for every $\alpha \in \theta$, so it suffices to show that $\ell_n^{-1}\Psf_\theta^-  \rightarrow F^-$, or equivalently, that $\ell F^- = \Psf_\theta^-$. If this were not the case, then there exists some $F \in  \Zc_{\mathsf P_\theta^-}\setminus \Zc_{\ell F^-}$. Then there is a compact set $K\subset\Fc_\theta\setminus\Zc_{F^-}$ such that $\ell_n^{-1}(F)\in K$ for all sufficiently large $n$. Then by assumption,
$$
m\lim_{n\to\infty}e^{\kappa(g_n)} F=\lim_{n\to\infty}g_n\ell_n^{-1} F=F^+=m\Psf_\theta,
$$
so $e^{\kappa(g_n)} F\to \Psf_\theta$. However, $\{e^{\kappa(g_n)} \} \subset \Psf_\theta^-$, so each $e^{\kappa(g_n)}$ preserves the closed set $\Zc_{\mathsf P_\theta^-}$, which implies that
$$
\Psf_\theta =  \lim_{n \rightarrow \infty} e^{\kappa(g_n)} F \in \Zc_{\mathsf P_\theta^-}.
$$
Since $\Psf_\theta$ and $\Psf_\theta^-$ are transverse, we have a contradiction.

Finally, we prove that (1) and (3) are equivalent. Let $k_0 \in N_{\Ksf}(\mathfrak a)$ be a representative of the longest element $w_0\in W$, and let
\[\Phi_\theta:\Fc_\theta^-\to \Fc_{\iota^*(\theta)}\] 
be the $\GG$-equivariant homeomorphism given by $\Phi_\theta(g \Psf_\theta^-) = g k_0\Psf_{\iota^*(\theta)}$. Observe that 
\[\Phi_\theta(\Fc_\theta^-\setminus\Zc_{F^+})=\Fc_{\iota^*(\theta)} \setminus \Zc_{\Phi_{\iota^*(\theta)}^{-1}(F^+)},\]
so (3) can be rewritten as:
\begin{itemize}
\item[(3')] $g_n^{-1}(F) \to \Phi_\theta(F^-)$ for all $F \in \Fc_{\iota^*(\theta)} \setminus \Zc_{\Phi_{\iota^*(\theta)}^{-1}(F^+)}$, and this convergence is uniform on compact subsets of $\Fc_{\iota^*(\theta)} \setminus \Zc_{\Phi_{\iota^*(\theta)}^{-1}(F^+)}$. 
\end{itemize}
By Observation \ref{obs: opposition involution 2}, $\alpha( \kappa(g_n)) = \iota^*(\alpha)(\kappa(g_n^{-1}))$ for all $n\in\Nb$ and all $\alpha\in\Delta$. Thus, (1) can be rewritten as:
\begin{itemize}
\item[(1')] $m_nk_0^{-1}\Psf_{\iota^*(\theta)}^- \rightarrow \Phi_{\iota^*(\theta)}^{-1}(F^+)$, $\ell_n^{-1}k_0^{-1}\Psf_{\iota^*(\theta)} \rightarrow \Phi_\theta(F^-)$ and $\alpha(\kappa(g_n^{-1})) \rightarrow +\infty$ for every $\alpha \in \iota^*(\theta)$. 
\end{itemize}
We also saw in the proof of Lemma \ref{weak version}(2) that if $g_n = m_n e^{\kappa(g_n)} \ell_n$ is a $\mathsf{KAK}$-decomposition of $g \in \GG$, then 
$$
g_n^{-1} = (\ell_n^{-1} k_0^{-1}) e^{  \kappa(g_n^{-1})} (k_0 m_n^{-1})
$$
is a $\mathsf{KAK}$-decomposition of $g_n^{-1}$. Thus, the equivalence between (1) and (2) implies the equivalence between (1') and (3').
\end{proof}

\section{Proofs of Theorem~\ref{thm:transverse image of visible general} and Proposition~\ref{prop:multiplicative along geodesics}}\label{app:proving transverse repn theorems}

In this appendix we prove Theorem~\ref{thm:transverse image of visible general} and Proposition~\ref{prop:multiplicative along geodesics}. 

When $\GG = \PSL(d,\Kb)$, where $\Kb$ is either the real numbers $\Rb$ or the complex numbers $\Cb$, recall from the introduction that $\Delta :=\{ \alpha_1,\dots, \alpha_{d-1}\} \subset \mathfrak{a}^*$ denotes the standard system of simple restricted roots, i.e.
$$
\alpha_j( {\rm diag}(a_1,\ldots,a_d)) = a_j-a_{j+1}
$$
for all ${\rm diag}(a_1,\ldots,a_d) \in \mathfrak{a}$. To simplify notation, we replace subscripts of the form $\{\alpha_{i_1}, \dots, \alpha_{i_k}\}$ with $i_1,\dots, i_k$. For instance, 
\[\Fc_{1,d-1}=\Fc_{\{\alpha_1,\alpha_{d-1}\}}\quad\text{and}\quad U_{1,d-1}(g)=U_{\{\alpha_1,\alpha_{d-1}\}}(g).\]

As mentioned before, in the case when $\GG = \PSL(d,\Kb)$, Theorem~\ref{thm:transverse image of visible general} and Proposition~\ref{prop:multiplicative along geodesics} were proven in ~\cite{CZZ2}. We will use results from~\cite{GGKW} to prove the following proposition, which allows us to generalize these results in \cite{CZZ2} to general $\GG$.

\begin{proposition}\label{prop:reduction to the linear case} For any symmetric $\theta \subset \Delta$ and $\chi \in \sum_{\alpha \in \theta} \Nb \cdot \omega_\alpha$ there exist $d \in \Nb$,  an irreducible linear representation $\Phi : \GG \rightarrow \mathsf{SL}(d,\Rb)$ and a $\Phi$-equivariant smooth embedding
$$
\xi : \Fc_\theta \rightarrow \Fc_{1,d-1}(\Rb^d)
$$
such that: 
\begin{enumerate}
 \item $F_1, F_2 \in \Fc_\theta$ are transverse if and only if $\xi(F_1)$ and $\xi(F_2)$ are transverse. 
 \item There exists $N \in \Nb$ such that 
$$
\log\sigma_1(\Phi(g)) =N\chi(\kappa(g))
$$
for all $g \in \GG$. 
\item $\alpha_1(\kappa\Phi(g))) = \min_{\alpha \in \theta} \alpha( \kappa(g))$ for all $g \in \GG$.
\item If $\min_{\alpha \in \theta} \alpha(\kappa(g)) > 0$, then 
$$
\xi( U_\theta(g)) = U_{1,d-1}(\Phi(g)). 
$$
\item  $\Gamma \subset \GG$ is $\Psf_{\theta}$-divergent (respectively $\Psf_{\theta}$-transverse) if and only if $\Phi(\Gamma)$ is $\Psf_{1,d-1}$-divergent (respectively $\Psf_{1,d-1}$-transverse). Moreover, in this case 
$$
\xi( \Lambda_{\theta}(\Gamma)) = \Lambda_{1,d-1}(\Phi(\Gamma)).
$$
\item If $\rho : \Gamma_0 \rightarrow \GG$ is a $\Psf_{\theta}$-transverse representation with boundary map $\xi_\rho : \Lambda_\Omega(\Gamma_0) \rightarrow \Fc_\theta$, then $\Phi \circ \rho$ is a $\Psf_{1,d-1}$-transverse representation with boundary map $\xi \circ \xi_\rho$.

\end{enumerate}
\end{proposition}

Delaying the proof of Proposition~\ref{prop:reduction to the linear case} for a moment, we prove Theorem~\ref{thm:transverse image of visible general} and Proposition~\ref{prop:multiplicative along geodesics}.

\subsection{Proof of Theorem~\ref{thm:transverse image of visible general}} Let $\Phi : \GG \rightarrow \PSL(d,\Rb)$ and $\xi_{\Phi} : \Fc_\theta \rightarrow \Fc_{1,d-1}(\Rb^d)$ satisfy Proposition~\ref{prop:reduction to the linear case} for some $\chi \in \sum_{\alpha \in \theta} \Nb \cdot \omega_\alpha$. 

Then $\Phi(\Gamma)$ is $\Psf_{1,d-1}$-transverse and so by~\cite[Thm.\ 4.2]{CZZ2} there exist $d_0\in \Nb$, a properly convex domain $\Omega \subset \Pb(\Rb^{d_0})$, 
a projectively visible subgroup $\Gamma_0 \subset \mathrm{Aut}(\Omega)$ and a faithful $\Psf_{1,d-1}$-transverse representation
$\rho_0: \Gamma_0 \rightarrow \PSL(d,\Rb)$ with limit map $\xi_0:\Lambda_{\Omega}(\Gamma_0)\to\mathcal F_{1,d-1}(\Rb^d)$ so that $\rho_0(\Gamma_0)=\Phi(\Gamma)$ and 
$$
\xi_0( \Lambda_{\Omega}(\Gamma_0)) = \Lambda_{1,d-1}(\Phi(\Gamma)) = \xi_{\Phi}( \Lambda_\theta(\Gamma)). 
$$

We claim that $\Phi$ is injective. Since $\GG$ is semisimple, $\ker \Phi$ is either discrete or contains a simple factor of $\GG$. Since $\xi : \Fc_\theta \rightarrow \Fc_{1,d-1}(\Rb^d)$ is a $\Phi$-equivariant embedding, $\ker \Phi$ must act trivially on $\Fc_\theta$. So $\ker \Phi \subset \Psf_\theta$. By assumption $\Psf_\theta$ contains no simple factors of $\GG$, so $\ker \Phi$ is discrete. However then, since $\ker \Phi$ is also normal, we see that $\ker \Phi$ is contained in the center of $\GG$ which by assumption is trivial. Hence $\Phi$ is injective.

Then $\rho := \Phi^{-1} \circ \rho_0$ and $\xi := \xi^{-1}_{\Phi} \circ \xi_{0}$ are well defined and have the desired properties.

\subsection{Proof of Proposition~\ref{prop:multiplicative along geodesics}} 
We start by recalling a result in~\cite{CZZ2} about transverse representations into $\PSL(d,\Kb)$. Let $\d_{\Pb(\Rb^{d_0})}$ be a distance on $\Pb(\Rb^{d_0})$ induced by a Riemannian metric. Given a properly convex domain $\Omega \subset \Pb(\Rb^{d_0})$ and $b_0 \in \Omega$ let 
$$\iota_{b_0}:\Omega\setminus\{b_0\}\to \partial\Omega$$ 
denote the {\em radial projection} map obtained by letting $\iota_{b_0}(z)\in\partial\Omega$ be the unique point so that $z\in (b_0,\iota_{b_0}(z))_\Omega$. The following lemma was proven as Lemma 6.2 and Observation 6.3 in~\cite{CZZ2}.

\begin{lemma}\label{lem:multiplicative along geodesics in linear case 1} 
Suppose $\theta \subset \{\alpha_1,\dots, \alpha_{d-1}\}$ is symmetric. Let $\rho : \Gamma_0 \rightarrow \PSL(d,\Kb)$ be a $\Psf_\theta$-transverse representation, where $\Gamma_0$ is a projectively visible subgroup of $\Aut(\Omega)$ for some properly convex domain $\Omega\subset\Pb(\Rb^{d_0})$.
For any $b_0\in\Omega$ and $\epsilon>0$, there exist $C>0$ such that  if $\gamma, \eta \in \Gamma_0$ and 
\begin{align*}
\d_{\Pb(\Rb^{d_0})}\left( \iota_{b_0}(\gamma^{-1}(b_0)), \iota_{b_0}(\eta(b_0)) \right) \ge\epsilon, 
\end{align*}
then 
\begin{align*}
\abs{\omega_{\alpha_k}\Big(\kappa(\rho(\gamma\eta))-\kappa(\rho(\gamma))-\kappa(\rho(\eta))\Big)} \le C
\end{align*}
for all $\alpha_k \in \theta$. 
\end{lemma} 

Lemma \ref{lem:multiplicative along geodesics in linear case 1} can be restated as follows.

\begin{lemma}\label{lem:multiplicative along geodesics in linear case 2}  
Suppose $\theta \subset \{\alpha_1,\dots, \alpha_{d-1}\}$ is symmetric. Let $\rho : \Gamma_0 \rightarrow \PSL(d,\Kb)$ be a $\Psf_\theta$-transverse representation where $\Gamma_0$ is a projectively visible subgroup of $\Aut(\Omega)$ for some properly convex domain $\Omega\subset\Pb(\Rb^{d_0})$.
For any $b_0\in\Omega$ and $r > 0$, there exist $C>0$ such that  if $\gamma, \eta \in \Gamma_0$ and 
\begin{align*}
\d_{\Omega}\left(\gamma(b_0),[b_0, \eta(b_0)]_\Omega \right) \le r, 
\end{align*}
then 
\begin{align*}
\abs{ \omega_{\alpha_k}\Big(\kappa(\rho(\eta))-\kappa(\rho(\gamma))-\kappa(\rho(\gamma^{-1}\eta))\Big)} \le C
\end{align*}
for all $\alpha_k \in \theta$. 
\end{lemma} 

\begin{proof} 
Suppose not. Then there exist $\alpha_k \in \theta$ and sequences $\{\gamma_n\}$, $\{\eta_n\}$ in $\Gamma$ such that 
\[
\d_\Omega\left( \gamma_n(b_0), [b_0, \eta_n(b_0)]_\Omega \right) \le r\quad\text{but}\quad\abs{ \omega_{\alpha_k}\Big(\kappa(\rho(\eta_n))-\kappa(\rho(\gamma_n))-\kappa(\rho(\gamma_n^{-1}\eta_n))\Big)} \ge n.
\]
Since 
$$
\norm{\kappa(\rho(\eta_n))-\kappa(\rho(\gamma_n^{-1}\eta_n))}\le \sqrt{d}\max\left\{\log\sigma_1(\rho(\gamma_n^{-1})),\log\sigma_1(\rho(\gamma_n))\right\},
$$
$\{\gamma_n\}$ is a diverging sequence.
A similar argument also shows that $\{\gamma_n^{-1}\eta_n\}$ is diverging.

Since both $\{\gamma_n\}$ and $\{\gamma_n^{-1}\eta_n\}$ are diverging and 
\[
\d_\Omega\left( b_0, [\gamma_n^{-1}(b_0), \gamma_n^{-1}\eta_n(b_0)]_\Omega \right)=\d_\Omega\left( \gamma_n(b_0), [b_0, \eta_n(b_0)]_\Omega \right) \le r,
\]
it follows that there is some $\epsilon>0$ so that
\[\d_{\Pb(\Rb^{d_0})}\left( \iota_{b_0}(\gamma_n^{-1}(b_0)), \iota_{b_0}(\gamma_n^{-1}\eta_n(b_0)) \right) \ge\epsilon\]
for all $n$. Thus, Lemma \ref{lem:multiplicative along geodesics in linear case 1} implies that $\abs{ \omega_{\alpha_k}\Big(\kappa(\rho(\eta_n))-\kappa(\rho(\gamma_n))-\kappa(\rho(\gamma_n^{-1}\eta_n))\Big)}$  has a uniform upper bound, which is a contradiction.
\end{proof}

\begin{proof}[Proof of Proposition~\ref{prop:multiplicative along geodesics}]

Since $\left\{\omega_\alpha|_{\mfa_\theta}\right\}_{\alpha \in \theta}$ is a basis for $\mfa_\theta^*$, it suffices to fix $\beta \in \theta$ and find $C > 0$ such that: if $\gamma, \eta \in \Gamma_0$ and 
\begin{align*}
\d_{\Omega}\left(\gamma(b_0),[b_0, \eta(b_0)]_\Omega \right) \le r, 
\end{align*}
then 
\begin{align*}
\abs{\omega_\beta\left( \kappa_\theta(\rho(\gamma\eta))-\kappa_\theta(\rho(\gamma))-\kappa_\theta(\rho(\eta))\right)}\le C.
\end{align*}

Let $\chi_1 := \sum_{\alpha \in \theta} \omega_\alpha$ and $\chi_2 := \omega_\beta + \sum_{\alpha \in \theta} \omega_\alpha$. For $j =1,2$, let $\Phi_j : \GG \rightarrow \PSL(d_j, \Rb)$ satisfy Proposition~\ref{prop:reduction to the linear case} for $\chi_j$, and let $\rho_j: = \Phi_j \circ \rho$. Then $\rho_j$ is a $\Psf_{1,d_j-1}$-transverse representation and there exists $N_j \in \Nb$ such that
\begin{align*}
\abs{\chi_j\left( \kappa_\theta(\rho(\gamma\eta))-\kappa_\theta(\rho(\gamma))-\kappa_\theta(\rho(\eta))\right)}=\frac{1}{N_j}\abs{\omega_{\alpha_1}\left( \kappa(\rho_j(\gamma\eta))-\kappa(\rho_j(\gamma))-\kappa(\rho_j(\eta))\right)}
\end{align*}
for all $\gamma, \eta \in \Gamma$. Applying Lemma~\ref{lem:multiplicative along geodesics in linear case 2} to $\rho_j$,  there exists $C_j > 0$ such that: if $\gamma, \eta \in \Gamma_0$ and 
\begin{align*}
\d_{\Omega}\left(\gamma(b_0),[b_0, \eta(b_0)]_\Omega \right) \le r, 
\end{align*}
then 
\begin{align*}
\abs{\omega_{\alpha_1}\left(\kappa(\rho_j(\gamma\eta))-\kappa(\rho_j(\gamma))-\kappa(\rho_j(\eta))\right)}\le C_j.
\end{align*}
Since $\chi_2 - \chi_1 = \omega_\beta$, we then have: if $\gamma, \eta \in \Gamma_0$ and 
\begin{align*}
\d_{\Omega}\left(\gamma(b_0),[b_0, \eta(b_0)]_\Omega \right) \le r, 
\end{align*}
then 
\begin{equation*}
\abs{\omega_\beta\left( \kappa_\theta(\rho(\gamma\eta))-\kappa_\theta(\rho(\gamma))-\kappa_\theta(\rho(\eta))\right)}\le \frac{C_1}{N_1}+\frac{C_2}{N_2}. \qedhere
\end{equation*}

\end{proof}

\subsection{The proof of Proposition~\ref{prop:reduction to the linear case}} Fix a symmetric set $\theta \subset \Delta$ and $\chi \in \sum_{\alpha\in \theta} \Nb \omega_\alpha$. By~\cite[Lem.\ 3.2, Prop.\ 3.3, Rem.\ 3.6 and Lem.\ 3.7]{GGKW} there exist $N,d \in \Nb$, an irreducible linear representation $\Phi : \GG \rightarrow \mathsf{SL}(d,\Rb)$ and a $\Phi$-equivariant smooth embedding 
$$
\xi : \Fc_\theta \rightarrow \Fc_{1,d-1}:=\Fc_{1,d-1}(\Rb^d)
$$
 such that: 
\begin{enumerate}[label=(\alph*)]
\item\label{item: highest weight of Phi} $\Phi$ is proximal and has highest weight $N \chi$, that is: if $H \in {\rm int}(\mfa^+)$, then $\Phi(e^H)$ is proximal and the eigenvalue with largest modulus is $e^{N \chi(H)}$. 
\item\label{item: SVD of Phi} $\Phi(\Ksf) \subset \mathsf{SO}(d,\Rb)$ and $\Phi(e^{\mfa})$ is a subgroup of the diagonal matrices in $\SL(d,\Rb)$. 
\item\label{item:alpha1 formula} $\alpha_1(\kappa(\Phi(g))) = \min_{\alpha \in \theta} \alpha( \kappa(g))$ for all $g \in \GG$. 
\item\label{item:transverse is preserved under phi} $F_1, F_2 \in \Fc_\theta$ are transverse if and only if $\xi(F_1), \xi(F_2) \in \Fc_{1,d-1}$ are transverse. 
\end{enumerate} 

In the statement of Proposition~\ref{prop:reduction to the linear case}, parts (1) and (3) are restatements of properties \ref{item:transverse is preserved under phi} and \ref{item:alpha1 formula} of $\Phi$ respectively, while part (2) is a consequence of properties~\ref{item: highest weight of Phi} and~\ref{item: SVD of Phi} of $\Phi$. Part (5) follows immediately from parts (1), (3), and (4), while part (6) follows immediately from part (5) and Proposition 2.5. Thus, it suffices to prove part (4). 

Let $e_1,\dots, e_d$ be the standard basis of $\Rb^d$. Using properties~\ref{item: highest weight of Phi} and~\ref{item: SVD of Phi}, we can conjugate $\Phi$ by a permutation matrix and assume that 
\begin{equation}\label{eqn:biggest and smallest eigenvalues}
\Phi(e^H)e_1 = e^{N \chi(H)}e_1 \quad \text{and} \quad \Phi(e^H)e_d = e^{N \bar{\chi}(H)}e_d
\end{equation} 
when $H \in \mfa$ (where as usual $\bar{\chi}=\chi \circ \iota$). We first observe that the value of $\xi(\Psf_\theta)$ is determined.

\begin{lemma}\label{lem:image of Ptheta} $\xi(\Psf_\theta) = \big( \ip{e_1}, \ip{e_1,\dots, e_{d-1}} \big)$. \end{lemma}

\begin{proof} Let $\hat{F}_0^+ :=\big( \ip{e_1}, \ip{e_1,\dots, e_{d-1}} \big)$ and $\hat{F}_0^- :=\big( \ip{e_d}, \ip{e_2,\dots, e_{d}} \big)$. Fix $H \in {\rm int}(\mfa^+)$. Then by property~\ref{item: SVD of Phi} and Equation~\eqref{eqn:biggest and smallest eigenvalues}, $\Phi(e^H)={\rm diag}(a_1,\dots, a_d)$ is a diagonal matrix with 
$$
\abs{a_1} > \max\{ \abs{a_j} : 2 \le j \le d\} \quad \text{and} \quad \abs{a_d} < \min\{ \abs{a_j} : 1 \le j \le d-1\}.
$$
So
$$
\Phi(e^{nH})\hat F \rightarrow \hat{F}_0^+
$$
for all $\hat{F} \in \Fc_{1,d-1}$ transverse to $\hat{F}_0^-$. Since $\Phi$ is irreducible, there exists some $F \in \Fc_\theta$ such that $\xi(F)$ is transverse to $\hat{F}^-_0$. Using Lemma~\ref{uniform convergence} and perturbing $F$ we may also assume that $e^{nH}(F) \rightarrow \Psf_\theta$. Then 
\begin{equation*}
\xi(\Psf_\theta)= \lim_{n \rightarrow \infty} \xi( e^{nH} F) = \lim_{n \rightarrow \infty} \Phi(e^{nH}) \xi(F) = \hat{F}_0^+. \qedhere
\end{equation*}

\end{proof} 

Now we prove (4). 

\begin{lemma}  If $\min_{\alpha \in \theta} \alpha(\kappa(g)) > 0$, then $\xi( U_\theta(g)) = U_{1,d-1}(\Phi(g)). $
\end{lemma}

\begin{proof} Fix a $\mathsf{KAK}$-decomposition $g = m e^H \ell$. By properties~\ref{item: highest weight of Phi} and~\ref{item: SVD of Phi}, there exists a permutation matrix $k \in \mathsf{O}(d)$ such that 
$$
\Phi(g) =\big( \Phi(m)k^{-1} \big) \big( k\Phi(e^H) k^{-1}\big) \big(k\Phi(\ell)\big)
$$
is a singular value decomposition of $\Phi(g)$. By Equation~\eqref{eqn:biggest and smallest eigenvalues}, $k(e_1) = e_1$ and $k(e_d) = e_d$. Further, by property~\ref{item:alpha1 formula}, we have
$$
\alpha_j(\Phi(g)) > 0 \quad \text{for $j=1,d-1$}, 
$$
so by Lemma~\ref{lem:image of Ptheta},
\begin{align*}
U_{1,d-1}(\Phi(g)) & = \big( \Phi(m)k^{-1} \big)\big( \ip{e_1}, \ip{e_1,\dots, e_{d-1}} \big)=\Phi(m)\big( \ip{e_1}, \ip{e_1,\dots, e_{d-1}} \big) \\
& = \Phi(m) \xi(\Psf_\theta) = \xi(m \Psf_\theta) = \xi( U_\theta(g)). \qedhere
\end{align*}
\end{proof}

\end{document}